\documentclass[12pt]{amsart}
\usepackage[margin=1in]{geometry}
\usepackage{amsmath,amssymb,setspace}
\usepackage{indentfirst}
\usepackage{comment}
\usepackage{booktabs}
\usepackage{wrapfig}
\usepackage{longtable}
\usepackage{lipsum}
\usepackage{bbm}

\newcommand{\ints}{\mathbb{Z}}
\newcommand{\real}{\mathbb{R}}
\newcommand{\rats}{\mathbb{Q}}
\newcommand{\cplx}{\mathbb{C}}

\newcommand{\ccal}{\mathcal{C}}

\newcommand{\supp}{\text{Supp}}

\newcommand{\sgn}{\text{sgn}}

\newcommand{\ch}{\text{ch}}
\newcommand{\sch}{\text{sch}}
\newcommand{\sign}{\text{sign}}

\numberwithin{equation}{section}

\newcommand{\Opw}{\text{Op}_h^\mathrm{w}}

\renewcommand{\phi}{\varphi}

\renewcommand{\bar}{\overline}

\renewcommand{\hom}{\text{Hom}}
\newcommand{\aut}{\text{Aut}}

\newcommand{\End}{\text{End}}

\newcommand{\ext}{\text{Ext}}

\newcommand{\tr}{\text{Tr}}

\newcommand{\id}{\text{Id}}

\newcommand{\ad}{\text{ad}}

\newcommand{\stab}{\text{Stab}}

\newcommand{\res}{\text{Res}}

\newcommand{\irr}{\text{Irr}}

\newcommand{\triv}{\text{triv}}

\newcommand{\slfr}{\mathfrak{sl}}

\newcommand{\oscr}{\mathcal{O}}

\newcommand{\spn}{\text{span}}

\newcommand{\gfr}{\mathfrak{g}}
\newcommand{\hfr}{\mathfrak{h}}
\newcommand{\lfr}{\mathfrak{l}}

\newcommand{\nfr}{\mathfrak{n}}
\newcommand{\pfr}{\mathfrak{p}}

\newcommand{\vol}{\text{vol}}

\newcommand{\spec}{\text{Spec}}
\newcommand{\rad}{\text{rad}}

\DeclareFontFamily{OT1}{rsfs}{}
\DeclareFontShape{OT1}{rsfs}{n}{it}{<-> rsfs10}{}
\DeclareMathAlphabet{\mathscr}{OT1}{rsfs}{n}{it}

\DeclareMathOperator*{\rank}{rank}

\mathchardef\mhyphen="2D 

\newtheorem{theorem}{Theorem}
\newtheorem{conjecture}[theorem]{Conjecture}
\newtheorem{corollary}[theorem]{Corollary}
\newtheorem{lemma}[theorem]{Lemma}
\newtheorem{proposition}[theorem]{Proposition}
\newtheorem{definition}[theorem]{Definition}
\newtheorem{remark}[theorem]{Remark}

\numberwithin{theorem}{subsection}
\newtheorem*{theorem*}{Theorem}

\begin{document}

\title{The Dunkl Weight Function for Rational Cherednik Algebras} \author{Seth Shelley-Abrahamson}\date{\today}

\begin{abstract} In this paper we prove the existence of the Dunkl weight function $K_{c, \lambda}$ for any irreducible representation $\lambda$ of any finite Coxeter group $W$, generalizing previous results of Dunkl. In particular, $K_{c, \lambda}$ is a family of tempered distributions on the real reflection representation of $W$ taking values in $\End_\cplx(\lambda)$, with holomorphic dependence on the complex multi-parameter $c$.  When the parameter $c$ is real, the distribution $K_{c, \lambda}$ provides an integral formula for Cherednik's Gaussian inner product $\gamma_{c, \lambda}$ on the Verma module $\Delta_c(\lambda)$ for the rational Cherednik algebra $H_c(W, \hfr)$.   In this case, the restriction of $K_{c, \lambda}$ to the hyperplane arrangement complement $\hfr_{\real, reg}$ is given by integration against an analytic function whose values can be interpreted as braid group invariant Hermitian forms on $KZ(\Delta_c(\lambda))$, where $KZ$ denotes the Knizhnik-Zamolodchikov functor introduced by Ginzburg-Guay-Opdam-Rouquier.  This provides a concrete connection between invariant Hermitian forms on representations of rational Cherednik algebras and invariant Hermitian forms on representations of Iwahori-Hecke algebras, and we exploit this connection to show that the $KZ$ functor preserves signatures, and in particular unitarizability, in an appropriate sense.\end{abstract}

\maketitle

\tableofcontents

\section{Introduction}  Let $W$ be a finite complex reflection group with reflection representation $\hfr$, and let $c : S \rightarrow \cplx$ be a $W$-invariant complex-valued function on the set of complex reflections $S \subset W$.  Associated to this data one has the rational Cherednik algebra $H_c(W, \hfr)$, introduced by Etingof and Ginzburg \cite{EG}.  The algebras $H_c(W, \hfr)$, parameterized by such functions $c : S \rightarrow \cplx$, form a family of infinite-dimensional noncommutative associative algebras providing a flat deformation of the algebra $H_0(W, \hfr) = \cplx W \ltimes D(\hfr)$, the semidirect product of $W$ with the algebra of polynomial differential operators on $\hfr$.

Since their introduction these algebras and their representation theory have been intensely studied.  Of particular interest is the category $\oscr_c(W, \hfr)$ of representations of $H_c(W, \hfr)$, introduced by Ginzburg, Guay, Opdam, and Rouquier \cite{GGOR}, deforming the category of finite-dimensional representations of $W$ and analogous to the classical Bernstein-Gelfand-Gelfand category $\oscr$ attached to a finite-dimensional complex semisimple Lie algebra $\gfr$.  In addition to being of interest from the purely representation-theoretic point of view taken in this paper, the category $\oscr_c(W, \hfr)$ has proved to have connections with various other topics, including Hilbert schemes \cite{GS1, GS2}, torus knot invariants \cite{GORS, EGL}, quantum integrable systems \cite{EG, E1, F}, and categorification \cite{S, SV}.

The category $\oscr_c(W, \hfr)$ has many structures and properties in common with classical categories $\oscr$.  It is a highest weight category with standard objects labeled by the set $\irr(W)$ of irreducible representations of the group $W$; to each $\lambda \in \irr(W)$ there is an associated standard module $\Delta_c(\lambda)$ with lowest weight $\lambda$, analogous to the Verma modules appearing in the representation theory of semisimple Lie algebras.  Each standard module $\Delta_c(\lambda)$ has a unique simple quotient $L_c(\lambda)$, and the correspondence $\lambda \mapsto L_c(\lambda)$ provides a bijection between $\irr(W)$ and the isomorphism classes of irreducible representations in $\oscr_c(W, \hfr)$.  All of the finite-dimensional representations of $H_c(W, \hfr)$ belong to $\oscr_c(W, \hfr)$.  Furthermore, each representation $M$ in $\oscr_c(W, \hfr)$ is naturally graded and therefore has an associated character, defined to be the generating function recording the characters of the finite-dimensional representations of $W$ appearing in each graded component.  However, unlike for semisimple Lie algebras, explicit character formulas for irreducible representations and a classification of the finite-dimensional representations remain unknown in many cases.

In addition to characters, another interesting class of invariants attached to irreducible representations of rational Cherednik algebras are the \emph{signature characters}.  When the parameter $c$ is \emph{real}, i.e. $c(s) = \bar{c(s^{-1})}$ for all reflections $s \in S$, each standard module $\Delta_c(\lambda)$ admits a $W$-invariant, graded, \emph{contravariant} Hermitian form $\beta_{c, \lambda}$, unique up to scaling by a positive real number.  The kernel of the form $\beta_{c, \lambda}$ coincides with the unique maximal proper submodule of $\Delta_c(\lambda)$, and in particular $\beta_{c, \lambda}$ descends to the irreducible quotient $L_c(\lambda)$.  Similarly to the definition of the character of $L_c(\lambda)$, the \emph{signature character} of $L_c(\lambda)$ is the generating function recording the signature of the form $\beta_{c, \lambda}$ in each finite-dimensional graded component of $L_c(\lambda)$.  Determining explicit formulas for these signature characters and describing the set of parameters $c$ such that $\beta_{c, \lambda}$ is unitary has proved very difficult, with complete results available in only a limited collection of cases, e.g. \cite{ESG} and \cite{Ve} for the type $A$ case and \cite{G} for the classical and cyclotomic types.  The study of the forms $\beta_{c, \lambda}$ can be viewed as an analogue for rational Cherednik algebras of the study of unitarizability (or, more generally, of signatures of invariant Hermitian forms) of representations of real reductive groups considered in detail by Adams, van Leeuwen, Trapa, and Vogan \cite{AvLTV}.

When $W$ is a finite real reflection group, the rational Cherednik algebra contains a natural $\slfr_2$-triple $\mathbf{e}, \mathbf{f}, \mathbf{h} \in H_c(W, \hfr)$.  The element $\mathbf{f}$ acts by a degree -2 operator, and hence nilpotently, on any representation $M \in \oscr_c(W, \hfr)$.  In particular, its exponential $\exp(\mathbf{f})$ is well-defined and preserves the natural filtration on any $M \in \oscr_c(W, \hfr)$.  In particular, the determination of the signatures of the \emph{Gaussian inner product} $\gamma_{c, \lambda}(v, v') := \beta_{c, \lambda}(\exp(\mathbf{f})v, \exp(\mathbf{f})v')$, considered in \cite{C} and \cite[Definition 4.5]{ESG}, in the filtered pieces of $\Delta_c(\lambda)$ is equivalent to the determination of the signature character of $L_c(\lambda)$.  As it happens, the Gaussian inner product $\gamma_{c, \lambda}$ appears easier to study than the contravariant form $\beta_{c, \lambda}$ itself.

The main purpose of this paper, generalizing previous results of Dunkl \cite{Dunkl-dihedral, Dunkl-B2, Dunkl-B2-II}, is to introduce a fundamental object, the \emph{Dunkl weight function} $K_{c, \lambda}$, for the study of the Gaussian inner product $\gamma_{c, \lambda}$.  In fact, it is natural to consider a slightly more general context in which the parameter $c$ is not required to be real; for general $c$, $\gamma_{c, \lambda}$ is a sesquilinear pairing $$\gamma_{c, \lambda} : \Delta_c(\lambda) \times \Delta_{c^\dagger}(\lambda) \rightarrow \cplx,$$  where $c^\dagger$ is the parameter given by $c^\dagger(s) = \bar{c(s^{-1})}$.  By the PBW theorem for rational Cherednik algebras \cite[Theorem 1.3]{EG}, the standard module $\Delta_c(\lambda)$ can be identified with $\cplx[\hfr] \otimes \lambda$ as a $\cplx[\hfr] \rtimes \cplx W$-module, independently of $c$, where $\cplx[\hfr]$ denotes the algebra of polynomial functions on the reflection representation $\hfr$.  With respect to this identification, and having chosen a $W$-invariant positive-definite Hermitian form $(v_1, v_2) \mapsto v_2^\dagger v_1$ on $\lambda$, we have the following result, where $\hfr_\real$ denotes the real reflection representation of $W$:

\begin{theorem*} For any finite Coxeter group $W$ and irreducible representation $\lambda \in \irr(W)$, there is a unique family $K_{c, \lambda}$, holomorphic in $c \in \pfr$, of $\End_\cplx(\lambda)$-valued tempered distributions on $\hfr_\real$ such that the following integral representation of the Gaussian pairing $\gamma_{c, \lambda}$ holds for all $c \in \pfr$: $$\gamma_{c, \lambda}(P, Q) = \int_{\hfr_\real} Q(x)^\dagger K_{c, \lambda}(x)P(x)e^{-|x|^2/2}dx \ \ \ \ \text{ for all } P, Q \in \cplx[\hfr] \otimes \lambda.$$\end{theorem*}

This theorem, along with additional properties of $K_{c, \lambda}$, appears in Theorem \ref{main-existence-theorem} and is proved in Sections \ref{small-c-section}-\ref{extension-to-hyperplanes-section}.  The family of tempered distributions $K_{c, \lambda}$ appearing in the theorem above will be called the \emph{Dunkl weight function}.  The case in which $\lambda$ is the trivial representation, including the extension of the weight function to arbitrary $c$ as a tempered distribution, has been studied previously in detail by Etingof \cite{E2}, leading to the complete determination of the parameters $c$ for which $L_c(\triv)$ is finite-dimensional for Coxeter groups $W$.  For dihedral groups $W$, Dunkl has provided explicit formulas for the matrix entries of $K_{c, \lambda}$ for $c$ small and real.

Additionally, for real parameters $c = c^\dagger$, the Dunkl weight function $K_{c, \lambda}$ provides a bridge between the study of invariant Hermitian forms on representations of rational Cherednik algebras and of associated finite Hecke algebras via the Knizhnik-Zamolodchikov (KZ) functor.  The KZ functor, introduced by Ginzburg, Guay, Opdam, and Rouquier \cite{GGOR}, is an exact functor $KZ : \oscr_c(W, \hfr) \rightarrow H_q(W)\mhyphen\text{mod}_{f.d.}$ from the category $\oscr_c(W, \hfr)$ to the category $H_q(W)\mhyphen\text{mod}_{f.d.}$ of finite-dimensional representations of the Hecke algebra $H_q(W)$ attached to the Coxeter group $W$ at a particular parameter $q$.  The construction of the KZ functor depends on a choice of point $x_0$ in the complement $\hfr_{reg}$ of the reflection hyperplanes in $\hfr$, with any two points giving rise to isomorphic functors.  Let $KZ_{x_0}$ denote the functor associated to the point $x_0 \in \hfr_{reg}$.  As a vector space, $KZ_{x_0}(\Delta_c(\lambda))$ is naturally identified with the vector space $\lambda$ itself.  In Theorem \ref{main-existence-theorem} we will see that the restriction of the distribution $K_{c, \lambda}$ to $\hfr_{\real, reg} := \hfr_\real \cap \hfr_{reg}$ is given by integration against an analytic function with values in $\End_\cplx(\lambda)$, and, with respect to the identification $KZ_{x_0}(\Delta_c(\lambda)) \cong_\cplx \lambda$, the value $K_{c, \lambda}(x_0)$ at any $x_0 \in \hfr_{\real, reg}$ determines a braid group invariant Hermitian form on $KZ_{x_0}(\Delta_c(\lambda))$.  In Section \ref{signature-comparision-section}, we exploit this relationship between invariant Hermitian forms on $\Delta_c(\lambda)$ and $KZ_{x_0}(\Delta_c(\lambda))$ to show that the KZ functor preserves signatures, and hence unitarizability, in an appropriate sense made precise in Theorem \ref{KZ-preserves-signatures-theorem}.  

This paper is organized as follows.  In Section \ref{background-section}, we recall necessary background on rational Cherednik algebras and the categories $\oscr_c(W, \hfr)$.  In Section \ref{sig-char-section}, we recall the definition of signature characters, introduce Janzten filtrations on standard modules $\Delta_c(\lambda)$, and use the Jantzen filtrations to prove that the shifted signature characters are rational functions, allowing for the definition of the \emph{asymptotic signature} $a_{c,\lambda} \in \rats \cap [-1, 1]$ recording the limiting behavior of the signatures of $\beta_{c, \lambda}$ in graded components of high degree.  In Section \ref{weight-function-section}, we state and prove the main theorem on the existence of the Dunkl weight function.  Finally, in Section \ref{signature-comparision-section}, we use the Dunkl weight function and its properties established in Section \ref{weight-function-section} to prove Theorem \ref{KZ-preserves-signatures-theorem} on the compatibility of the KZ functor with signatures.  In Section \ref{conj-section}, we discuss related conjectures and further directions.

\subsection{Acknowledgements}  This paper arose from a project which was developed by Pavel Etingof and benefited from his conversations with Larry Guth, Richard Melrose, Leonid Polterovich, and Vivek Shende.  I would like to express my deep gratitude to Pavel Etingof for suggesting this project, for many useful conversations, and for his countless ideas and insights that permeate this paper and led to the conjectures appearing in Section \ref{conj-section}.  I would also like to thank Semyon Dyatlov for providing essentially all of the content of Sections \ref{semiclassical-reminder-section} and \ref{key-analysis-lemma-section}, notably including the crucial Lemma \ref{l:matrix-case-lemma} and its proof, and for explaining to me the techniques from semiclassical analysis used in those sections.

\section{Background and Definitions}\label{background-section}  In this section we will recall the definition of rational Cherednik algebras $H_c(W, \hfr)$ and some related objects and constructions that will be important later, including the category of representations $\oscr_c(W, \hfr)$, standard modules $\Delta_c(\lambda)$, and the KZ functor.  The material recalled in this section can be found, for example, in \cite{EM} and \cite{GGOR}.

\subsection{Rational Cherednik Algebras}  A \emph{real reflection} is an invertible linear transformation $s \in GL(\hfr_\real)$ of a finite-dimensional real vector space $\hfr_\real$ such that $s^2 = \id$ and $\dim \ker(s - 1) = \dim \hfr_\real - 1$.  A \emph{finite real reflection group} is a finite group $W$ along with a real faithful representation $\hfr_\real$ such that the set of elements $S \subset W$ acting in $\hfr_\real$ as real reflections generates $W$.  We will refer to the representation $\hfr_\real$ as the real reflection representation of $W$.  A finite group $W$ may be a finite real reflection group in more than one way, and we will always consider the pair $(W, \hfr_\real)$ of $W$ along with a given real reflection representation $\hfr_\real$.  The finite real reflection groups coincide with the finite Coxeter groups.  The standard classification of finite real reflection groups and their basic properties can be found, for example, in \cite{H}.

Similarly, a \emph{complex reflection} is an invertible linear transformation $s \in GL(\hfr)$ of a finite-dimensional complex vector space $\hfr$ that has finite order and such that $\dim \ker(s - 1) = \dim \hfr - 1$.  A \emph{finite complex reflection group} is a finite group $W$ along with a complex faithful representation $\hfr$ such that the set of elements $S \subset W$ acting in $\hfr$ as complex reflections generates $W$.  For example, if $W$ is a finite real reflection group with real reflection representation $\hfr_\real$, then $W$ is also a finite complex reflection group with complex reflection representation $\hfr := \cplx \otimes_\real \hfr_\real$.  The standard classification of complex reflection groups can be found in \cite{ST}.

Fix a finite complex reflection group $W$ with complex reflection representation $\hfr$, and let $S \subset \cplx$ denote the set of complex reflections in $W$ with respect to the representation $\hfr$.  Let $(\cdot, \cdot)$ denote the natural pairing $\hfr^* \times \hfr \rightarrow \cplx$.  Mimicking the roots and coroots appearing in Lie theory, for each complex reflection $s \in S$ choose eigenvectors $\alpha_s \in \hfr^*$ and $\alpha_s^\vee \in \hfr$ for $s$ with eigenvalues $\lambda_s \neq 1$ and $\lambda_s^{-1} \neq 1$, respectively, partially normalized so that $(\alpha_s, \alpha_s^\vee) = 2$.  Let $c : S \rightarrow \cplx$ be a $W$-invariant function with respect to the action of $W$ on $S$ by conjugation.  We will refer to such $c$ as a \emph{parameter}, and $\pfr$ will denote the $\cplx$-vector space of such parameters.  Given a parameter $c \in \pfr$ one may define a \emph{rational Cherednik algebra}:

\begin{definition}
The \emph{rational Cherednik algebra} $H_c(W, \hfr)$ is the associative unital algebra with presentation $$H_c(W, \hfr) := \frac{\cplx W \ltimes T(\hfr \oplus \hfr^*)}{\langle[x, x'], [y, y'], [y, x] - (y, x) + \sum_{s \in S} c_s(\alpha_s, y)(x, \alpha_s^\vee)s : x, x' \in \hfr^*, y, y' \in \hfr\rangle}.$$
\end{definition}

\noindent In the definition above, $T(\hfr \oplus \hfr^*)$ denotes the tensor algebra of $\hfr \oplus \hfr^*$, and $\cplx W \ltimes T(\hfr \oplus \hfr^*)$ denotes the semidirect product algebra of $T(\hfr \oplus \hfr^*)$ with $\cplx W$ with respect to the natural action of $W$ on $T(\hfr \oplus \hfr^*)$: as a vector space we have $\cplx W \ltimes T(\hfr \oplus \hfr^*) = \cplx W \otimes T(\hfr \oplus \hfr^*)$, and the multiplication is given by $(w_1 \otimes t_1)(w_2 \otimes t_2) = w_1w_2 \otimes w_2^{-1}(t_1)t_2$.  For example, when $c = 0$ the algebra $H_0(W, \hfr)$ is naturally identified with the algebra $\cplx W \ltimes D(\hfr)$, where $D(\hfr)$ denotes the algebra of polynomial differential operators on $\hfr$.

The rational Cherednik algebra $H_c(W, \hfr)$ may alternatively be defined via its faithful \emph{polynomial} representation in the space $\cplx[\hfr]$ in which an element $w \in W$ acts via the representation of $W$ on $\hfr$, an element $x \in \hfr^*$ acts by multiplication, and an element $y \in \hfr$ acts by the \emph{Dunkl operator} $$D_y := \partial_y - \sum_{s \in S} \frac{2c_s\alpha_s(y)}{(1 - \lambda_s)\alpha_s}(1 - s),$$ where $\partial_y$ denotes the derivative with respect to $y$.

\subsection{The PBW Theorem, Category $\oscr_c(W, \hfr)$, and Standard Modules}  Let $\cplx[\hfr]$ denote the algebra of polynomial functions on $\hfr$, and similarly let $\cplx[\hfr^*]$ denote the algebra of polynomial functions on $\hfr^*$.  From \cite[Theorem 1.3]{EG}, we have the following PBW theorem describing $H_c(W, \hfr)$ as a vector space independent of $c$:

\begin{theorem}[PBW Theorem for Rational Cherednik Algebras] For any parameter $c$, the natural map $$\cplx[\hfr] \otimes \cplx W \otimes \cplx[\hfr^*] \rightarrow H_c(W, \hfr)$$ given by multiplication is an isomorphism of vector spaces.\end{theorem}

The PBW theorem above should be viewed as a direct analogue for rational Cherednik algebras of the triangular decomposition $U(\gfr) \cong U(\nfr^-) \otimes U(\hfr) \otimes U(\nfr^+)$ of the universal enveloping algebra $U(\gfr)$ of a complex semisimple Lie algebra $\gfr$ with Cartan subalgebra $\hfr$ and associated positive and negative nilpotent subalgebras $\nfr^+$ and $\nfr^-$.  The category $\oscr_c(W, \hfr)$ of representations of $H_c(W, \hfr)$ introduced in \cite{GGOR} is then defined in parallel with the classical Bernstein-Gelfand-Gelfand category $\oscr$:

\begin{definition}[Category $\oscr$ for Rational Cherednik Algebras, \cite{GGOR}]  The category $\oscr_c(W, \hfr)$ is the full subcategory of the category of representations of $H_c(W, \hfr)$ consisting of those representations $M$ that are finitely generated over $\cplx[\hfr]$ and on which $\hfr \subset \cplx[\hfr^*]$ acts locally nilpotently.\end{definition}

An important class of representations in $\oscr_c(W, \hfr)$ are the \emph{standard modules}, which are direct analogues of the classical Verma modules of semisimple Lie algebras:

\begin{definition}[Standard Modules and Lowest Weight Representations]  For any irreducible representation $\lambda \in \irr(W)$, the \emph{standard module $\Delta_c(\lambda) \in \oscr_c(W, \hfr)$ with lowest weight} $\lambda$ is $$\Delta_c(\lambda) := H_c(W,\hfr) \otimes_{\cplx W \ltimes \cplx[\hfr^*]} \lambda,$$ where $\hfr \subset \cplx[\hfr^*]$ acts on $\lambda$ by $0$.  A module $M \in \oscr_c(W, \hfr)$ will be called \emph{lowest weight with lowest weight $\lambda$} if $M$ is isomorphic to a nonzero quotient of $\Delta_c(\lambda)$.\end{definition}

Note that by the PBW theorem any standard module $\Delta_c(\lambda)$ is naturally isomorphic to $\cplx[\hfr] \otimes \lambda$ as a module over $\cplx W \ltimes \cplx[\hfr] \subset H_c(W, \hfr).$  We will use this identification frequently.

Modules in the category $\oscr_c(W, \hfr)$ are naturally graded by their decomposition into generalized eigenvectors for the \emph{grading element} $\mathbf{h} \in H_c(W, \hfr)$, a deformation of the Euler vector field:

\begin{definition}[Grading Element] The \emph{grading element} $\mathbf{h} \in H_c(W, \hfr)$ is the element $$\mathbf{h} := \sum_{i = 1}^{\dim \hfr} x_iy_i + \frac{\dim \hfr}{2} - \sum_{s \in S} \frac{2c_s}{1 - \lambda_s}s,$$ where $x_1, ..., x_{\dim \hfr}$ is any basis of $\hfr^*$ and $y_1, ..., y_{\dim \hfr}$ is the associated dual basis of $\hfr$.\end{definition}

Clearly, the element $\mathbf{h}$ does not depend on the choice of basis $x_1, ..., x_{\dim \hfr}$.  When $W$ is a finite real reflection group, as will be the case relevant for most of this paper, we have $\lambda_s = -1$ for all $s \in S$, so $\mathbf{h}$ takes the form $\mathbf{h} = \sum_{i = 1}^{\dim \hfr} x_iy_i + \frac{\dim \hfr}{2} - \sum_{s \in S} c_ss.$  From \cite[Proposition 3.18, Theorem 3.28]{EM} we have:

\begin{proposition} The element $\mathbf{h} \in H_c(W, \hfr)$ satisfies the commutation relations $$[\mathbf{h}, x] = x, x \in \hfr^* \ \ \ \ [\mathbf{h}, y] = -y, y \in \hfr.$$  Furthermore, $\mathbf{h}$ acts locally finitely on any $M \in \oscr_c(W, \hfr)$.\end{proposition}

It follows that any $M \in \oscr_c(W, \hfr)$ has a direct sum decomposition $M = \bigoplus_{z \in \cplx} M_z$ into generalized eigenspaces for the action of $\mathbf{h}$ and that with respect to this decomposition $M$ is a graded module with elements $x \in \hfr^*$ acting by degree 1 operators, elements $y \in \hfr$ acting by degree -1 operators, and elements $w \in W$ acting by degree 0 operators.

This grading on $\Delta_c(\lambda)$ coincides with the usual grading on $\cplx[\hfr] \otimes \lambda$ shifted by $h_c(\lambda)$, where $h_c(\lambda) = (\dim \hfr)/2 - \chi_\lambda(\sum_{s \in S}\frac{2c_s}{1 - \lambda_s})$ is the scalar by which the element $(\dim \hfr)/2 - \sum_{s \in S} \frac{2c_s}{1 - \lambda_s}s \in Z(\cplx W)$ acts in the irreducible representation $\lambda$ of $W$ and $\chi_\lambda$ denotes the character of the representation $\lambda$.  In particular, the $\mathbf{h}$-weights appearing nontrivially in the standard module are those in the set $\{h_c(\lambda) + n : n \in \ints^{\geq 0}\}$.  Any proper submodule of $\Delta_c(\lambda)$ is graded and has all nontrivial $\mathbf{h}$-weights appearing in the set $\{h_c(\lambda) + n : n \in \ints^{> 0}\}$, and it follows that there is a maximal proper submodule $N_c(\lambda)$ of $\Delta_c(\lambda)$.  We denote the irreducible quotient $\Delta_c(\lambda)/N_c(\lambda)$ by $L_c(\lambda)$.  By \cite[Proposition 3.30]{EM} we have

\begin{proposition} Any irreducible representation $L \in \oscr_c(W, \hfr)$ is isomorphic to some $L_c(\lambda)$ for a unique $\lambda \in \irr(W)$.\end{proposition}

If $N_c(\lambda) \neq 0$, then there is a nonzero homomorphism $\Delta_c(\mu) \rightarrow N_c(\lambda)$ for some $\mu \in \irr(W)$ such that $h_c(\mu) = h_c(\lambda) + k$ for some integer $k > 0$.  The quantity $h_c(\mu) - h_c(\lambda)$ is visibly a linear function of $c$, and $\{c \in \pfr : h_c(\mu) = h_c(\lambda) + k\}$ is an affine hyperplane in $\pfr$.  There are countably many such hyperplanes for various $\lambda, \mu \in \irr(W)$ and $k \in \ints^{> 0}$, and when $c$ lies on none of these hyperplanes it follows that $\Delta_c(\lambda) = L_c(\lambda)$ for all $\lambda \in \irr(W)$.  Furthermore, for such $c$ we have by a similar argument that $\ext^1_{\oscr_c(W, \hfr)}(\Delta_c(\lambda), \Delta_c(\mu)) = 0$ for all $\lambda, \mu \in \irr(W)$.  In summary, we have \cite[Proposition 3.35]{EM}:

\begin{proposition} For Weil generic $c \in \pfr$, the category $\oscr_c(W, \hfr)$ is semisimple with simple objects $L_c(\lambda)$ for $\lambda \in \irr(W)$.\end{proposition}

Finally, we will need the notion of the \emph{support} of a module $M \in \oscr_c(W, \hfr)$:

\begin{definition} Any $M \in \oscr_c(W, \hfr)$ is by definition a finitely generated module over $\cplx[\hfr]$ with additional structure.  Let $\supp(M) \subset \hfr$ denote the \emph{support} of $M$.  When $\supp(M) = \hfr$ we will say that $M$ has \emph{full support}.\end{definition}

For example, the standard module $\Delta_c(\lambda)$ is a finitely generated free module over $\cplx[\hfr]$, and in particular $\supp(\Delta_c(\lambda)) = \hfr$.  

\subsection{Hecke Algebras and the KZ Functor}  In this section we will recall the Knizhnik-Zamolodchikov functor (KZ functor) \cite{GGOR}.  For simplicity, and because it is the only case used in this paper, we will restrict attention to the case that $W$ is a finite real reflection group.

Fix a point $x_0 \in \hfr_{reg}$.  The (generalized) \emph{braid group} $B_W$ associated to $W$ is the fundamental group $B_W := \pi_1(\hfr_{reg}/W, x_0)$.  As $\hfr_{reg}/W$ is connected, $B_W$ does not depend on the choice of $x_0$ up to isomorphism.  For each reflection $s \in S$, let $T_s \in B_W$ be a representative of the conjugacy class in $B_W$ determined by a small positively-oriented loop in $\hfr_{reg}/W$ about the hyperplane $\ker(\alpha_s)$ as in \cite[Theorem 2.17]{BMR}.

\begin{definition} Let $q : S \rightarrow \cplx^\times$ be a $W$-invariant function.  The associated \emph{Hecke algebra} $H_q(W)$ is the quotient of the group algebra $\cplx B_W$ by the quadratic relations $(T_s - 1)(T_s + q) = 0$ for all $s \in S$: $$H_q(W) := \frac{\cplx B_W}{\langle (T_s - 1)(T_s + q) : s \in S\rangle}.$$\end{definition}

The \emph{KZ functor}, defined for each $c \in \pfr$, is an exact functor $KZ : \oscr_c \rightarrow H_q(W)\mhyphen\text{mod}_{f.d.}$, where $H_q(W)\mhyphen\text{mod}_{f.d.}$ denotes the category of finite-dimensional representations of $H_q(W)$, defined as follows.  As the element $\delta^{2}$ is $W$-invariant and the operator $\ad(\delta^{2}) := [\delta^{2}, \cdot]$ acts locally nilpotently on $H_c(W, \hfr)$, the noncommutative localization $H_c(W, \hfr)[\delta^{-1}]$ coincides with the usual localization of $H_c(W, \hfr)$ as a module over $\cplx[\hfr]$.  Furthermore, it follows from the formula for the Dunkl operators appearing in the polynomial representation of $H_c(W, \hfr)$ that the localized algebra $H_c(W, \hfr)[\delta^{-1}]$ is naturally identified with the algebra $\cplx W \ltimes D(\hfr_{reg})$, where $D(\hfr_{reg})$ denotes the algebra of algebraic differential operators on the regular locus $\hfr_{reg} := \hfr \backslash \cup_{s \in S} \ker(\alpha_s)$ where $\delta$ is non-vanishing.  In particular, localization $M \mapsto M[\delta^{-1}]$ defines a functor $\oscr_c(W, \hfr) \rightarrow (\cplx W \ltimes D(\hfr_{reg}))\mhyphen\text{mod}$.  The $W$-equivariant $D$-modules on $\hfr_{reg}$ appearing in the image of this functor are $\oscr$-coherent with regular singularities \cite[Proposition 5.7]{GGOR}.  Taking the space of local flat sections at $x_0 \in \hfr_{reg}$ of the vector bundle with flat connection associated to $M[\hfr_{reg}]$, one obtains a finite-dimensional representation of $B_W = \pi_1(\hfr_{reg}/W, x_0)$.  it is shown in \cite{GGOR} that the quadratic relations $(T_s - 1)(T_s + q) = 0$ hold for the Hecke algebra parameter $q : S \rightarrow \cplx^\times$, $q(s) = e^{-2\pi i c_s}$ (the parameter $c$ appearing in this paper follows the conventions of \cite{EG} and differs by a sign from the parameter $c$ appearing in \cite{GGOR}).  The resulting representation of $H_q(W)$ is functorial in $M$, and the resulting functor $\oscr_c(W, \hfr) \rightarrow H_q(W)\mhyphen\text{mod}_{f.d.}$ is the KZ functor.  To emphasize the dependence on the base point $x_0$, we denote this functor by $KZ_{x_0}$.

Clearly, the functor $KZ_{x_0}$ does not depend on $x_0 \in \hfr_{reg}$ up to isomorphism.  At the level of vector spaces, $KZ_{x_0}(M)$ is the fiber at $x_0$ of the $\cplx[\hfr]$-module $M \in \oscr_c(W, \hfr)$.  As a consequence, $KZ_{x_0}(M) \neq 0$ if and only if $M$ has full support.

We will be particularly concerned with the image of the standard modules $\Delta_c(\lambda)$ under $KZ_{x_0}$.  Recall that by the PBW theorem for rational Cherednik algebras we have $\Delta_c(\lambda) = \cplx[\hfr] \otimes \lambda$ as a $\cplx W \ltimes \cplx[\hfr]$-module, giving rise to an identification $\Delta_c(\lambda)[\delta^{-1}] = \cplx[\hfr_{reg}] \otimes \lambda$ of $\cplx W \ltimes \cplx[\hfr_{reg}]$-modules.  For any $y \in \hfr \subset H_c(W, \hfr)$ we have $y\lambda = 0$, by the definition of $\Delta_c(\lambda)$.  At the same time, with respect to the natural identification of $H_c(W, \hfr)[\delta^{-1}]$ with $\cplx W \ltimes D(\hfr_{reg})$ the element $y \in \hfr$ corresponds to the Dunkl operator $D_y = \partial_y - \sum_{s \in S} c_s\frac{\alpha_s(y)}{\alpha_s}(1 - s).$  In particular, it follows that the vector field $\partial_y \in \cplx W \ltimes D(\hfr_{reg})$ acts on $\lambda \subset \Delta_c(\lambda)[\delta^{-1}]$ by $$\partial_yv = \sum_{s \in S} c_s\frac{\alpha_s(y)}{\alpha_s}(1 - s)$$ for all $v \in \lambda \subset \Delta_c(\lambda)[\delta^{-1}]$.  It folows that the $W$-equivariant $D(\hfr_{reg})$-module structure on $\Delta_c(\lambda)[\delta^{-1}]$ arises from the flat \emph{KZ connection} $$\nabla_{KZ} := d + \sum_{s \in S} c_s\frac{d\alpha_s}{\alpha_s}(1 - s)$$ on the trivial vector bundle on $\hfr_{reg}$ with fiber $\lambda$.

\section{Signature Characters and the Jantzen Filtration}\label{sig-char-section}

\subsection{Jantzen Filtrations on Standard Modules} \label{jantzen-section}

Let $W$ be a finite complex reflection group with reflection representation $\hfr$.  Let $S \subset W$ be the set of complex reflections in $W$, and let $\pfr$ be the $\cplx$-vector space of $W$-invariant functions $c : S \rightarrow \cplx$.  For any $c \in \pfr$, let $c^\dagger \in \pfr$ be defined by $c^\dagger(s) = \bar{c(s^{-1})}$.  Refer to any $c \in \pfr$ satisfying $c = c^\dagger$ as \emph{real}, and let $\pfr_\real$ be the $\real$-vector space $\pfr_\real := \{c \in \pfr : c = c^\dagger\}$.

Recall that $\pfr$ is the parameter space for the family of rational Cherednik algebras attached to $(W, \hfr)$; we denote the rational Cherednik algebra attached to $(W, \hfr)$ and parameter $c \in \pfr$ by $H_c(W, \hfr)$.  For each irreducible representation $\lambda \in \irr(W)$, we have the associated standard module $$\Delta_c(\lambda) := H_c(W, \hfr) \otimes_{\cplx W \ltimes S\hfr} \lambda$$ over $H_c(W, \hfr)$.  As a $\ints^{\geq 0}$-graded $\cplx$-vector space, $\Delta_c(\lambda)$ is naturally identified with $\cplx[\hfr] \otimes \lambda$ independently of $c$.

Fix a $W$-invariant positive definite Hermitian form $(\cdot, \cdot)_\lambda$ on $\lambda$ - we will take the convention that Hermitian forms are conjugate-linear in the second factor.  Such a form is uniquely determined up to $\real^{> 0}$-scaling.  Similarly, fix a $W$-invariant positive-definite Hermitian form $(\cdot, \cdot)_\hfr$ on $\hfr$.  This determines the conjugate-linear isomorphism $T : \hfr \rightarrow \hfr^*$ given by $$T(y)(x) := (x, y)_\hfr.$$  When the parameter $c \in \pfr$ is \emph{real}, i.e. $c = c^\dagger$, the standard module $\Delta_c(\lambda)$ admits a unique $W$-invariant Hermitian form $\beta_{c, \lambda}$ \cite[Proposition 2.2]{ESG} such that the contravariance condition $$\beta_{c, \lambda}(yv, v') = \beta_{c, \lambda}(v, T(y)v)$$ holds for all $v, v' \in \Delta_c(\lambda)$ and $y \in \hfr$ and that coincides with $(\cdot, \cdot)_\lambda$ in degree 0.

As above, regarding the standard modules $\Delta_c(\lambda)$ for various $c$ as the same $\ints^{\geq 0}$-graded vector space $\Delta(\lambda)$, we may view the forms $\beta_{c, \lambda}$ on $\Delta(\lambda)$ as an algebraic family of Hermitian forms $\beta_{c, \lambda}[d]$, parameterized by $c \in \pfr_\real$ in a polynomial manner, on each finite-dimensional graded component $\Delta(\lambda)[d]$.  In particular, we are naturally led to consider Jantzen filtrations, as follows.

Let $c_0, c_1 \in \pfr_\real$ be real parameters such that there exists $\delta > 0$ such that $\beta_{c(t), \lambda}$ is nondegenerate for all $c(t) := c_0 + tc_1$ with $t \in (-\delta, \delta)\backslash \{0\}$.  The finite-dimensional $\cplx$-vector spaces $\Delta(\lambda)[d]$ for $n \geq 0$ along with the polynomial families of Hermitian forms $\beta_{c(t), \lambda}[d]$ satisfy the conditions of \cite[Definition 3.1]{Vo}.  In particular, for each $d \geq 0$ define the finite descending filtration $$\Delta(\lambda)[d] = \Delta(\lambda)[d]^{\geq 0} \supset \Delta(\lambda)[d]^{\geq 1} \supset \cdots \supset \Delta(\lambda)[d]^{\geq N} = 0$$ on $\Delta(\lambda)[d]$ as follows (to simplify the notation, we do not include the choice of $c_0, c_1$ in the notation for the filtration, but of course this filtration and its properties may be dependent on this choice).  Let $\Delta(\lambda)[d]^{\geq k}$ consist of those vectors $v \in \Delta(\lambda)[d]$ such that there is some $\epsilon > 0$ and an analytic function $f_v : (-\epsilon, \epsilon) \rightarrow \Delta(\lambda)[d]$ satisfying $f_v(0) = v$ and such that the analytic function $$t \mapsto \beta_{c(t), \lambda}[d](f_v(t), v')$$ vanishes at least to order $k$ for all $v' \in \Delta_c(\lambda)[d]$ (clearly, one may equivalently consider only polynomial functions $f_v$).  For each $k \geq 0$, define the Hermitian form $\beta_{c_0, \lambda}[d]^{\geq k}$ on $\Delta(\lambda)[d]^{\geq k}$ by $$\beta_{c_0, \lambda}[d]^{\geq k}(v, v') = \lim_{t \rightarrow 0} \frac{1}{t^k} \beta_{c(t), \lambda}[d](f_v(t), f_{v'}(t)),$$ where $v, v' \in \Delta(\lambda)[d]^{\geq k}$ and $f_v, f_{v'}$ are any analytic functions as above (this limit does not depend on the choice of such $f_v, f_{v'}$).  We then have the following theorem:

\begin{theorem} (Jantzen \cite[5.1]{J}, Vogan \cite[Theorem 3.2]{Vo}) The radical of the Hermitian form $\beta_{c_0, \lambda}[d]^{\geq k}$ is precisely $\Delta(\lambda)[d]^{\geq k + 1}$.\end{theorem}

In particular, the Hermitian form $\beta_{c_0, \lambda}[d]^{\geq k}$ descends to a nondegenerate Hermitian form on each filtration subquotient $\Delta(\lambda)[d]^{(k)} := \Delta(\lambda)[d]^{\geq k}/\Delta(\lambda)[d]^{\geq k + 1}.$  Denote this induced nondegenerate Hermitian form by $\beta_{c_0, \lambda}[d]^{(k)}$.  For all $k, d \geq 0$, let $p_d^{(k)}$ (resp. $q_d^{(k)}$) be the dimension of a maximal positive definite (resp., negative definite) subspace of $\Delta(\lambda)[d]^{(k)}$ with respect to the form $\beta_{c_0, \lambda}[d]^{(k)}$.  Note that for any fixed $d \geq 0$ we have $p_d^{(k)} = q_d^{(k)} = 0$ for all sufficiently large $k$.

We will take the convention that the \emph{signature} of a Hermitian form $\beta$ on a finite dimensional $\cplx$-vector space $V$ is the integer $p - q$, where $p$ (respectively, $q$) is the dimension of any maximal positive-definite (respectively, negative-definite) subspace of $V$ with respect to $\beta$.  For example, in the context of the previous paragraph, $p_d^{(k)} - q_d^{(k)}$ is the signature of the form $\beta_{c_0, \lambda}[d]^{(k)}$.  If the form $\beta$ is non-degenerate then the dimension of $V$ and the quantity $p - q$ determines the tuple $(p, q)$ considered to be the signature of $\beta$ in some references.

\begin{proposition} \label{sig-flip-prop} (Vogan \cite[Proposition 3.3]{Vo}) \ 

(a)  For all small positive $t$ (i.e. $t \in (0, \delta)$) and any $d \geq 0$, the signature of the nondegenerate Hermitian form $\beta_{c(t), \lambda}[d]$ on $\Delta(\lambda)[d]$ is $$\sum_{k \geq 0} p_d^{(k)} - \sum_{k \geq 0} q_d^{(k)}.$$

(b)  Similarly, for all small negative $t$ (i.e. $t \in (-\delta, 0)$) and any $d \geq 0$, the signature of $\beta_{c(t) ,\lambda}[d]$ is $$\left(\sum_{k \text{ even}} p_d^{(k)} + \sum_{k \text{ odd}}q_d^{(k)}\right) -  \left(\sum_{k \text{ odd}} p_d^{(k)} + \sum_{k \text{ even}}q_d^{(k)}\right).$$\end{proposition}

Define the descending filtration $$\Delta(\lambda) = \Delta(\lambda)^{\geq 0} \supset \Delta(\lambda)^{\geq 1} \supset \cdots$$ by $$\Delta(\lambda)^{\geq k} := \bigoplus_{d \geq 0} \Delta(\lambda)[d]^{\geq k} \subset \bigoplus_{d \geq 0} \Delta(\lambda)[d] = \Delta(\lambda).$$

\begin{lemma} \label{Jantzen-filtration-for-standards} The filtration of $\Delta_{c_0}(\lambda)$ by the subspaces $\Delta(\lambda)^{\geq k}$ is a filtration by $H_{c_0}(W, \hfr)$-submodules.  We have $\Delta(\lambda)^{\geq k} = 0$ for sufficiently large $k$.\end{lemma}

\begin{proof} Let $v \in \Delta(\lambda)[d]^{\geq k}$, and let $f_v : (-\epsilon, \epsilon) \rightarrow \Delta(\lambda)[d]$ be as in the definition of the filtration, exhibiting that $v \in \Delta(\lambda)[d]^{\geq k}$.  Then for any homogeneous $h \in H_{c_0}(W, \hfr)$ of degree $d'$, viewing $h \in H_{c(t)}(W, \hfr)$ for all $t \in \real$ via the PBW basis, the path $hf_v$ exhibits $hv$ as an element of $\Delta(\lambda)[d + d']^{\geq k}$.  In particular, $\Delta(\lambda)^{\geq k}$ is a $H_{c_0}(W, \hfr)$-submodule of $\Delta(\lambda)$.  By the finite-length property of $H_{c_0}(W, \hfr)$-modules in category $\oscr_{c_0}(W, \hfr)$, it follows that the filtration $\Delta(\lambda)^{\geq k}$ stabilizes in $k$.  For any fixed $d$ we have $\Delta(\lambda)[d]^{\geq k} = 0$ for any $k$ sufficiently large, and it follows that $\Delta(\lambda)^{\geq k} = 0$ for all sufficiently large $k$.\end{proof}

We refer to the filtration of $\Delta_{c_0}(\lambda)$ appearing in Lemma \ref{Jantzen-filtration-for-standards} as the \emph{Jantzen filtration} of $\Delta_{c_0}(\lambda)$.  Note that this Jantzen filtration depends on the choice of the additional parameter $c_1 \in \pfr_\real$ determining the direction for the deformation.

\subsection{Hermitian Duals}\label{hermitian-duals-section}  In this section we will introduce Hermitian duals in the setting of rational Cherednik algebras, analogous to the Hermitian duals considered by Vogan \cite{Vo} in the Lie-theoretic setting.  First we will briefly recall contragredient duals.  Let $c \in \pfr$ be any parameter for the rational Cherednik algebra attached to $(W, \hfr)$.  Let $\bar{c} \in \pfr$ be the parameter defined by $\bar{c}(s) = c(s^{-1})$.  As explained in \cite[Section 3.11]{EM}, there is a natural isomorphism $$\gamma : H_c(W, \hfr)^{opp} \rightarrow H_{\bar{c}}(W, \hfr^*)$$ acting trivially on $\hfr$ and $\hfr^*$ and sending $w \mapsto w^{-1}$ for all $w \in W$.  For any $M \in \oscr_c(W, \hfr)$, the restricted dual $M^\dagger := \bigoplus_{z \in \cplx} M_z^*$ is naturally a $H_c(W, \hfr)^{opp}$-module; by transfer of structure along $\gamma$, we regard $M^\dagger$ as a $H_{\bar{c}}(W, \hfr^*)$-module.  We have:

\begin{proposition} (\cite[Proposition 3.32]{EM}) The assignment $M \mapsto M^\dagger$ determines a $\cplx$-linear equivalence of categories $\ ^\dagger : \oscr_c(W, \hfr) \rightarrow \oscr_{\bar{c}}(W, \hfr^*)^{opp}.$\end{proposition}

Given a $\cplx$-algebra $A$, let $\bar{A}$ denote the $\cplx$-algebra that is equal to $A$ as a ring and as a ring is equal to the complex conjugate vector space of $A$.  In other words, the identity map $\id : A \rightarrow \bar{A}$ is an isomorphism of rings and satisfies $z\id(a) = \id(\bar{z}a)$.  Clearly $\bar{A}$ is a $\cplx$-algebra, with unit $\bar{\eta}$ satisfying $\bar{\eta}(z) = \eta(\bar{z})$ for all $z \in \cplx$, where $\eta : \cplx \rightarrow A$ is the unit map for the $\cplx$-algebra $A$.  Similarly, for any $A$-module $M$ the complex conjugate vector space $\bar{M}$ is naturally an $\bar{A}$-module, and this clearly defines an conjugate-linear equivalence of categories $A\mhyphen\text{mod} \rightarrow \bar{A}\mhyphen\text{mod}$.

The complex conjugate of a rational Cherednik algebra is again a rational Cherednik algebra.  More precisely:

\begin{lemma} Fix a nondegenerate $W$-invariant Hermitian form $(\cdot, \cdot)_\hfr$ on $\hfr$, and let $T : \hfr \rightarrow \hfr^*$ be the conjugate-linear isomorphism introduced in Section \ref{jantzen-section}.  Then the mappings $y \mapsto Ty$ for $y \in \hfr$, $x \mapsto T^{-1}x$ for $x \in \hfr^*$, and $w \mapsto w$ for $w \in W$ extend uniquely to an isomorphism of $\cplx$-algebras $$\omega : H_{\bar{c}}(W, \hfr^*) \rightarrow \bar{H_{c^\dagger}(W, \hfr)}.$$\end{lemma}

\begin{proof} Regarded as a map $\hfr \rightarrow \bar{\hfr^*}$, $T$ is an isomorphism of complex representations of $W$, and similarly for $T^{-1} : \hfr^* \rightarrow \bar{\hfr}$.  It follows that the assignments in the lemma extend uniquely to a $\cplx$-linear isomorphism $\cplx W \ltimes T(\hfr^* \oplus \hfr) \rightarrow \bar{\cplx W \ltimes T(\hfr \oplus \hfr^*)}$ which determines a map $\cplx W \ltimes T(\hfr^* \oplus \hfr) \rightarrow \bar{H_{c^\dagger}(W, \hfr)}.$  As $\hfr$ is sent to $\hfr^*$ and $\hfr^*$ to $\hfr$ under this map, the commutators $[x, x']$ and $[y, y']$ for $x, x' \in \hfr^*$ and $y, y' \in \hfr$ are sent to 0 by this map.  Let $\langle \cdot, \cdot \rangle$ denote the natural pairing of $\hfr$ with $\hfr^*$.  By definition of $T$, we have, for any $y \in \hfr$ and $x \in \hfr^*$, $$\langle Ty, T^{-1}x\rangle = (T^{-1}x, y)_\hfr = \bar{(y, T^{-1}x)_\hfr} = \bar{\langle x, y\rangle}.$$   In particular, under the map $\cplx W \ltimes T(\hfr^* \oplus \hfr) \rightarrow \bar{H_{c^\dagger}(W, \hfr)}$, for any $x \in \hfr^*$ and $y \in \hfr$ the image of the element $$[x, y] - \langle x, y \rangle + \sum_{s \in S} \bar{c}_s\langle x, \alpha_s^\vee \rangle \langle y, \alpha_s\rangle s$$ in $\bar{H_{c^\dagger}(W, \hfr)}$ is $$[T^{-1}x, Ty] - \bar{\langle x, y \rangle} + \sum_{s \in S} c^\dagger_s \bar{\langle x, \alpha_s^\vee \rangle}\bar{\langle y, \alpha_s\rangle}s$$ $$= [T^{-1}x, Ty] - \langle T^{-1}x, Ty \rangle + \sum_{s \in S} c^\dagger_s \langle T^{-1}x, T\alpha_s^\vee \rangle \langle Ty, T^{-1}\alpha_s\rangle s.$$  As $T\alpha_s^\vee \in \hfr^*$ and $T\alpha_s \in \hfr$ are eigenvectors for $s$ with nontrivial eigenvalues and $$\langle T\alpha_s^\vee, T^{-1}\alpha_s \rangle = \bar{\langle \alpha_s, \alpha_s^\vee\rangle} = \bar{2} = 2,$$ the expression above is 0 in $H_{c^\dagger}(W, \hfr).$  It follows that there is an induced map of $\cplx$-algebras $H_{\bar{c}}(W, \hfr^*) \rightarrow \bar{H_{c^\dagger}(W, \hfr)}$, and clearly this map is an isomorphism.\end{proof}

Let $\sigma : \bar{H_c(W, \hfr)^{opp}} \rightarrow H_{c^\dagger}(W, \hfr)$ be the isomorphism of $\cplx$-algebras obtained by compositing $\gamma$ and $\omega$.  Its action on generators is given by $\sigma(x) = T^{-1}x$ for $x \in \hfr^*$, $\sigma(y) = Ty$ for $y \in \hfr$, and $\sigma(w) = w^{-1}$ for $w \in W$.  As $T$ depends on the choice of the form $(\cdot, \cdot)_\hfr$, so does $\sigma$.

For any $M \in \oscr_c(W, \hfr)$, the \emph{Hermitian dual} $M^h := \bar{M^\dagger}$ is naturally an $\bar{H_c(W, \hfr)^{opp}}$-module, and by transfer of structure along $\sigma$ we regard $M^h$ as a $H_{c^\dagger}(W, \hfr)$-module.  Clearly $M^h \in \oscr_{c^\dagger}(W, \hfr)$, so we have:

\begin{lemma} The assignment $M \mapsto M^h$ defines a conjugate-linear equivalence of categories $$\ ^h: \oscr_c(W, \hfr) \rightarrow \oscr_{c^\dagger}(W, \hfr)^{opp}.$$\end{lemma}

\begin{definition} Given modules $M \in \oscr_c(W, \hfr)$ and $N \in \oscr_{c^\dagger}(W, \hfr)$, a sesquilinear pairing $$\beta : M \times N \rightarrow \cplx$$ will be called \emph{contravariant} if $$\beta(hm, n) = \beta(m, \sigma(h)n) \ \ \ \ \text{ for all } h \in H_c(W, \hfr), m \in M, n \in N.$$ \end{definition}

As $\sigma(w) = w^{-1}$ for all $w \in W$ and as $\sigma$ sends the grading element of $H_c(W, \hfr)$ to the grading element of $H_{c^\dagger}(W, \hfr)$, it follows that any contravariant pairing $\beta$ is automatically graded and $W$-invariant.

We will be particularly concerned with contravariant Hermitian forms on modules $M \in \oscr_c(W, \hfr)$ for real parameters $c \in \pfr_\real$.  We will use the following lemma later:

\begin{lemma} \label{herm-dual-iso-lemma} Suppose the parameter $c \in \pfr$ is real, i.e. $c = c^\dagger$.  Let $M \in \oscr_c(W, \hfr)$ be equipped with a nondegenerate $W$-invariant contravariant Hermitian form $\beta$.  Then the assignment $$m \mapsto \beta(\cdot, m)$$ defines an isomorphism $M \cong M^h$ of $H_c(W, \hfr)$-modules.\end{lemma}

\begin{proof} That the map in question is a map of $H_c(W, \hfr)$-modules follows from the observation that the statement that $\beta$ is $W$-invariant and contravariant means precisely that $\beta(hv, v') = \beta(v, \sigma(h)v')$ for all $v, v' \in M$ and $h \in H_c(W, \hfr) = H_{c^\dagger}(W, \hfr)$.  That the map is an isomorphism follows from the nondegeneracy of $\beta$.\end{proof}

\subsection{Characters and Signature Characters}

Let us now briefly recall the definition of characters and signature characters.

\begin{definition}  Let $M \in \oscr_c(W, \hfr)$.  Then the \emph{character} of $M$ \cite[Section 3.9]{EM}, denoted $\ch(M)$, is the formal series $$\ch(M)(w, t) := \sum_{z \in \cplx} t^z\tr_{M_z}(w), \ \ \ w \in W.$$
\end{definition}

In particular, taking $w = 1$, one obtains the graded dimension of $M$:  $$\ch(M)(1, t) = \sum_{z \in \cplx} t^z\dim M_z.$$

\begin{definition} When $M$ is a lowest weight module with lowest weight $\lambda$, we define the \emph{shifted character} $\ch_0(M)$ by $$\ch_0(M)(w, t) := t^{-h_c(\lambda)}\ch(M)(w, t).$$\end{definition}

The character of a standard module $\Delta_c(\lambda)$ is given by \cite[Proposition 3.27]{EM} $$\ch(\Delta_c(\lambda))(w, t) = \frac{\chi_\lambda(w)t^{h_c(\lambda)}}{\det_{\hfr^*} (1 - tw)}.$$  Furthermore, it is a standard result (e.g. \cite[Section 3.13]{EM}) that the set of classes $\{[\Delta_c(\lambda)] : \lambda \in \irr(W)\}$ of the standard modules form a basis of the integral Grothendieck group $K_0(\oscr_c(W, \hfr))$ and that whenever $$[M] = \sum_{\lambda \in \irr(W)} n_\lambda [\Delta_c(\lambda)]$$ in $K_0(\oscr_c(W, \hfr))$ we have $$\ch(M) = \sum_{\lambda \in \irr(W)} n_\lambda \ch(\Delta_c(\lambda)).$$  In particular, it follows from standard facts about Hilbert series that for any lowest weight module $M$ we have $\ch_0(M)(1, t) = (1 - t)^{-r}p_M(t)$, where $r$ is the dimension of support of $M$ and where $p_M(t)$ is a polynomial in $t$ with integer coefficients such that $p_M(t)(1) \neq 0$.

When $M \in \oscr_c(W, \hfr)$ is equipped with a graded Hermitian form, we may similarly define the signature character of the tuple $(M, \beta)$:

\begin{definition} Let $M \in \oscr_c(W, \hfr)$ be equipped with a graded Hermitian form $\beta$.  Then the \emph{signature character} $\sch(M, \beta)$ is the formal series $$\sch(M, \beta)(t) := \sum_{z \in \cplx} t^z\sign(\beta_z)$$ where, for each $z \in \cplx$, $\sign(\beta_z)$ denotes the signature of the restriction $\beta_z$ of the form $\beta$ to the weight space $M_z$.\end{definition}

When $M \in \oscr_c(W, \hfr)$ is a lowest weight module and the parameter $c$ is real, we will often write $\sch(M)$ rather than $\sch(M, \beta)$, where it is implicit that the Hermitian form $\beta$ is $W$-invariant, contravariant, and positive definite in the lowest weight space.  In this setting, we also define the shifted signature character:

\begin{definition} When $M \in \oscr_c(W, \hfr)$ is a lowest weight module with lowest weight $\lambda$, we define the \emph{shifted signature character} $\sch_0(M)$ by $$\sch_0(M)(t) := t^{-h_c(\lambda)}\sch(M)(t).$$\end{definition}

\subsection{Rationality of Signature Characters}

The following rationality result for the shifted signature characters $\sch_0(L_c(\lambda))$ generalizes to arbitrary finite complex reflection groups $W$ a corresponding result for signature characters in type $A$ due to Venkateswaran \cite[Corollary 1.3]{Ve}:

\begin{proposition} \label{rationality-prop} For any irreducible complex representation $\lambda \in \irr(W)$ and real parameter $c \in \pfr_\real$, the shifted signature character $\sch_0(L_c(\lambda))$ is of the form $$\sch_0(L_c(\lambda))(t) = (1 - t)^{-r}p_{L_c(\lambda)}(t)$$ for some polynomial $p_{L_c(\lambda)}(t)$ with integer coefficients, where $r = \dim \supp(L_c(\lambda)).$\end{proposition}

The proof of Proposition \ref{rationality-prop} is an inductive argument relying on the following lemmas:

\begin{lemma} \label{wall-crossing-lemma} Let $\lambda \in \irr(W)$, and let $c_0, c_1 \in \pfr_\real$ be real parameters so that the Jantzen filtration of $\Delta_{c_0}(\lambda)$ is defined.  Let $c(s) = c_0 + sc_1$.  For sufficiently small $s > 0$, we have $$\sch_0(\Delta_{c(s)}(\lambda))(t) = \sch_0(\Delta_{c(-s)}(\lambda))(t) + 2t^{-h_{c_0}(\lambda)}\sum_{k \text{ odd}} \sch(\Delta_{c_0}(\lambda)^{(k)}, \beta_{c_0, \lambda}^{(k)})(t)$$ and $$\sch_0(L_{c_0}(\lambda))(t) = \sch_0(\Delta_{c(s)}(\lambda))(t) - t^{-h_{c_0}(\lambda)}\sum_{k \geq 1} \sch(\Delta_{c_0}(\lambda)^{(k)}, \beta_{c_0, \lambda}^{(k)})(t).$$\end{lemma}

\begin{proof} This is an immediate consequence of Proposition \ref{sig-flip-prop} and the definition of signature characters. \end{proof}

The following lemma (and its proof) is a reformulation for rational Cherednik algebras of Vogan's \cite[Lemma 3.9]{Vo}.

\begin{lemma} \label{K0-lemma} Suppose $c = c^\dagger$ and $M \in \oscr_c(W, \hfr)$ admits a $W$-invariant contravariant nondegenerate Hermitian form $\beta$.  Suppose $[M] = \sum_{i = 1}^n [L_i]$ in $K_0(\oscr_c(W, \hfr))$ for some simple modules $L_1, ..., L_n$.  Then there are $\epsilon_1, ..., \epsilon_n \in \{\pm 1\}$ such that $$\sch(M, \beta) = \sum_{i = 1}^n \epsilon_i \sch(L_i).$$ \end{lemma}

\begin{proof} The proof is by induction on the length of $M$.  The case $M = 0$ (or $M = 1$) is trivial, so suppose $M > 0$ and let $L \subset M$ be a simple submodule.  If the restriction of $\beta$ to $L$ is nondegenerate, then $M = L \oplus L^\perp$ and $L^\perp$ is an $H_c(W, \hfr)$-submodule of $M$ on which $\beta$ is nondegenerate.  As all $W$-invariant contravariant Hermitian forms on $L$ are proportional, we have $\sch(L, \beta|_L) = \pm\sch(L)$, and the claim follows by induction.

If the restriction of $\beta$ to $L$ is degenerate, it is 0.  The inclusion $L \subset M$ gives rise to a surjection $M^h \rightarrow L^h$.  By Lemma \ref{herm-dual-iso-lemma}, we have an isomorphism $M \rightarrow M^h$, $m \mapsto \beta(\cdot, m)$.  The resulting composition $\phi : M \rightarrow L^h$ satisfies $$\phi(m)(l) = \beta(l, m)$$ for all $l \in L$ and $m \in M$. As $\beta|_L = 0$, we have $L^\perp = \ker \phi \supset L$, and in particular the form $\beta$ descends to a nondegenerate $W$-invariant contravariant Hermitian form on the subquotient $N := \ker \phi/L.$  Now we have that $L \cong L^h$ so $$M = N + L + L^h = N + 2L$$ in $K_0(\oscr_c(W, \hfr))$.  Furthermore it follows from \cite[Sublemma 3.18]{Vo} that $\sch(M) = \sch(N)$ and hence $\sch(M) = \sch(N) + \sch(L) - \sch(L)$, and the claim follows by induction.\end{proof}

\begin{proof}[Proof of Proposition \ref{rationality-prop}] Let $c \in \pfr_\real$.  It follows from \cite[Proposition 3.35]{EM} that there are only finitely many $s \in [0, 2]$, say $\{s_1, ..., s_N\}$ for some $N \ge 0$, such that $\oscr_{sc}(W, \hfr)$ is not semisimple.  Furthermore, at $s = 0$ we have both that $\oscr_{sc}(W, \hfr)$ is semisimple and that every simple module $L_0(\lambda) = \Delta_0(\lambda)$ is unitary.  In particular, for all $\lambda$ we have $\sch_0(L_0(\lambda)) = (1 - t)^{-l}\dim \lambda $ where $l = \dim \hfr$, so the proposition holds for $c = 0$.  Furthermore, note that the signature character $\sch_0(\Delta_{c(s)}(\lambda))$ does not depend on $s$ for those $s$ in a fixed interval $(s_i, s_{i + 1})$.  So, by induction, we may assume that the proposition holds for all $sc$ with $s \in [0, 1)$ and we need then only prove that it holds for all $sc$ with $s \in [1, 1 + \delta)$ for some $\delta > 0$.

For any $\lambda \in \irr(W)$ such that $L_{c_0}(\lambda) = \Delta_{c_0}(\lambda)$, the signature character $\sch_0(\Delta_{(1 + s)c_0}(\lambda))$ does not depend on $s$ for $|s|$ sufficiently small.  In particular, the proposition holds for those $\lambda \in \irr(W)$ minimal in any highest weight ordering $\leq_{c_0}$ for $\oscr_{c_0}(W, \hfr)$.  By induction, we may then assume that the proposition holds for those lowest weights $\mu \in \irr(W)$ strictly lower than $\lambda$ with respect to $\leq_{c_0}$.  Taking $c_1 := c_0$ in Lemma \ref{wall-crossing-lemma}, it then follows from Lemmas \ref{wall-crossing-lemma} and \ref{K0-lemma} that $\sch_0(L_{c_0}(\lambda))$ and all $\sch_0(\Delta_{(1 + s)c_0}(\lambda))$ for sufficiently small $s > 0$ are of the form $(1 - t)^{-l}p(t)$ for some polynomial $p(t)$ with integer coefficients.  As the absolute value of the coefficients of the series $\sch_0(L_{c_0}(\lambda))$ is bounded by the coefficients of $\ch_0(L_{c_0}(\lambda))$, it follows that $\sch_0(L_{c_0}(\lambda))$ is of the form $(1 - t)^{-r}p(t)$ for some polynomial $p$ with integer coefficients, as needed.\end{proof}

\subsection{Asymptotic Signatures}

In this section we will introduce the \emph{asymptotic signature} $a_{c, \lambda}$ of the irreducible representation $L_c(\lambda)$, generalizing to arbitrary finite complex reflection groups the corresponding notion in type $A$ studied by Venkateswaran \cite{Ve}, and prove that it is a rational number in the interval $[-1,1]$.  Later, in the case that $\supp L_c(\lambda) = \hfr$, we will see in Theorem \ref{KZ-preserves-signatures-theorem} that, in the Coxeter case, $a_{c, \lambda}$ is in fact a normalization of the signature of an invariant Hermitian form on $KZ(L_c(\lambda))$, generalizing to arbitrary finite Coxeter groups a corresponding result in type $A$ due to Venkateswaran \cite[Theorem 1.4]{Ve}.

\begin{lemma}\label{coef-ratio-lemma} Let $g(t) = \sum_{n = 0}^\infty g_nt^n$ be a rational function regular for $|t| < 1$ and with a pole of maximal order at $t = 1$.  Let $p(t)$ and $q(t)$ be polynomials with $q(1) \neq 0$, and let $s(t) = p(t)g(t) = \sum_{n = 0}^\infty s_nt^n$ and $d(t) = q(t)g(t) = \sum_{n = 0}^\infty d_nt^n$.  Then the limits $$\lim_{t \rightarrow 1^-} \frac{s(t)}{d(t)}, \ \ \ \lim_{N \rightarrow \infty} \frac{\sum_{n \leq N} s_n}{\sum_{n \leq N} d_n}$$ both exist and equal $p(1)/q(1)$.\end{lemma}

\begin{proof} That the first limit exists and equals $p(1)/q(1)$ is clear.  Let $r$ be the order of the pole of $g$ at $t = 1$.  We have $$\sum_{N = 0}^\infty \left(\sum_{n \leq N} s_n\right)t^N = p(t)g(t)/(1 - t) = (p(t) - p(1))g(t)/(1 - t) + p(1)g(t)/(1 - t).$$  As $g(t)/(1 - t)$ has a pole of order $r + 1$ at $t = 1$, greater than the order of any other pole of $g(t)/(1 - t)$ or of any pole of $(p(t) - p(1))g(t)/(1 - t)$, it follows from a consideration of partial fractions that $\lim_{N \rightarrow \infty} (\sum_{n \leq N} s_n)/N^r = p(1)C$, where $C \neq 0$ is a nonzero constant depending only on the rational function $g(t)$.  Similarly, we have $\lim_{N \rightarrow \infty} (\sum_{n \leq N} d_n)/N^r = q(1)C$, and the claim follows.\end{proof}

We will use Lemma \ref{coef-ratio-lemma} in the special cases that $g(t) = (1 - t)^{-l}$ when working signature characters and $g(t) = \prod_{i = 1}^l (1 - t^{d_i})^{-1}$, where $d_1, ..., d_l$ are the fundamental degrees of the reflection group $W$, when working with the \emph{isotypic signature characters} introduced in Section \ref{isotypic-sig-char-section}.

\begin{lemma} \label{asymptotic-sig-lemma} Let $c \in \pfr_\real$ be a real parameter and let $\lambda \in \irr(W)$, so that the irreducible lowest weight module $L_c(\lambda)$ admits a $W$-invariant contravariant nondegenerate Hermitian form $\beta_{c, \lambda}$, normalized to be positive definite on $\lambda$.  For any $n \geq 0$, let $\beta_{c, \lambda}^{\leq n}$ denote the restriction of $\beta$ to the space $L_c(\lambda)^{\leq n} := \bigoplus_{k \leq n} L_c(\lambda)[k].$  Then the limits $$\lim_{n \rightarrow \infty} \frac{\sign(\beta_{c, \lambda}^{\leq n})}{\dim L_c(\lambda)^{\leq n}}, \ \ \  \lim_{t \rightarrow 1^-} \frac{\sch_0(L_c(\lambda))(t)}{\ch_0(L_c(\lambda))(1, t)}$$ exist and are equal to the same rational number $a_{c, \lambda} \in [-1, 1]$.  If $r := \dim \supp(L_c(\lambda)) > 0$, then the limit $$\lim_{n \rightarrow \infty} \frac{\sign(\beta_{c, \lambda}[n])}{\dim L_c(\lambda)[n]}$$ also exists and equals $a_{c, \lambda}$.\end{lemma}

\begin{proof} In the case $r = 0$ the claim is clear, so we may assume $r > 0$.  Using Proposition \ref{rationality-prop}, write $\ch_0(L_c(\lambda))(1, t) = \sum_{n = 0}^\infty d_nt^n = (1 - t)^{-r}q(t)$ and $\sch_0(L_c(\lambda)) = \sum_{n = 0}^\infty s_nt^n = (1 - t)^{-r}p(t)$, where $p(t)$ and $q(t)$ are polynomials with integer coefficients satisfying $q(1) \neq 0$, where we use the series expansion about 0 for $(1 - t)^{-r}$.  In particular, Lemma \ref{coef-ratio-lemma} applies, and the first two limits in the lemma statement exist and equal $p(1)/q(1) \in \rats$.  As we clearly have $|\sign(\beta^{\leq n}_{c, \lambda})| \leq \dim L_c(\lambda)^{\leq n}$, we also have $p(1)/q(1) \in [-1, 1]$, giving the first claim.  The final claim follows similarly, noting that for $n >> 0$ we have that $\sign(\beta_{c, \lambda}[n])$ and $\dim L_c(\lambda)[n]$ are polynomials in $n$ of degree at most $r$ and with coefficients of $n^r$ given by $q(1)/(r - 1)!$ and $p(1)/(r - 1)!$, respectively.\end{proof}

\begin{definition}\label{def:quasi-unitary} For $c \in \pfr_\real$, the rational number $a_{c, \lambda} \in [-1, 1]$ appearing in Lemma \ref{asymptotic-sig-lemma} is the \emph{asymptotic signature} of $L_c(\lambda)$.  If $a_{c, \lambda} = \pm 1$, we say $L_c(\lambda)$ is \emph{quasi-unitary}.\end{definition}

\subsection{Isotypic Signature Characters}\label{isotypic-sig-char-section}  As contravariant forms are not only graded but also $W$-invariant, it is natural to consider the \emph{isotypic signature characters} recording the signatures of contravariant forms both within graded components and also within $W$-isotypic components:

\begin{definition} Let $c \in \pfr_\real$ be a real parameter, let $\beta$ be a contravariant Hermitian form on the module $M \in \oscr_c(W, \hfr)$, and for all $z \in \cplx$ and $\pi \in \irr(W)$ let $M_z^\pi \subset M$ denote the $\pi$-isotypic component of the homogeneous degree $z$ part of $M$ and let $\beta^\pi_z = \beta|_{M^\pi_z}$.  The \emph{$\pi$-isotypic signature character} $\sch^\pi(M, \beta)$ of $M$ with respect to $\beta$ is the formal series $$\sch^\pi(M, \beta)(t) := \sum_{z \in \cplx} \sign(\beta^\pi_z)t^z.$$  \end{definition}

Clearly, we have $\sch(M, \beta) = \sum_{\pi \in \irr(W)} \sch^\pi(M, \beta)$.  When $M$ is lowest weight, we define $\sch_0^\pi(M)(t) \in \ints[[t]]$ completely analogously to $\sch_0(M)$.  Similarly, let $\ch^\pi(M)$ denote the graded dimension of $M^\pi$, and when $M$ is lowest weight with lowest weight $\lambda$ let $\ch_0^\pi(M)(t) = t^{-h_c(\lambda)}\ch^\pi(M)(t) \in \ints^{\geq 0}[[t]]$.

The identification of any standard module $\Delta_c(\lambda)$ with $\cplx[\hfr] \otimes \lambda$ used in Section \ref{jantzen-section} respects the $W$-action.  In particular, in the setting considered in that section, each isotypic subspace $\Delta_c(\lambda)^\pi$ is equipped with a Jantzen filtration $\{\Delta_c(\lambda)^{\pi, \geq k}\}_{k \geq 0}$, and we have $\Delta_c(\lambda)^{\geq k} = \bigoplus_{\pi \in \irr(W)} \Delta_c(\lambda)^{\pi, \geq k}.$  Furthermore, the wall-crossing formulas in Lemma \ref{wall-crossing-lemma} and the decomposition in Lemma \ref{K0-lemma} of an arbitrary signature character $\sch(M, \beta)$ in terms of signature characters $\sch(L_i)$ of irreducible representations, and their proofs, have direct analogues for the isotypic signatures characters $\sch^\pi$.

Let $l = \dim \hfr$, let $\{d_i\}_{i = 1}^l$ denote the fundamental degrees of $W$, and let $\cplx[\hfr]^{coW} := \cplx[\hfr]/\cplx[\hfr](\cplx[\hfr]^W)^+$ be the coinvariant algebra.  Recall that $\cplx[\hfr]^{coW}$ is graded and is isomorphic to $\cplx W$ as a $\cplx W$-module and that the graded dimension of $\cplx[\hfr]^W$ is given by $\prod_{i = 1}^l (1 - t^{d_i})^{-1}$.  For any $\lambda, \pi \in \irr(W)$, let $\theta_\lambda^\pi \in \ints^{\geq 0}[t]$ denote the graded dimension of the $\pi$-isotypic subspace of $\lambda \otimes \cplx[\hfr]^{coW}$.

\begin{lemma} \label{unitary-isotypic-lemma} For any $\lambda, \pi \in \irr(W)$, we have:

(1) $\theta_\lambda^\pi(1) = (\dim \lambda)(\dim \pi)^2$

(2) $\ch_0^\pi(\Delta_0(\lambda)) = \sch_0^\pi(\Delta_0(\lambda)) = \theta^\pi_\lambda(t)/\prod_{i = 1}^l (1 - t^{d_i}).$

(3) $$\lim_{n \rightarrow \infty} \frac{\dim \Delta_0(\lambda)^{\pi, \leq n}}{\dim \Delta_0(\lambda)^{\leq n}} = \lim_{t \rightarrow 1^-} \frac{\ch_0^\pi(\Delta_0(\lambda))(t)}{\ch_0(\Delta_0(\lambda))(1, t)} = \frac{(\dim \pi)^2}{|W|}.$$\end{lemma}

\begin{proof} $\theta_\lambda^\pi(1)$ is the dimension of the $\pi$-isotypic subspace of $\lambda \otimes \cplx[\hfr]^{coW};$ as $\cplx[\hfr]^{coW} \cong \cplx W$ as a $\cplx W$-module and as $\lambda \otimes \cplx W \cong \cplx W^{\oplus \dim \lambda}$, the first statement follows.  The second statement follows from the fact that $\cplx[\hfr] \cong \cplx[\hfr]^{coW} \otimes \cplx[\hfr]^W$ as a graded $\cplx W$-module.  The third statement follows from the first two statements and from Lemma \ref{coef-ratio-lemma}, using the facts that $\ch_0(\Delta_0(\lambda))(1, t) = (\dim \lambda)P_W(t)/\prod_{i = 1}^l (1 - t^{d_i})$, where $P_W(t)$ is the Poincar\'{e} polynomial of $W$, and $P_W(1) = |W|$.\end{proof}

Note that Lemma \ref{wall-crossing-lemma} holds for the isotypic signature characters $\sch^\pi$ as well, by the same proof as for the usual signature characters $\sch$.  As the signs $\epsilon_i$ appearing in Lemma \ref{K0-lemma} have no dependence on the irreducible representation $\pi$, it follows from Lemma \ref{unitary-isotypic-lemma} and the proof of Proposition \ref{rationality-prop} that we have:

\begin{lemma} \label{isotypic-rationality-lemma} For any $c \in \pfr_\real$, there exists a collection of polynomials $\{n_c^{\lambda, \mu}(t) \in \ints[t] : \lambda, \mu \in \irr(W)\}$ such that for every $\lambda, \pi \in \irr(W)$ we have $$\sch_0^\pi(L_c(\lambda))(t) = \sum_{\mu \in \irr(W)} n_c^{\lambda, \mu}(t) \sch_0^\pi(\Delta_0(\mu))(t)$$ $$= \prod_{i = 1}^l (1 - t^{d_i})^{-1}\sum_{\mu \in \irr(W)} n_c^{\lambda, \mu}(t) \theta_\mu^\pi(t).$$  In particular, $\sch_0^\pi(L_c(\lambda))(t)$ is a rational function of $t$.\end{lemma}

\begin{corollary} In the setting of Lemma \ref{isotypic-rationality-lemma}, the rational function $\sch_0^\pi(L_c(\lambda))(t)$ has a pole of order at most $r := \dim \supp(L_c(\lambda))$ at $t = 1$.\end{corollary}

\begin{proof} By Lemma \ref{isotypic-rationality-lemma}, we see that $\sch_0^\pi(L_c(\lambda))(t)$ is absolutely convergent for $|t| < 1$, and by a comparison of coefficients we see that the rational function $$\sch_0^\pi(L_c(\lambda))(t)/\ch_0(L_c(\lambda))(1, t)$$ takes values in $[-1, 1]$ on the interval $[0, 1)$.  As $\ch_0(L_c(\lambda))(1, t)$ has a pole of order $r$ at $t = 1$, the claim follows.\end{proof} 

The following proposition allows the asymptotic signature $a_{c, \lambda}$ of $L_c(\lambda)$ to be computed in any isotypic component when $L_c(\lambda)$ has full support.  Considering the case when $L_c(\lambda)$ is finite-dimensional, note that this need not be true when $L_c(\lambda)$ has proper support.

\begin{proposition} Let $c \in \pfr_\real$, $\lambda \in \irr(W)$, and suppose $L_c(\lambda)$ has full support.  Then for all $\pi \in \irr(W)$ the limit $$\lim_{n \rightarrow \infty} \frac{\sign(\beta_{c, \lambda}^{\pi, \leq n})}{\dim L_c(\lambda)^{\pi, \leq n}}$$ exists and equals $a_{c, \lambda}$.  In particular, this limit is independent of $\pi$.\end{proposition} 

\begin{proof}As we have $$a_{c, \lambda} = \lim_{n \rightarrow \infty} \frac{\sign(\beta_{c, \lambda}^{\leq n})}{\dim L_c(\lambda)^{\leq n}} = \lim_{n \rightarrow \infty} \frac{\sum_{\pi \in \irr(W)} \sign(\beta_{c, \lambda}^{\pi, \leq n})}{\sum_{\pi \in \irr(W)} \dim L_c(\lambda)^{\pi, \leq n}}$$ it suffices to show that the limit in the proposition statement exists and is independent of $\pi$.  As $L_c(\lambda)$ has full support, considering the decomposition of $[L_c(\lambda)]$ in $K_0(\oscr_c(W, \hfr))$ in terms of the classes of standard modules $[\Delta_c(\mu)]$ for various $\mu \in \irr(W)$, we have
\begin{equation}\label{eq:isotypic-to-whole-ratio}
\lim_{n \rightarrow \infty} \frac{\dim L_c(\lambda)^{\pi, \leq n}}{\dim L_c(\lambda)^{\leq n}} = \lim_{t \rightarrow 1^-} \frac{\ch_0^\pi(L_c(\lambda))(t)}{\ch_0(L_c(\lambda))(1, t)} = \frac{(\dim \pi)^2}{|W|}
\end{equation}
where the first equality follows from Lemma \ref{coef-ratio-lemma} and the second equality follows from Lemma \ref{unitary-isotypic-lemma}(3) and the fact that $L_c(\lambda)$ has full support in $\hfr$.  By equation (\ref{eq:isotypic-to-whole-ratio}) and Lemmas \ref{coef-ratio-lemma} and \ref{unitary-isotypic-lemma} we have $$\lim_{n \rightarrow \infty} \frac{\sign(\beta_{c, \lambda}^{\pi, \leq n})}{\dim L_c(\lambda)^{\pi, \leq n}}$$ $$=\frac{|W|}{(\dim \pi)^2} \lim_{n \rightarrow \infty} \frac{\sign(\beta_{c, \lambda}^{\pi, \leq n})}{\dim L_c(\lambda)^{\leq n}}$$ $$= \frac{|W|}{(\dim \pi)^2} \lim_{t \rightarrow 1^-} \frac{\sch_0^\pi(L_c(\lambda))(t)}{\ch_0(L_c(\lambda))(1, t)}$$ $$= \frac{|W|}{(\dim \pi)^2} \lim_{t \rightarrow 1^-} \frac{\sum_{\mu \in \irr(W)} n_c^{\lambda, \mu}(t)\theta_{\mu}^\pi(t)/\prod_{i = 1}^l (1 - t^{d_i})}{\ch_0(L_c(\lambda))(1, t)}$$ $$= |W| \lim_{t \rightarrow 1^-} \frac{\sum_{\mu \in \irr(W)} n_c^{\lambda, \mu}(1)\dim \mu}{\ch_0(L_c(\lambda))(1, t)\prod_{i = 1}^l (1 - t^{d_i})},$$ which is visibly independent of $\pi$, as needed.  Note that this final limit exists because $L_c(\lambda)$ has full support in $\hfr$, so $\ch_0(L_c(\lambda))(1, t)$ as a pole of order $l$ at $t = 1$.\end{proof}

\section{The Dunkl Weight Function}\label{weight-function-section}

\subsection{The Contravariant Pairing and Contravariant Form}

As in the previous section, let $W$ be a finite complex reflection group with reflection representation $\hfr$.  Fix a positive-definite $W$-invariant inner product on $\hfr_\real$, and let $(\cdot, \cdot)_\hfr$ be its unique extension to a $W$-invariant positive-definite Hermitian form on $\hfr$.  Let $T : \hfr \rightarrow \hfr^*$, $T(y)(x) = (x, y)_\hfr$, be the conjugate-linear $W$-invariant isomorphism introduced in Section \ref{jantzen-section}, giving rise for any parameter $c \in \pfr$ to the isomorphism $\sigma : \bar{H_c(W, \hfr)^{opp}} \rightarrow H_{c^\dagger}(W, \hfr)$ of $\cplx$-algebras and conjugate-linear equivalence of categories $$^h : \oscr_c(W, \hfr) \rightarrow \oscr_{c^\dagger}(W, \hfr)^{opp}$$ introduced in Section \ref{hermitian-duals-section}.

Let $\lambda \in \irr(W)$ be an irreducible representation of $W$, and fix a positive-definite $W$-invariant Hermitian form $(\cdot, \cdot)_\lambda$ on $\lambda$.  For any $c \in \pfr$, the $W$-invariant isomorphism $\phi : \lambda \rightarrow \bar{\lambda^*}$ given by $\phi(v)(v') = (v', v)$ between the lowest weight spaces of $\Delta_{c^\dagger}(\lambda)$ and $\Delta_c(\lambda)^h$ extends uniquely to an $H_{c^\dagger}(W, \hfr)$-homomorphism $\phi : \Delta_{c^\dagger}(\lambda) \rightarrow \Delta_c(\lambda)^h$.  When it is convenient, we will use the notation $v_2^\dagger v_1$ to denote the inner product $(v_1, v_2)_\lambda$ for $v_1, v_2 \in \lambda$, borrowing notation from the standard inner product on $\cplx^N$.  Similarly, we will use the notation $A^\dagger$ to denote the adjoint of an operator $A \in \End_\cplx(\lambda)$ with respect to the form $(\cdot, \cdot)_\lambda$.

\begin{definition}\label{beta-def} For $c \in \pfr$ and $\lambda \in \irr(W)$, let $\beta_{c, \lambda}$ be the sesquilinear pairing $$\beta_{c, \lambda}: \Delta_c(\lambda) \times \Delta_{c^\dagger}(\lambda) \rightarrow \cplx$$ given by $$\beta_{c, \lambda}(v, v') = \phi(v')(v).$$\end{definition}

\begin{remark} $\beta_{c, \lambda}$ depends on the choice of forms $(\cdot, \cdot)_\hfr$ and $(\cdot, \cdot)_\lambda$ and is therefore determined by $c \in \pfr$ and $\lambda \in \irr(W)$ only up to a positive real multiple.\end{remark}

\begin{remark} When the parameter $c \in \pfr$ is real, i.e. $c \in \pfr_\real$ and $c = c^\dagger$, the pairing $\beta_{c, \lambda}$ is precisely the contravariant Hermitian form on $\Delta_c(\lambda)$ considered in \cite{C} and \cite[Definition 4.5]{ESG}.  While this case is the primary case of interest, we will consider the more general setting of complex $c$ to show that the associated Dunkl weight function, introduced later, is defined and holomorphic for all $c \in \pfr$.\end{remark}

The following proposition is standard (for $c \in \pfr_\real$ it is a reformulation of \cite[Proposition 2.2]{ESG}):

\begin{proposition}\label{beta-props}\ 

(i)  $\beta_{c, \lambda}$ is the unique sesquilinear form $\beta_{c, \lambda} : \Delta_c(\lambda) \times \Delta_{c^\dagger}(\lambda) \rightarrow \cplx$ that is:

\indent \indent (a) normalized: $\beta_{c, \lambda}|_\lambda = (\cdot, \cdot)_\lambda$

\indent \indent (b) contravariant: $\beta_{c, \lambda}(hv, v') = \beta_{c, \lambda}(v, \sigma(h)v')$ for all $h \in H_c(W, \hfr)$, $v \in \Delta_c(\lambda)$, and \indent \indent \indent \indent $v' \in \Delta_{c^\dagger}(\lambda),$

(ii) Any contravariant form $\beta : \Delta_c(\lambda) \times \Delta_{c^\dagger}(\lambda) \rightarrow \cplx$ is proportional to $\beta_{c, \lambda}.$

(iii)  $\beta_{c, \lambda}$ is graded, $W$-invariant, and satisfies $\beta_{c, \lambda} = \beta_{c^\dagger, \lambda}^\dagger$.

(iv)  $\beta_{c, \lambda}$ descends to a nondegenerate pairing $\beta_{c, \lambda} : L_c(\lambda) \times L_{c^\dagger}(\lambda) \rightarrow \cplx$.

\end{proposition}

\subsection{The Gaussian Pairing and Gaussian Inner Product}

For the remainder of this paper, let $W$ be a finite real reflection group (equivalently, a finite Coxeter group).  Let $\hfr_\real$ denote its real reflection representation, let $\hfr = \cplx \otimes_\real \hfr_\real$ denote the complexified reflection representation, let $\{\alpha_s\}_{s \in S} \subset \hfr_\real^*$ be a system of positive roots for $W$, let $\{\alpha_s^\vee\}_{s \in S} \subset \hfr_\real$ be the associated system of positive coroots, and fix an orthonormal basis $y_1, ..., y_l$ of $\hfr_\real$ with associated dual basis $x_1, ..., x_l \in \hfr_\real^*$.  In this setting, $H_c(W, \hfr)$ contains an internal $W$-invariant $\slfr_2$-triple $\mathbf{e}, \mathbf{f}, \mathbf{h}$ given by $$\mathbf{h} = \sum_i x_iy_i + \frac{1}{2}\dim \hfr - \sum_{s \in S} c_ss = \sum_i (x_iy_i + y_ix_i)/2, \ \ \ \mathbf{e} := -\frac{1}{2}\sum_i x_i^2 \ \ \ \mathbf{f} := \frac{1}{2}\sum_i y_i^2.$$  The elements $\mathbf{e}, \mathbf{f}, \mathbf{h}$ do not depend on the choice of orthonormal basis $y_1, ..., y_l \in \hfr_\real$.

Note that the operator $\mathbf{f}$ acts locally nilpotently on any lowest weight module.  In particular, following \cite{C} and \cite[Definition 4.5]{ESG}, we make the following definition:

\begin{definition} \label{gamma-def} The \emph{Gaussian pairing} $\gamma_{c, \lambda}$ is the sesquilinear pairing $$\gamma_{c, \lambda} : \Delta_c(\lambda) \times \Delta_{c^\dagger}(\lambda) \rightarrow \cplx$$ defined by $$\gamma_{c, \lambda}(P, Q) = \beta_{c, \lambda}(\exp(\mathbf{f})P, \exp(\mathbf{f})Q).$$  When $c \in \pfr_\real$, $\gamma_{c, \lambda}$ is a Hermitian form on $\Delta_c(\lambda)$ and will be called the \emph{Gaussian inner product}.\end{definition}

\begin{remark} For $c \in \pfr_\real$, the study of the signatures of the Hermitian forms $\beta_{c, \lambda}$ and $\gamma_{c, \lambda}$ is equivalent; we have $\sign(\beta_{c, \lambda}^{\leq n}) = \sign(\gamma_{c, \lambda}^{\leq n})$ for all $n \geq 0$, where $\beta_{c, \lambda}^{\leq n}$ and $\gamma_{c, \lambda}^{\leq n}$ denote the restrictions of $\beta_{c, \lambda}$ and $\gamma_{c, \lambda}$, respectively, to $\Delta_c(\lambda)^{\leq n}$.  In particular, the asymptotic signature $a_{c, \lambda}$ can be equally well computed using the form $\gamma_{c, \lambda}$.  As we will see, the advantage of the form $\gamma_{c, \lambda}$ is that $\hfr_\real^*$ acts on $\Delta_c(\lambda)$ by self-adjoint operators, making an integral representation of this form possible and allowing for analytic techniques to be used to study the signatures.\end{remark}

The following proposition generalizes \cite[Proposition 4.6]{ESG} to complex $c \in \pfr$ and provides a useful characterization of $\gamma_{c, \lambda}$:

\begin{proposition}\label{gamma-props}\ 

(i) The pairing $\gamma_{c, \lambda}$ satisfies: $$\gamma_{c, \lambda}(xP, Q) = \gamma_{c, \lambda}(P, xQ) \ \ \ \ \text{for all } x \in \hfr_\real^*, P \in \Delta_c(\lambda), Q \in \Delta_{c^\dagger}(\lambda).$$

(ii) Up to scaling, $\gamma_{c, \lambda}$ is the unique $W$-invariant sesquilinear pairing $\Delta_c(\lambda) \times \Delta_{c^\dagger}(\lambda) \rightarrow \cplx$ satisfying $$\gamma_{c, \lambda}((-y + Ty)P, Q) = \gamma_{c, \lambda}(P, yQ) \ \ \ \ \text{for all } y \in \hfr_\real, P \in \Delta_c(\lambda), Q \in \Delta_{c^\dagger}(\lambda)$$ and $$\gamma_{c, \lambda}(yP, Q) = \gamma_{c, \lambda}(P, (-y + Ty)Q) \ \ \ \ \text{for all } y \in \hfr_\real, P \in \Delta_c(\lambda), Q \in \Delta_{c^\dagger}(\lambda).$$ $\gamma_{c, \lambda}$ is determined among such forms by the normalization condition $\gamma_{c, \lambda}|_\lambda = (\cdot, \cdot)_\lambda$.

(iii) Up to scaling, $\gamma_{c, \lambda}$ is the unique $W$-invariant sesquilinear pairing $\Delta_c(\lambda) \times \Delta_{c^\dagger}(\lambda) \rightarrow \cplx$ satisfying $$\gamma_{c, \lambda}(xP, Q) = \gamma_{c, \lambda}(P, xQ) \ \ \ \ \text{for all } x \in \hfr_\real^*, P \in \Delta_c(\lambda), Q \in \Delta_{c^\dagger}(\lambda)$$ and $$\gamma_{c, \lambda}((-y + Ty)P, v) = 0 \ \ \ \ \text{for all } y \in \hfr_\real, P \in \Delta_c(\lambda), v \in \lambda.$$ $\gamma_{c, \lambda}$ is determined among such forms by the normalization condition $\gamma_{c, \lambda}|_\lambda = (\cdot, \cdot)_\lambda$.

(iv) $\gamma_{c, \lambda} = \gamma_{c^\dagger, \lambda}^\dagger$.\end{proposition}

\begin{proof} Parts (i) and (ii) are proved using Proposition \ref{beta-props} and exactly the same calculations and arguments appearing in the proof of \cite[Proposition 4.6]{ESG}.  To prove (iii), it suffices to show that the properties in (iii) imply the properties in (ii).  Let $\gamma$ be a pairing as in (iii).  As $\hfr_\real^*$ acts by operators self-adjoint with respect to $\gamma$, the two equations appearing in (ii) are equivalent because $Ty \in \hfr_\real^*$ is self-adjoint, so it suffices to prove that
\begin{equation}\label{eq:y-adjointness}
\gamma((-y + Ty)P, Q) = \gamma(P, yQ)
\end{equation} for all $y \in \hfr_\real, P \in \Delta_c(\lambda),$ and $Q \in \Delta_{c^\dagger}(\lambda)$.  We will prove equation (\ref{eq:y-adjointness}) by induction on $\deg Q$.  As $yv = 0$ for all $y \in \hfr_\real$ and $v \in \lambda$, equation (\ref{eq:y-adjointness}) holds for all $y \in \hfr_\real, P \in \Delta_c(\lambda),$ and $Q \in \lambda$ by the second property in (iii), establishing the base case $\deg Q = 0$.  Now suppose, for some $N \geq 0$, that equation (\ref{eq:y-adjointness}) holds for all $y \in \hfr_\real, P \in \Delta_c(\lambda),$ and $Q \in \Delta_{c^\dagger}(\lambda)$ with $\deg Q \leq N$.  For any $x \in \hfr_\real^*, y \in \hfr_\real, P \in \Delta_c(\lambda)$, and $Q \in \Delta_{c^\dagger}(\lambda)$ we have $$\gamma((-y + Ty)P, xQ) - \gamma(P, yxQ)$$ $$= \gamma(x(-y + Ty)P, Q) - \gamma\left(P, \left(xy +x(y) - \sum_{s \in S} \bar{c_s}\alpha_s(y)\alpha_s^\vee(x)s\right)Q\right)$$ $$= \gamma(x(-y + Ty)P, Q) - \gamma\left(\left((-y + Ty)x + x(y) - \sum_{s \in S} c_s\alpha_s(y)\alpha_s^\vee(x)s\right)P, Q\right)$$ $$= 0$$ where we use the commutation relation for $[y, x]$ in $H_{c^\dagger}(W, \hfr)$ and $H_c(W, \hfr)$ in the first and last equality, respectively, and the fact that $[x, Ty] = 0$, establishing the inductive step.

Statement (iv) follows immediately from $\beta_{c, \lambda} = \beta_{c^\dagger, \lambda}^\dagger$.\end{proof}

\subsection{Dunkl Weight Function: Characterization and Properties}\label{weight-function-char-and-props-section}  The remainder of Section \ref{weight-function-section} is dedicated to proving the following main theorem:

\begin{theorem} \label{main-existence-theorem} For any finite Coxeter group $W$ and irreducible representation $\lambda \in \irr(W)$, there is a unique family $K_{c, \lambda}$, holomorphic in $c \in \pfr$, of $\End_\cplx(\lambda)$-valued tempered distributions on $\hfr_\real$ such that the following integral representation of the Gaussian pairing $\gamma_{c, \lambda}$ holds for all $c \in \pfr$: $$\gamma_{c, \lambda}(P, Q) = \int_{\hfr_\real} Q(x)^\dagger K_{c, \lambda}(x)P(x)e^{-|x|^2/2}dx \ \ \ \ \text{ for all } P, Q \in \cplx[\hfr] \otimes \lambda,$$ where we make the standard identification of $\Delta_c(\lambda)$ and $\Delta_{c^\dagger}(\lambda)$ with $\cplx[\hfr] \otimes \lambda$.  Furthermore, $K_{c, \lambda}$ satisfies the additional properties:

(i)  For all $M > 0$ there exists an integer $N \geq 0$, which may be taken to be $0$ for $M$ sufficiently small, such that for $c \in \pfr$ with $|c_s| < M$ for all $s \in S$ the distribution $\delta^{N}K_{c, \lambda}$, where $\delta := \prod_{s \in S} \alpha_s \in \cplx[\hfr]$ is the discriminant element, is given by integration against an analytic function on $\hfr_{\real, reg}$ that is locally integrable over $\hfr_\real$.

(ii) For any $x \in \hfr_{\real, reg}$ the operator $K_{c, \lambda}(x) \in \End_\cplx(\lambda)$ determines a $B_W$-invariant sesquilinear pairing $$KZ_x(\Delta_c(\lambda)) \times KZ_x(\Delta_{c^\dagger}(\lambda)) \rightarrow \cplx$$ by the formula $(v_1, v_2) = v_2^\dagger K_{c, \lambda}(x)v_1.$  When $c \in \pfr_\real$, i.e. $c = c^\dagger$, this pairing is a Hermitian form.
\end{theorem}

\begin{remark}The uniqueness of $K_{c, \lambda}$ is immediate from the standard fact that the subspace $\cplx[\hfr]e^{-|x|^2/2}$ is dense in the space of complex-valued Schwartz functions $\mathscr{S}(\hfr_\real)$ on $\hfr_\real$.\end{remark}

\begin{remark} The existence of the distribution $K_{c, \lambda}$ was shown in the case that $W$ is a dihedral group and $c$ is real and small (for any $\lambda)$ by Dunkl \cite{Dunkl-dihedral}.  The case in which $\lambda$ is 1-dimensional is much simpler than the case of general $\lambda$ and was used by Etingof in \cite{E2} to study the support of the irreducible representation $L_c(\triv)$.  Related results in the trigonometric case in type $A$ appear in \cite{Dunkl-trig-1, Dunkl-trig-2, Dunkl-trig-3}.\end{remark}

In preparation for the proof of Theorem \ref{main-existence-theorem}, in the remainder of Section \ref{weight-function-char-and-props-section} we will now establish several necessary properties and an additional characterization of the distribution $K_{c, \lambda}$.  The proof of Theorem \ref{main-existence-theorem} will then proceed in three steps: in Section \ref{small-c-section} we will establish the existence of $K_{c, \lambda}$ for small $c$ and show that in this case it is given by integration against a locally $L^1$ function; in Section \ref{extension-to-reg-locus-section} we will produce an analytic continuation for all $c \in \pfr$ of this function over $\hfr_{\real, reg}$; and in Section \ref{extension-to-hyperplanes-section} we will show that these functions extend naturally to distributions on all of $\hfr_\real$, completing the proof of Theorem \ref{main-existence-theorem}.

An element $y \in \hfr \subset H_c(W, \hfr)$ acts in the representation $\Delta_c(\lambda) = \cplx[\hfr] \otimes \lambda$ by the operator $$D_y = \partial_y - \sum_{s \in S} c_s\frac{\alpha_s(y)}{\alpha_s}(1 - s) \otimes s.$$  This operator $D_y$ acts naturally as a continuous operator on the Schwartz space $\mathscr{S}(\hfr_\real)$, and we will use the notation $D_y$ for this continuous operator as well, with the representation $\lambda$ and parameter $c$ implicit and suppressed from the notation.  We may equivalently view any tempered distribution $K$ on $\hfr_\real$ with values in $\End_\cplx(\lambda)$ as a continuous operator $$K : \mathscr{S}(\hfr_\real) \otimes \lambda \rightarrow \lambda.$$  In particular, for any such distribution $K$, the composition $K \circ D_y$ is defined and is another such distribution.  With this in mind, it is now straightforward to translate the conditions in Proposition \ref{gamma-props}(iii) into a convenient characterization of the Dunkl weight function $K_{c, \lambda}$:

\begin{proposition}\label{dunkl-weight-characterization-proposition} Let $K$ be an $\End_\cplx(\lambda)$-valued tempered distribution on $\hfr_\real$, and let $\gamma$ be the sesquilinear form on $\cplx[\hfr] \otimes \lambda$ defined by $$\gamma(P, Q) = \int_{\hfr_\real} Q(x)^\dagger K(x)P(x)e^{-|x|^2/2}dx.$$  Then $\gamma = \gamma_{c, \lambda}$ if and only if the following hold:

(i) normalization: $$\int_{\hfr_\real} K(x)e^{-|x|^2/2}dx = \id_\lambda$$

(ii) $W$-invariance: $wK(w^{-1}x)w^{-1} = K(x)$ for all $w \in W$

(iii) annihilation by Dunkl operators: $K \circ D_y = 0$ for all $y \in \hfr_\real$.\end{proposition}

\begin{proof} It is immediate that the normalization condition (i) above and the normalization condition $\gamma|_\lambda = (\cdot, \cdot)_\lambda$ are equivalent and that the $W$-invariance condition (ii) above is equivalent to the $W$-invariance of the form $\gamma$.  It is also immediate that $\gamma(xP, Q) = \gamma(P, xQ)$ for all $x \in \hfr_\real$ and $P, Q \in \cplx[\hfr] \otimes \lambda$, so it remains to show that condition (iii) is equivalent to the condition $\gamma((-y_i + x_i)P, v) = 0$ for all $P \in \cplx[\hfr] \otimes \lambda, v \in \lambda,$ and $i = 1, ..., l = \dim \hfr$.  As $e^{-|x|^2/2}$ is $W$-invariant we have $D_{y_i}(P(x)e^{-|x|^2/2}) = (D_{y_i}P(x) - x_i)e^{-|x|^2/2}$, and a direct calculation then gives $$\gamma((-y_i + x_i)P, v) = \int_{\hfr_\real} v^\dagger K(x)(-y_i + x_i)(P(x))e^{-|x|^2/2}dx$$ $$= \int_{\hfr_\real}v^\dagger(K \circ D_{y_i})(P(x)e^{-|x|^2/2})dx.$$  As $\cplx[\hfr]e^{-|x|^2/2}$ is dense in $\mathscr{S}(\hfr_\real)$, we see that $\gamma((-y_i + x_i)P, v) = 0$ for all $P \in \cplx[\hfr] \otimes \lambda$ and $v \in \lambda$ if and only if $K \circ D_{y_i} = 0$, as needed.\end{proof}

\begin{proposition}
Any family of tempered distributions $K_{c, \lambda}$ as in Theorem \ref{main-existence-theorem} satisfies the equation $K_{c, \lambda} = K_{c^\dagger, \lambda}^\dagger.$  In particular, $K_{c, \lambda}$ takes values in Hermitian forms on $\lambda$ when $c \in \pfr_\real$.
\end{proposition}

\begin{proof} Immediate from the equality $\gamma_{c, \lambda} = \gamma_{c^\dagger, \lambda}^\dagger$.\end{proof}

\begin{proposition}\label{dunkl-weight-props} Any distribution $K$ on $\hfr_\real$ with values in $\End_\cplx(\lambda)$ satisfying the conditions (i), (ii), and (iii) appearing in Proposition \ref{dunkl-weight-characterization-proposition} satisfies the following properties:

(i) The restriction $K|_{\hfr_{\real, reg}}$ of $K$ to the real regular locus $$\hfr_{\real, reg} := \{x \in \hfr_\real : \alpha_s(x) \neq 0 \text{ for all } s \in S\}$$ satisfies the 2-sided KZ-type system of differential equations:
\begin{equation}\label{eq:2-sided-KZ}
\partial_y K + \sum_{s \in S} c_s\frac{\alpha_s(y)}{\alpha_s}(sK + Ks) = 0 \ \ \ \ \text{ for all } y \in \hfr_\real
\end{equation}

(ii) $K$ is a homogeneous distribution of degree $-2\chi_\lambda(\sum_{s \in S} c_ss)/\dim \lambda,$ where $\chi_\lambda$ is the character of $\lambda$.  In particular, $K$ is tempered.

\end{proposition}

\begin{proof}  To prove (i), take an arbitrary compactly supported test function $\phi \in C^\infty_c(\hfr_{\real, reg}) \otimes \lambda$.  For any $y \in \hfr_\real$, we have $(K \circ D_y)(\phi) = 0$.  From the definition of $D_y$ and the fact that $(1 \otimes s)(\phi)/\alpha_s$ and $(s \otimes s)(\phi)/\alpha_s$ are also well-defined test functions in $C^\infty_c(\hfr_\real) \otimes \lambda$ for any $s \in S$, we have: $$0 = \int_{\hfr_\real} (K \circ D_y)(\phi)dx$$ $$= \int_{\hfr_\real} K_c(x)\left(\partial_y \phi(x) - \sum_{s \in S} c_s\frac{\alpha_s(y)}{\alpha_s(x)}(s\phi(x) - s\phi(sx))\right)dx$$ $$= \int_{\hfr_\real} \left(-\partial_y K(x) - \sum_{s \in S} c_s\frac{\alpha_s(y)}{\alpha_s(x)}K(x)s\right)\phi(x)dx$$ $$+ \sum_{s \in S} c_s\int_{\hfr_\real} \frac{\alpha_s(y)}{\alpha_s(x)}K(x)s\phi(sx)dx$$ $$= \int_{\hfr_\real} \left(-\partial_y K(x) - \sum_{s \in S} c_s\frac{\alpha_s(y)}{\alpha_s(x)}K(x)s\right)\phi(x)dx$$ $$+ \sum_{s \in S} c_s\int_{\hfr_\real}\frac{\alpha_s(y)}{\alpha_s(sx)}K(sx)s\phi(x)dx$$ $$= \int_{\hfr_\real} \left(-\partial_y K(x) - \sum_{s \in S} c_s\frac{\alpha_s(y)}{\alpha_s(x)}K(x)s\right)\phi(x)dx$$ $$- \sum_{s \in S} c_s\int_{\hfr_\real}\frac{\alpha_s(y)}{\alpha_s(x)}sK(x)\phi(x)dx$$ $$= -\int_{\hfr_\real} \left(\partial_y K + \sum_{s \in S} c_s\frac{\alpha_s(y)}{\alpha_s(x)}(sK + Ks)\right)\phi(x)dx,$$ where in the fourth equality we change variables $x \mapsto sx$ and in the fifth equality we use the $W$-invariance of $K$ and the equality $\alpha_s(sx) = -\alpha_s(x)$.  As $\phi \in C^\infty_c(\hfr_{\real, reg}) \otimes \lambda$ was arbitrary, this proves (i).

Next, note that as $x_1, ..., x_l \in \hfr_\real^*$ and $y_1, ..., y_l \in \hfr_\real$ are dual bases we have $$\sum_{i = 1}^l f(y_i)w(x_i) = w(f)$$ for all $f \in \hfr^*$ and $s \in S$.  It follows that we have, as operators on $\cplx[\hfr] \otimes \lambda$,
\begin{equation}\label{eq:dual-basis-calculation}
\sum_{i = 1}^l \frac{\alpha_s(y_i)}{\alpha_s}(1 - s)(x_i) = \frac{1}{\alpha_s}\left(\sum_{i = 1}^l \alpha_s(y_i)(x_i - s(x_i)s)\right) = \frac{1}{\alpha_s}(\alpha_s - s(\alpha_s)s) = 1 + s
\end{equation}  Using the fact that $K \circ D_y = 0$ for all $y \in \hfr_\real$, the action of the Euler operator $\sum_{i = 1}^l x_i\partial_{y_i}$ on the distribution $K(x)$ is therefore given by $$\left(\sum_{i = 1}^l x_i\partial_{y_i}\right)K(x) = -K(x) \sum_{i = 1}^l \partial_{y_i}x_i$$ $$= -K(x)\sum_{i = 1}^l \left(\sum_{s \in S}c_s\frac{\alpha_s(y_i)}{\alpha_s}(1 - s) \otimes s\right)x_i$$ $$= -K(x) \left(\sum_{s \in S}c_s\left(\sum_{i = 1}^l \frac{\alpha_s(y_i)}{\alpha_s}(1 - s)x_i\right) \otimes s\right)$$ $$= -K(x)\sum_{s \in S} c_s(1 + s) \otimes s$$ $$= -\sum_{s \in S} c_s(sK(x) + K(x)s)$$ $$= \frac{-2\chi_\lambda(\sum_{s \in S} c_ss)}{\dim \lambda}K(x),$$ where the second equality follows from the equation $K \circ D_{y_i} = 0$, the fourth equality uses equation (\ref{eq:dual-basis-calculation}), the fifth equality uses the $W$-invariance of $K(x)$, and the final equality uses the fact that $\sum_{s \in S} c_ss$ is central in $\cplx W$ and therefore acts on $\lambda$ by the scalar $\chi_\lambda(\sum_{s \in S} c_ss)/\dim \lambda$, giving (ii).\end{proof}

\begin{proposition} For any $c \in \pfr$, the support of the distribution $K_{c, \lambda}$ coincides with the set of real points of the support of $L_c(\lambda)$: $$\supp(K_{c, \lambda}) = \supp(L_c(\lambda))_\real.$$
\end{proposition}

\begin{proof} This was shown by Etingof \cite[Proposition 3.10]{E2} in the case $\lambda = \triv$, and the proof generalizes to arbitrary $\lambda$ without modification.\end{proof}

\subsection{Existence of Dunkl Weight Function for Small $c$}\label{small-c-section}

Let $\alpha_1, ..., \alpha_l$ be the simple roots associated to the system of positive roots $\{\alpha_s\}_{s \in S}$, let $$\ccal := \{x \in \hfr_\real : \alpha_i(x) > 0 \text{ for } i = 1, ..., l\}$$ be the open associated fundamental Weyl chamber, and let $\bar{\ccal}$ denote its closure.  Recall that $$\delta := \prod_{s \in S} \alpha_s \in \cplx[\hfr]$$ is the discriminant element defining the reflection hyperplanes of $W$.

\begin{lemma} \label{homogeneous-L1-lemma} Let $c \in \pfr$ and let $K$ be a $\End(\lambda)$-valued analytic function on $\hfr_{\real, reg}$ satisfying the differential equation $$\partial_y K + \sum_{s \in S} \frac{c_s\alpha_s(y)}{\alpha_s}(sK + Ks) = 0$$ for each $y \in \hfr_\real$.  Then

(i) $K$ is homogeneous of degree $-2\chi_\lambda(\sum_{s \in S} c_ss)/\dim \lambda$;

(ii) there exists an integer $N \geq 0$ such that $\delta^NK$, extended to $\hfr_\real$ with value $0$ along the reflection hyperplanes $\ker(\alpha_s)$, is a continuous function on $\hfr_\real$;

(iii) if $|c_s|$ is sufficiently small for all $s \in S$ then $K(x)$ is a locally $L^1$ function on $\hfr_\real$. \end{lemma}

\begin{proof} Statement (i) follows immediately from the differential equation and the fact that the central element $\sum_{s \in S} c_ss \in \cplx W$ acts on $\lambda$ by the scalar $\chi_\lambda(\sum_{s \in S} c_ss)/\dim \lambda$.

To establish (ii) and (iii) it suffices to show that $K$ is locally integrable on $\ccal$ when $|c_s|$ is sufficiently small for all $s \in S$ and that there exists $N \geq 0$ such that $\delta^NK|_\ccal$ extends to a continuous function on $\bar{\ccal}$ vanishing on the boundary.  Let $\omega_1^\vee, ..., \omega_l^\vee \in \bar{\mathcal{C}}$ be the fundamental dominant coweights, so that $\alpha_i(\omega_j^\vee) = \delta_{ij}$, and let $R$ be the region $$R := \{x \in \hfr_{\real, reg} : \alpha_i(x) \in (0, 1), i = 1, ..., l\} \subset \ccal.$$  Given a point $v \in \ccal$, we have, for those $t > 0$ such that $\alpha_i(v - t\omega_i^\vee) > 0$, $$\frac{d}{dt}||K(v - t\omega_i^\vee)|| \leq ||\frac{d}{dt}K(v - t\omega_i^\vee)|| $$ $$= ||\sum_{s \in S} \frac{c_s\alpha_s(\omega_i^\vee)}{\alpha_s(v - t\omega_i^\vee)}(sK(v - t\omega_i^\vee) + K(v - t\omega_i^\vee)s)||$$ $$\leq \sum_{s \in S} \frac{2|c_s|\alpha_s(\omega_i^\vee)}{\alpha_s(v - t\omega_i^\vee)}||K(v - t\omega_i^\vee)||,$$ where the norm $||\cdot||$ on $\End_\cplx(\lambda)$ is the norm arising from the natural inner product on $\End_\cplx(\lambda)$ associated to the inner product $(\cdot, \cdot)_\lambda$ on $\lambda$.  Note that this norm is differentiable away from $0 \in \End_\cplx(\lambda)$, the system of differential equations $K$ satisfies implies that if $K(x) = 0$ for any $x \in \ccal$ then $K$ is identically 0 on $\ccal$, and any $s \in S$ acts by an isometry with respect to this norm.  It follows that $$||K(v - t\omega_i^\vee)|| \leq ||K(v)||\prod_{s \in S}\left(\frac{\alpha_s(v)}{\alpha_s(v - t\omega_i^\vee)}\right)^{2|c_s|}.$$  With $\rho^\vee := \sum_i \omega_i^\vee \in \ccal$, we therefore have 
\begin{equation}
\label{eq:wall-growth}
||K(x)|| \leq ||K(\rho^\vee)||\prod_{s \in S} (\alpha_s(\rho^\vee)/\alpha_s(x))^{2|c_s|}
\end{equation} for all $x \in R$, and the claim follows immediately from this estimate and the homogeneity of $K$.\end{proof}

We need to impose further conditions, in addition to $W$-equivariance, on such a locally integrable function $K$ in order for it to represent the form $\gamma_{c, \lambda}$.  In \cite[Section 5]{Dunkl-B2} and \cite[Equation 9]{Dunkl-dihedral}, Dunkl derived certain conditions on such $K$ near the boundaries of the Weyl chambers.  We will now establish similar conditions in terms of the nature of the singularities of $K$ along the reflection hyperplanes, and we will relate these conditions to the invariance of sesquilinear pairings on certain representations of the Hecke algebra.

\begin{definition} For each simple reflection $s_i$, let $$\ccal_i := \{x \in \hfr_\real : \alpha_i(x) = 0, \alpha_j(x) > 0 \text{ for } j \neq i\}$$ be the open codimension-1 face of $\bar{\ccal}$ determined by $\alpha_i$.\end{definition}

Let $\lambda \in \irr(W)$ and let $c \in \pfr$ be such that $c_s \notin \frac{1}{2}\ints \backslash \{0\}$ for all $s \in S$.  Fix a simple reflection $s_i$, and choose coordinates $z_1, ..., z_l \in \hfr_\real^* \subset \hfr^*$ for $\hfr$ with $z_1 = \alpha_i$ and $z_j(\alpha_i^\vee) = 0$ for $j > 1$.   Consider the modified $KZ$ connection $$\nabla_{KZ}' := d - \sum_{s \in S} c_s\frac{d\alpha_s}{\alpha_s}s$$ on the trivial vector bundle on $\hfr_{reg}$ with fiber $\lambda$.  The connection $\nabla_{KZ}'$ is flat and therefore has a unique local extension of any initial value to a local homomorphic flat section.  Furthermore, considering the restriction of $\nabla_{KZ}'$ in the $z_1$ direction near a point $x_0 \in \ccal_i$ lying on the reflection hyperplane $\ker(\alpha_i)$, it follows from the standard theory of ordinary differential equations with regular singularities (see, e.g. \cite[Theorem 5.5]{W}) that there is a holomorphic function $P_i(z)$, defined in a complex analytic $z_1$-disc about $x_0$ and satisfying $P_i(x_0) = \id$, such that the function $$z \mapsto P_i(z)z_1^{c_is_i}$$ gives a fundamental $\End_\cplx(\lambda)$-valued (multivalued) solution to the restriction of the system $\nabla'_{KZ}$ to a small punctured $z_1$-disc about $x_0$.  We may extend $P_i(z)$ in the $z_2, ..., z_l$ directions by taking solutions of the restricted connection $$\nabla_{KZ}'' := d - \sum_{s \in S \backslash \{s_i\}} c_s\frac{d\alpha_s}{\alpha_s}s$$ along affine hyperplanes parallel to $\ker(\alpha_i)$ to produce a single valued function $P_i(z)$, holomorphic in $z_2, ..., z_l$, defined in a simply connected complex analytic $s_i$-stable neighborhood of $D$ of $\ccal \cup \ccal_i \cup s_i(\ccal)$ (note that the restricted connection above is regular along $\ccal_i$ itself).  The independence of the function $z_1^{c_is_i}$ on the variables $z_2, ..., z_l$ and the uniqueness of extensions of solutions of $\nabla_{KZ}'$ then implies that the such that the function $$N_i(z) := P_i(z)z_1^{c_is_i}$$ gives a fundamental $\End_\cplx(\lambda)$-valued (multivalued) solution to the system $\nabla_{KZ}'$ on the region $D_{reg} := D \cap \hfr_{reg}$.  In particular, $P_i(z)$ is holomorphic on $D_{reg}$.  Continuous dependence of solutions to ordinary differential equations on parameters and initial conditions implies that $P_i(z)$ is continuous on all of $D$.  In particular, viewing $P_i(z)$ as a function on $D_{reg}$ holomorphic in $z_1$, the singularities of $P_i(z)$ along $\ker(\alpha_i)$ are removable.  It follows that $P_i(z)$ is holomorphic in $z_1$ on all of $D$.  As $P_i(z)$ is holomorphic in each $z_j$ for $j > 1$, it follows by Hartogs' theorem that $P_i(z)$ is holomorphic on all of $D$.  

We next derive further properties of $P_i(z)$ arising from $W$-equivariance of the connection $\nabla_{KZ}'$ and the initial condition $P_i(x_0) = \id$.  In particular, we see that $s_iN_i(s_iz)s_i$ is another multivalued fundamental solution of $\nabla_{KZ}'$ on $D_{reg}$.  As $s_iN_i(s_iz)s_i = s_iP_i(s_iz)s_i(-z_1)^{c_is_i} = s_iP_i(s_iz)s_iz_1^{c_is_i}e^{\pi \sqrt{-1} c_is_i}$, it follows that $s_iP_i(s_iz)s_iz_1^{c_is_i}$ is such a solution as well (to reduce ambiguity, here we use the notation $\sqrt{-1}$ for the imaginary unit $i \in \cplx$ to avoid confusion with the index of the simple reflection $s_i$).  By the proof of \cite[Theorem 5.5]{W}, in particular its use of \cite[Theorem 4.1]{W}, and the assumption that $c_i \notin \frac{1}{2}\ints\backslash\{0\}$, any function $\widetilde{P}_i(z)$ holomorphic in a small complex analytic $z_1$-disc about $x_0$ such that $\widetilde{P}_i(x_0) = \id$ and the function $\widetilde{P}_i(z)z_1^{-c_is_i}$ is a solution to the system $\nabla_{KZ}'$ on the punctured $z_1$-disc must coincide with $P_i(z)$ along the entire disc.  It follows that $s_iP_i(s_iz)s_i = P_i(z)$ for $z$ in a small complex analytic $z_1$-disc about $x_0$ and therefore for all $z \in D$.

We summarize the conclusions of the above discussion in the following lemma:

\begin{lemma} \label{P-fn-lemma} Let $\lambda \in \irr(W)$, let $c \in \pfr$ be such that $c_i \notin \frac{1}{2}\ints\backslash\{0\}$ for a given simple reflection $s_i \in S$, and let $z_1, ..., z_l$ be coordinates for $\hfr$ as in the discussion above.  There is a $s_i$-stable simply connected complex analytic neighborhood $D$, independent of $c$, of $\ccal \cup \ccal_i \cup s_i(\ccal)$ in $\hfr$ and a holomorphic $GL(\lambda)$-valued function $P_i(z)$ on $D$ such that 

(1) $P_i(z)\alpha_i(z)^{c_is_i}$ is a $\End_\cplx(\lambda)$-valued multivalued holomorphic fundamental solution to the modified $KZ$ system $$\nabla_{KZ}' := d - \sum_{s \in S} c_s\frac{d\alpha_s}{\alpha_s}s$$ on the domain $D_{reg} := D \cap \hfr_{reg}$

(2) $s_iP_i(s_iz)s_i = P_i(z)$ for all $z \in D$.  In particular, $P_i$ takes values in $\End_{s_i}(\lambda)$ along $\ker(\alpha_i) \cap D$.

(3) $P_i(z)$ is a solution of the system $$\nabla_{KZ}'' := d - \sum_{s \in S \backslash\{s_i\}}c_s\frac{d\alpha_s}{\alpha_s}s$$ along $\ker(\alpha_i) \cap D$.

Furthermore, any such $P_i$ is determined by its value at any point $x_0 \in \ccal_i$, and for any constant invertible $A \in \aut_{s_i}(\lambda)$ the function $P_i(z)A$ is another such function.  In particular, any two such functions $P_i(z), \widetilde{P}_i(z)$ are related by an equality $\widetilde{P}_i(z) = P_i(z)A$ for a unique such $A$.\end{lemma}

Now, let $K : \ccal \rightarrow \End_\cplx(\lambda)$ be a real analytic function satisfying differential equation (\ref{eq:2-sided-KZ}) appearing in Proposition \ref{dunkl-weight-props} $$\partial_yK + \sum_{s \in S} c_s\frac{\alpha_s(y)}{\alpha_s}(sK + Ks) = 0 \ \ \ \ \text{ for all } y \in \hfr_\real.$$  Choose a simple reflection $s_i$, and fix functions $P_{i, c}(z)$ and $P_{i, c^\dagger}$ as in Lemma \ref{P-fn-lemma} for parameters $c$ and $c^\dagger$, respectively.  By the uniqueness of solutions of differential equation (\ref{eq:2-sided-KZ}) given initial conditions, there is a matrix $K_i \in \End_\cplx(\lambda)$, uniquely determined given $P_i(z)$, such that $K$ is given by $$z \mapsto P_{i, c^\dagger}(z)^{\dagger, -1}\alpha_i(z)^{-c_is_i}K_i\alpha_i(z)^{-c_is_i}P_{i, c}(z)^{-1}$$ for $z \in \ccal$.  As $P_{i, c}(z)$ and $P_{i, c^\dagger}(z)$ are determined only up to right multiplication by some $A \in \aut_{s_i}(\lambda)$, here $K_i$ is determined by $K$ only up to the action of $\aut_{s_i}(\lambda)$ on $\End_\cplx(\lambda)$ by $A.M = AMA^\dagger$.  As $s_i^\dagger = s_i$, $s_i$-invariance of $A$ implies that of $A^\dagger$, and we see that $s_i$-invariance of $K_i$ is a property of $K$, not depending on the choice of $P_{i, c}(z)$ and $P_{i, c^\dagger}(z)$.  In particular, we make the following definition:

\begin{definition} \label{asym-inv-def} In the setting of the previous paragraph, we say that the function $K$ is \emph{asymptotically $W$-invariant} if $K_i \in \End_{s_i}(\lambda)$ for all simple reflections $s_i$.\end{definition}

We will need the following lemma:

\begin{lemma}\label{mero-form-lemma}
Fix $x_0 \in \hfr_{\real, reg}$.  There is a holomorphic function $B : \pfr \rightarrow \End_\cplx(\lambda)$ such that $B(0) = \id$ and, for all $c \in \pfr$, $B(c)$ defines a $B_W$-invariant sesquilinear pairing $$KZ_{x_0}(\Delta_c(\lambda)) \times KZ_{x_0}(\Delta_{c^\dagger}(\lambda)) \rightarrow \cplx$$ $$(v_1, v_2) \mapsto v_2^\dagger B(c)v_1,$$ where we make the standard identification of $KZ_{x_0}(\Delta_c(\lambda))$ and $KZ_{x_0}(\Delta_{c^\dagger}(\lambda))$ with $\lambda$ as vector spaces.  Any two such functions $B(c)$ and $\widetilde{B}(c)$ are related by $\widetilde{B}(c) = b(c)B(c)$ for a meromorphic function $b(c)$.  Replacing $B(c)$ by $(B(c) + B(c^\dagger)^\dagger)/2$, we may assume $B(c)$ takes values in Hermitian forms for all $c \in \pfr_\real$.
\end{lemma}

\begin{proof}  An operator $A_c \in \End_\cplx(\lambda)$ defines a $B_W$-invariant sesquilinear pairing $$KZ_{x_0}(\Delta_c(\lambda)) \times KZ_{x_0}(\Delta_{c^\dagger}(\lambda)) \rightarrow \cplx$$ as in the lemma statement if and only if
\begin{equation}\label{eq:Ti-inv}
T_{i, c^\dagger}^{\dagger, -1}A_cT_{i, c}^{-1} = A_c,
\end{equation}
where $T_{i, c}$ and $T_{i, c^\dagger}$ denote the action of the generator $T_i \in B_W$ in $KZ_{x_0}(\Delta_c(\lambda))$ and $KZ_{x_0}(\Delta_{c^\dagger}(\lambda))$, respectively, for all generators $T_i$ of $B_W$.  Similarly to Lemma \ref{herm-dual-iso-lemma}, such a sequilinear pairing is equivalent to a homomorphism of $B_W$-representations
\begin{equation}\label{eq:KZ-BW-iso}
KZ_{x_0}(\Delta_{c^\dagger}(\lambda)) \rightarrow KZ_{x_0}(\Delta_c(\lambda))^h
\end{equation}
where $KZ_{x_0}(\Delta_c(\lambda))^h$ denotes the Hermitian dual of $KZ_{x_0}(\Delta_c(\lambda))$ as a $B_W$-representation, i.e. the representation of $B_W$ which is $\bar{KZ_{x_0}(\Delta_c(\lambda))^*}$ as a $\cplx$-vector space and in which the action of $g \in B_W$ is by $g.f = f \circ g^{-1}$.  When $|c_s|$ is small for all $s \in S$, the Hecke algebra $H_q(W)$, $q(s) = e^{-2 \pi i  c(s)}$, is semisimple and isomorphic to $\cplx W$, and the KZ functor is an equivalence of categories $\oscr_c(W, \hfr) \cong H_q(W)\mhyphen\text{mod}_{f.d.}$.  For such $c$, the irreducible representations of $H_q(W)$ therefore are precisely the images of the standard modules under the KZ functor, and in particular they can be distinguished by their character as $B_W$-representations, as this is so for $c = 0$ and these characters are continuous (in fact, holomorphic) functions of $c$.  In particular, for such $c$, as $KZ_{x_0}(\Delta_{c^\dagger}(\lambda))$ and $KZ_{x_0}(\Delta_c(\lambda))^h$ are irreducible, isomorphic for $c = 0$, and have characters as $B_W$-representations that are continuous functions of $c$, it follows that they are isomorphic.  It follows that, when $|c_s|$ is small for all $s \in S$, there is a 1-dimensional space of $B_W$-isomorphisms as in (\ref{eq:KZ-BW-iso}), and hence a unique solution $A_c \in \End_\cplx(\lambda)$ to the system of equations
\begin{equation}\label{eq:B-W-inv-equation}
\begin{cases}
T_{i, c^\dagger}^{\dagger, -1}A_cT_{i, c}^{-1} = A_c, \text{ for all generators } T_i \in B_W \\
\tr(A_c) = \dim(\lambda).
\end{cases}
\end{equation}
The operators $T_{i, c}^{-1}$ and $T_{i, c^\dagger}^{\dagger, -1}$ are holomorphic in $c \in \pfr$, so we may view system (\ref{eq:B-W-inv-equation}) as a system of linear system of equations in the variable $A_c \in \End_\cplx(\lambda)$ with coefficients that are holomorphic in $c \in \pfr$.  It follows that there is a solution $A_c$ that is meromorphic in $c \in \pfr$ with singularities where the system is degenerate.  Clearing denominators by multiplying by an appropriate determinant $d(c)$, holomorphic in $c$, the resulting holomorphic function $d(c)A_c$, up to scaling by a global constant, satisfies the properties of $B(c)$ stated in the lemma.  The uniqueness statement follows as well.  Finally, if $B(c)$ satisfies the properties in the Lemma then $\widetilde{B}(c) := (B(c) + B(c^\dagger)^\dagger)/2$ is also holomorphic in $c$ and satisfies equation (\ref{eq:Ti-inv}) for all $c \in \pfr_\real$, and therefore for all $c \in \pfr$ as well by analyticity.
\end{proof}

\begin{remark}\label{complex-herm-form-remark}
Lemma \ref{mero-form-lemma} holds for complex reflection groups as well, with the same proof, although this level of generality will not be needed.
\end{remark}

\begin{theorem}[Existence of Dunkl Weight Function for Small $c$] \label{existence-theorem} Let $\lambda \in \irr(W)$ and let $c \in \pfr$ with $|c_s|$ sufficiently small for all $s \in S$.  Let $K : \hfr_{\real, reg} \rightarrow \End_\cplx(\lambda)$ be a nonzero $W$-equivariant function satisfying system (\ref{eq:2-sided-KZ}).  Then the following are equivalent:

(a) $K$ represents the Gaussian pairing $\gamma_{c, \lambda}$ up to rescaling by a nonzero complex number.

(b) For any point $x \in \hfr_{\real, reg}$, $K(x)$ determines a $B_W$-invariant sesquilinear pairing $$KZ_x(\Delta_c(\lambda)) \times KZ_x(\Delta_{c^\dagger}(\lambda)) \rightarrow \cplx.$$

(c) $K|_\ccal$ is asymptotically $W$-invariant in the sense of Definition \ref{asym-inv-def}.

Furthermore, the space of such $K$ satisfying (a)-(c) forms a one-dimensional complex vector space.\end{theorem}

\begin{proof} We will first show that (a) and (c) are equivalent.  As the function $K$ is locally integrable and homogeneous by Lemma \ref{homogeneous-L1-lemma} when $|c_s|$ is small for all $s \in S$, it determines a tempered distribution with values in $\End_\cplx(\lambda)$.  Therefore, we may define a sesquilinear pairing $\gamma_K : \Delta_c(\lambda) \times \Delta_{c^\dagger}(\lambda) \rightarrow \cplx$ by the formula $$\gamma_K(P, Q) = \int_{\hfr_\real} Q(x)^\dagger K(x)P(x)e^{-|x|^2/2}dx$$ for all $P, Q \in \cplx[\hfr] \otimes \lambda.$  As $K$ is nonzero and $W$-equivariant, the same is true for $\gamma_K$ by the density of $\cplx[\hfr]e^{-|x|^2/2} \subset \mathscr{S}(\hfr_\real)$.  By Proposition \ref{dunkl-weight-characterization-proposition}, $\gamma_K$ is proportional to $\gamma_{c, \lambda}$ if and only if $$\int_{\hfr_\real} K(x)D_y\phi(x)dx = 0$$ for all $\phi \in \mathscr{S}(\hfr_\real)$ and $y \in \hfr_\real$.  Following Dunkl, for any $\epsilon > 0$ define the region $$\Omega_\epsilon := \{x \in \hfr_\real : |\alpha_s(x)| > \epsilon \text{ for all } s \in S\} \subset \hfr_{\real, reg}.$$  For any $\phi \in \mathscr{S}(\hfr_\real) \otimes \lambda$ the function $K(D_y\phi)$ is integrable, so we have $$\lim_{\epsilon \rightarrow 0} \int_{\Omega_e} K(x)D_y\phi(x)dx = \int_{\hfr_\real} K(x)D_y\phi(x)dx$$ by the dominated convergence theorem (see, e.g., \cite[Section 4.4]{RF}).  The region $\Omega_\epsilon$ avoids the singularities of $1/\alpha_s$ for all $s \in S$, and so, following the calculations appearing after Equation (14) in \cite[Section 5]{Dunkl-B2}, a direct calculation using the $W$-invariance of $K$, the $W$-invariance of $\Omega_\epsilon$, and the system of differential equations that $K$ satisfies shows that $$\int_{\Omega_\epsilon} K(x)D_y\phi(x)dx = \int_{\Omega_\epsilon} \partial_y(K(x)\phi(x))dx.$$  By Proposition \ref{dunkl-weight-characterization-proposition}, we see that $\gamma_K$ is proportional to $\gamma_{c, \lambda}$ if and only if 
\begin{equation}
\label{eq:necc-limit}
\lim_{\epsilon \rightarrow 0} \int_{\Omega_\epsilon} \partial_y(K(x)\phi(x))dx = 0
\end{equation} for a spanning set of $y \in \hfr_\real$ and dense set of $\phi \in \mathscr{S}(\hfr_\real)$.

Fix a simple reflection $s_i$.  Note that the boundary wall of the cone $\Omega_\epsilon \cap \ccal$ parallel to $\ccal_i$ is the translate of $\ccal_i$ by $\epsilon\rho^\vee$ and that the boundary wall of the cone $\Omega_\epsilon \cap s_i\ccal$ parallel to $\ccal_i$ is the translate of $\ccal_i$ by $\epsilon s_i\rho^\vee = \epsilon(\rho^\vee - \alpha_i^\vee)$.  Computing the integral above as an iterated integral with inner integrals chosen as in \cite[Section 5]{Dunkl-B2} and choosing test functions $\phi$ to be bump functions supported near a point $x_0 \in \ccal_i$ and $s_i$-invariant, we see that the vanishing of the limit $$\lim_{\epsilon \rightarrow 0} \int_{\Omega_\epsilon} \partial_{\alpha_i^\vee}(K\phi)dx$$ for all $\phi$ implies that $$\lim_{\epsilon \rightarrow 0} \int_{C_i} (K - s_iKs_i)(x + \epsilon \rho^\vee)\psi(x)dx = \lim_{\epsilon \rightarrow 0} \int_{C_i} (K(x + \epsilon \rho^\vee) - K(x + \epsilon s_i\rho^\vee))\psi(x) dx = 0$$ for all $\psi \in C^\infty_c(\ccal_i)$, where the first equality follows from the $W$-invariance of $K$.  Choosing coordinates $z_1, ..., z_l$ adapted to the wall $\ccal_i$ as in the discussion preceding Definition \ref{asym-inv-def}, we have that $K$ is given on $\ccal$ by the formula $$z \mapsto P_{i, c^\dagger}(z)^{\dagger, -1}\alpha_i(z)^{-c_is_i}K_i\alpha_i(z)^{-c_is_i}P_{i, c}(z)^{-1}$$ with $P_{i, c}(z)$, $P_{i, c^\dagger}$, and $K_i$ as above.  As $P_{i, c}$ and $P_{i, c^\dagger}$ take values in $\aut_{s_i}(\lambda)$ on $\ccal_i$, it follows that the commutators $[s_i, P_{i, c^\dagger}(z)^{\dagger, -1}]$ and $[P_{i, c}(z)^{-1}, s_i]$ are antiholomorphic and holomorphic, respectively, in $z$ and vanish along $\ccal_i$.  In particular, there is a holomorphic function $Q_{i, c}(z)$ and an antiholomorphic function $Q_{i, c^\dagger}(z)$ on $D$ with values in $\End_\cplx(\lambda)$ such that $[s_i, P_{i, c^\dagger}(z)^{\dagger, -1}] = \bar{z_1}Q_{i, c^\dagger}(z)$ and $[P_{i, c}(z)^{-1}, s_i] = z_1Q_{i, c}(z)$.  For $\sigma_1, \sigma_2 \in \{\pm 1\}$ let $K_i^{\sigma_1,\sigma_2} \in \End_\cplx(\lambda)$ be the projection of $K_i$ to the $(\sigma_1, \sigma_2)$ simultaneous eigenspace of the commuting operators $\lambda_{s_i}$ and $\rho_{s_i}$ on $\End_\cplx(\lambda)$ of left and right, respectively, multiplication by $s_i$.  We have $K_i = K_i^{1, 1} + K_i^{1, -1} + K_i^{-1, 1} + K_i^{-1, -1}$, and a straightforward computation shows
\begin{equation}
\label{eq:wall-diff} (K - s_iKs_i)(z) = 2P_{i, c^\dagger}(z)^{\dagger, -1}(K_i^{1, -1} + K_i^{-1, 1})P_{i, c}(z)^{-1} + z_1^{1 - 2|c_i|}R_i(z)
\end{equation}
for $z \in \ccal$, where $R_i(z) \in \End_\cplx(\lambda)$ is analytic on $\ccal$ and extends continuously to $\ccal \cup \ccal_i$.  It follows that we have $$\lim_{\epsilon \rightarrow 0} \int_{\ccal_i} (K - s_iKs_i)(x + \epsilon \rho^\vee)\psi(x)dx$$ $$=\int_{\ccal_i} 2P_{i, c^\dagger}(x + \epsilon\rho^\vee)^{\dagger, -1}(K_i^{1, -1} + K_i^{-1, 1})P_{i, c}(x + \epsilon \rho^\vee)^{-1}\psi(x)dx$$ for all $\psi \in C^\infty_c(\ccal_i)$.  For this integral to vanish for all such $\psi$, we must have that $K_i^{1, -1} + K_i^{-1, 1} = 0$, and hence that $K_i = K_i^{1, 1} + K_i^{-1, -1}$ commutes with $s_i$, so that $K|_\ccal$ is asymptotically $W$-invariant.  In particular, we see that (a) implies (c).

Conversely, if (c) holds, equation (\ref{eq:wall-diff}) implies that $K - s_iKs_i$ extends continuously to all of $\ccal \cup \ccal_i \cup s_i(\ccal)$ with value $0$ along $\ccal_i$.  By the $W$-equivariance of $K$ and the inner integrals used by Dunkl recalled in the previous paragraph, to show the vanishing of limit (\ref{eq:necc-limit}), for $y \in \hfr_{\real, reg}$ and $\phi \in C^\infty_c(\hfr_\real)$, it suffices to show that the limit
\begin{equation}
\label{eq:wall-integral}
\lim_{\epsilon \rightarrow 0} \int_{\ccal_i} (K\phi)(x + \epsilon\rho^\vee) - (K\phi)(s_i(x + \epsilon\rho^\vee))dx
\end{equation} vanishes.  It suffices to treat the cases in which $\phi \circ s_i = \pm \phi$.  In the case $\phi \circ s_i = \phi$, as $K$ is $W$-invariant the integrand of (\ref{eq:wall-integral}) is $(K - s_iKs_i)(x + \epsilon\rho^\vee)\phi(x + \epsilon\rho^\vee)$.  In the case $\phi \circ s_i = -\phi$, there exists a test function $\psi \in C^\infty_c(\hfr_\real)$ such that $\phi = \alpha_i\psi$, and the integrand of $(\ref{eq:wall-integral})$ takes the form $(\alpha_i\cdot(K + s_iKs_i))(x + \epsilon\rho^\vee)\psi(x + \epsilon\rho^\vee).$  In particular, in either case the integrand is of the form $f(x + \epsilon\rho^\vee)\psi(x + \epsilon\rho^\vee)$ where $\psi \in C^\infty_c(\hfr_\real)$ is a test function and $f$ is either $K - s_iKs_i$ or $\alpha_i\cdot(K + s_iKs_i)$.  In either case, from estimate (\ref{eq:wall-growth}) in Lemma \ref{homogeneous-L1-lemma} and equation (\ref{eq:wall-diff}), there is a constant $\mu > 0$, independent of $c$, such that $$\delta_i(x)^{\mu\sum_{s \in S} |c_s|}\alpha_i(x)^{2|c_i| - 1}f(x)$$ extends to a continuous function on $\bar{\ccal}$.  In particular, let $M \geq 0$ be $$M := \max_{x \in \supp(\psi) \cap \bar{\ccal}} ||\delta_i(x)^{\mu\sum_{s \in S} |c_s|}\alpha_i(x)^{2|c_i| - 1}f(x)||$$ and let $M' := \max_{x \in \hfr_\real} |\psi(x)|$.  We then have $$||f(x)|| \leq M \delta_i(x)^{-\mu\sum_{s \in S} |c_s|}\alpha_i(x)^{1 - 2|c_i|}$$ for $x \in \supp(\psi) \cap \bar{\ccal}$, and hence $$||f(x + \epsilon \rho^\vee)\psi(x + \epsilon\rho^\vee)||$$ $$\leq MM' \delta_i(x + \epsilon\rho^\vee)^{-\mu\sum_{s \in S} |c_s|}\alpha_i(x + \epsilon\rho^\vee)^{1 - 2|c_i|}$$ $$\leq MM' \delta_i(x)^{-\mu\sum_{s\in S} |c_s|}\epsilon^{1 - 2|c_i|}$$ for all $x \in \ccal_i \cap \supp(\psi)$.  Taking $|c_s|$ sufficiently small for all $s \in S$ so that $\delta_i(x)^{-\mu\sum_{s \in S} |c_s|}$ is locally integrable on $\ccal_i$ and $1 - 2|c_i| > 0$, the vanishing of limit (\ref{eq:wall-integral}) now follows from the dominated convergence theorem.  This completes the proof that (c) implies (a).

Now we will show the equivalence of (b) and (c).  Fix a point $x_0 \in \hfr_{\real, reg}$ and identify $K(x_0)$ with a sesquilinear pairing $KZ_{x_0}(\Delta_c(\lambda)) \times KZ_{x_0}(\Delta_{c^\dagger}(\lambda)) \rightarrow \cplx$.  Here we make the usual identification of $KZ_{x_0}(\Delta_c(\lambda))$ and $KZ_{x_0}(\Delta_{c^\dagger}(\lambda))$ with $\lambda$ as $\cplx$-vector spaces, and $K(x_0) \in \End_\cplx(\lambda)$ defines a sesquilinear form by $(v_1, v_2)_{K(x_0)} = v_2^\dagger K(x_0)v_1$.  Without loss of generality, we take $x \in \ccal$.  The braid group $B_W = \pi_1(\hfr_{reg}/W, x_0)$ is generated by the elements $T_i$ given by positively oriented half loops around the hyperplanes $\ker(\alpha_i)$ for simple roots $\alpha_i$.  For each simple root $\alpha_i$, let $T_{i, c}$ and $T_{i, c^\dagger}$ denote the action of $T_i$ on $KZ_{x_0}(\Delta_c(\lambda))$ and on $KZ_{x_0}(\Delta_{c^\dagger}(\lambda))$, respectively, so that the action of $T_i$ on the space of sesquilinear pairings $KZ_{x_0}(\Delta_c(\lambda)) \times KZ_{x_0}(\Delta_{c^\dagger}(\lambda)) \rightarrow \cplx$, identified with $\End_\cplx(\lambda)$, is given by
\begin{equation}
\label{eq:BW-form-action}
X \mapsto T_{i, c^\dagger}^{\dagger, -1}XT_{i, c}^{-1}.
\end{equation}

Consider the $W$-equivariant $\End_\cplx(\lambda)$-valued multivalued function $$\widetilde{K}(z) := (f_{c^\dagger}(z)M_{c^\dagger}(z))^{\dagger, -1}K(x_0)(f_c(z)M_c(z))^{-1}$$ on $\hfr_{reg}$, where  $M_c(z)$ is the monodromy of the $KZ$ connection $$\nabla_{KZ} := d + \sum_{s \in S} c_s\frac{d\alpha_s}{\alpha_s}(1 - s)$$ on the trivial bundle $\lambda \times \hfr_{reg} \rightarrow \hfr_{reg}$ from $x_0$ to $z$, $f_c(z)$ is the (scalar) monodromy of the scalar-valued $W$-equivariant flat connection $$d - \sum_{s \in S} c_s\frac{d\alpha_s}{\alpha_s}$$ on the trivial bundle $\cplx \times \hfr_{reg} \rightarrow \hfr_{reg}$ from $x_0$ to $z$, and $M_{c^\dagger}$ and $f_{c^\dagger}$ are defined similarly with the parameter $c^\dagger$ in place of the parameter $c$.  A straightforward computation shows that $\widetilde{K}(z)$ satisfies system (\ref{eq:2-sided-KZ}) of differential equations on $\ccal$, so it follows from the equality $\widetilde{K}(x_0) = K(x_0)$ and the uniqueness of solutions to (\ref{eq:2-sided-KZ}) given initial conditions that $\widetilde{K}(x) = K(x)$ for all $x \in \ccal$.

View $\widetilde{K}(z)$ as a multivalued section of the vector bundle $$(\End_\cplx(\lambda) \times \hfr_{reg})/W \rightarrow \hfr_{reg}/W.$$  It follows by considering residues that the monodromy of $f_c(z)$ is inverse to the monodromy of $\bar{f_{c^\dagger}(z)}$, and in particular the function $f_c(z)\bar{f_{c^\dagger}(z)}$ is single valued on $\hfr_{reg}/W$.  As the monodromy of $M_c(z)$ based at $x_0$ is given, by definition, by the $B_W$-action on $KZ_{x_0}(\Delta_c(\lambda))$ and similarly for $M_{c^\dagger}(z)$, it follows that the monodromy in $\hfr_{reg}/W$ of $\widetilde{K}(z)$ based at $x_0$ is precisely given by the $B_W$-action on sesquilinear pairings appearing in (\ref{eq:BW-form-action}) above.  In particular, $K(x_0)$ is $B_W$-invariant as a sesquilinear pairing $KZ_{x_0}(\Delta_c(\lambda)) \times KZ_{x_0}(\Delta_{c^\dagger}(\lambda)) \rightarrow \cplx$ if and only if $\widetilde{K}(z)$ is single-valued on $\hfr_{reg}/W$.

It remains to show that $\widetilde{K}(z)$ is single-valued if and only if $K(x)$ is asymptotically $W$-invariant in the sense of Definition \ref{asym-inv-def}.  Fix a simple reflection $\alpha_i$, and let $P_{i, c}(z)$ and $P_{i, c^\dagger}(z)$ be holomorphic functions on a domain $D$ as in Lemma \ref{P-fn-lemma}.  It follows from that same lemma and the remarks following it that there is a unique $K_i \in \End_\cplx(\lambda)$ such that $$\widetilde{K}(z) = P_{i, c^\dagger}(z)^{\dagger, -1}\bar{\alpha_i(z)}^{-c_is_i}K_i\alpha_i(z)^{-c_is_i}P_{i, c}(z)^{-1}$$ for all $z \in D \cap \hfr_{reg}$.  Decomposing $K_i$ into simultaneous left and right eigenspaces for $s_i$ as $K_i = K_i^{1,1} + K_i^{1, -1} + K_i^{-1, 1} + K_i^{-1, -1}$ as before, we have that $P_{i, c^\dagger}(z)^\dagger \widetilde{K}(z)P_{i, c}(z)$, for $z \in D \cap \hfr_{reg}$, is given by: $$|\alpha_i(z)|^{-2c_i}K_i^{1, 1} + \left(\frac{\bar{\alpha_i(z)}}{\alpha_i(z)}\right)^{-c_i}K_i^{1, -1} + \left(\frac{\bar{\alpha_i(z)}}{\alpha_i(z)}\right)^{c_i}K_i^{-1, 1} + |\alpha_i(z)|^{2c_i}K_i^{-1, -1}.$$  Clearly, this function, and hence $\widetilde{K}(z)$ itself, is single-valued in $\hfr_{reg}$ near the hyperplane $\ker(\alpha_i)$ if and only if $K_i^{1, -1} = K_i^{-1, 1} = 0$, i.e. if $K_i = K_i^{1, 1} + K_i^{-1, -1}$ is $s_i$-invariant.  In that case, by the $s_i$-equivariance of $P_{i, c}$ and $P_{i, c^\dagger}$, it then follows that $\widetilde{K}(z)$ is also single-valued in $\hfr_{reg}/W$ near the hyperplane $\ker(\alpha_i)$.  So, $\widetilde{K}(z)$ is invariant under the monodromy action of $T_i \in B_W$ if and only if $K_i \in \aut_{s_i}(\lambda)$.  In particular, $\widetilde{K}(z)$ is single-valued if and only if $K$ is asymptotically $W$-invariant, completing the proof of the equivalence of (b) and (c).

The final statement of the theorem follows from Lemma \ref{mero-form-lemma}.
\end{proof}

\subsection{Analytic Continuation on the Regular Locus}\label{extension-to-reg-locus-section}

\begin{theorem}[Existence of Dunkl Weight Function on $\hfr_{\real, reg}$] \label{integral-rep-theorem} There exists a unique family of analytic functions $K_c : \hfr_{\real, reg} \rightarrow \End_\cplx(\lambda)$ with holomorphic dependence on $c$ such that the following holds: for all $M > 0$ there exists an integer $N \geq 0$, which may be taken to be $0$ for $M$ sufficiently small, such that for $c \in \pfr$ with $|c_s| < M$ for all $s \in S$ the function $\delta^{2N}K_c$ is locally integrable on $\hfr_\real$ and represents the Gaussian pairing $\gamma_c$ on $\delta^N\Delta_c(\lambda) \times \delta^N\Delta_{c^\dagger}(\lambda)$ in the following sense: $$\gamma_c(\delta^NP, \delta^NQ) = \int_{\hfr_\real} Q(x)^\dagger\delta^{2N}(x)K_c(x)P(x)e^{-|x|^2/2}dx \ \ \ \ \text{ for all } P, Q \in \cplx[\hfr] \otimes \lambda.$$  For each $c \in \pfr$ and $x \in \hfr_{\real, reg}$, $K_c(x)$ defines a $B_W$-invariant sesquilinear pairing $$KZ_x(\Delta_c(\lambda)) \times KZ_x(\Delta_{c^\dagger}(\lambda)) \rightarrow \cplx.$$  For $c \in \pfr_\real$, $K_c(x)$ defines a $B_W$-invariant Hermitian form on $KZ_x(\Delta_c(\lambda))$.
\end{theorem}
 
\begin{proof}  Uniqueness of $K_c$ follows from the density of $\cplx[\hfr]e^{-|x|^2/2}$ in the Schwartz space $\mathscr{S}(\hfr_\real)$, so it suffices to establish existence of $K_c$.  Let $B : \pfr \rightarrow \End_\cplx(\lambda)$ be a holomorphic function as in Lemma \ref{mero-form-lemma} defined relative to the point $x_0 \in \ccal$.  It follows from the holomorphic dependence on the parameter $c$ and on initial conditions of solutions of the system of differential equations (\ref{eq:2-sided-KZ}) appearing in Proposition \ref{dunkl-weight-props} that there is a unique function $K' : \pfr \times \hfr_{\real, reg} \rightarrow \End_\cplx(\lambda)$, $K'(c, x) = K'_c(x)$, holomorphic in $c \in \pfr$ and analytic in $x \in \hfr_{\real, reg}$, such that

(1) $K'_c(x_0) = B(c)$ for all $c \in \pfr$

(2) $wK'_c(w^{-1}x)w^{-1} = K'_c(x)$ for all $c \in \pfr$, $x \in \hfr_{\real, reg}$, and $w \in W$

(3) For all fixed $c \in \pfr$, $K'_c(x)$, as a function of $x \in \ccal$ satisfies the system of differential equations (\ref{eq:2-sided-KZ}).

\noindent Furthermore, as $B(c)$ defines a $B_W$-invariant sesquilinear pairing $$KZ_{x_0}(\Delta_c(\lambda)) \times KZ_{x_0}(\Delta_{c^\dagger}(\lambda)) \rightarrow \cplx,$$ it follows from the proof of the equivalence of statements (b) and (c) in Theorem \ref{existence-theorem} that $K'_c(x)$ defines a $B_W$-invariant sesquilinear pairing $$KZ_x(\Delta_c(\lambda)) \times KZ_x(\Delta_{c^\dagger}(\lambda)) \rightarrow \cplx$$ for all $c \in \pfr$ and $x \in \hfr_{\real, reg}.$  Note that $K_0'(x) = \id$ for all $x \in \hfr_{\real, reg}$.

 Let $M > 0$.  By Lemma \ref{homogeneous-L1-lemma}, there exists an integer $N \geq 0$, which may be taken to be $0$ if $M$ is sufficiently small, such that for all $c$ with $|c_s| < M$ for all $s \in S$ the function $\delta^{2N}K_c'$ is locally integrable on $\hfr_\real$.  As $\delta^{2N}K_c'$ is homogeneous by the same lemma, it follows that integration against $\delta^{2N}K_c'$ determines a tempered distribution on $\hfr_\real$.  In particular, we may define a sesquilinear pairing $\widetilde{\gamma}_c$ on $\delta^N\Delta_c(\lambda) \times \delta^N\Delta_{c^\dagger}(\lambda)$ as follows: $$\widetilde{\gamma}_c(\delta^NP, \delta^NQ) := \int_{\hfr_\real} Q(x)^\dagger \delta^{2N}(x)K'_c(x)P(x)e^{-|x|^2/2}dx \ \ \ \ \text{ for all } P, Q \in \cplx[\hfr] \otimes \lambda.$$  Let $U_M \subset \pfr$ be the open subset $U_M := \{c \in \pfr : |c_s| < M \text{ for all } s \in S\}$.  For any $P, Q \in \cplx[\hfr] \otimes \lambda$, we see that $\widetilde{\gamma}_c(\delta^NP, \delta^NQ)$ is a holomorphic function of $c \in U_M$, and the same is true for $\gamma_c(\delta^NP, \delta^NQ)$.  In particular, for any such $P$ and $Q$ such that $\widetilde{\gamma}_c(\delta^NP, \delta^NQ)$ is not identically zero for $c \in U_M$, we may define the meromorphic function $f_{P, Q} : U_M \rightarrow \cplx$ by $$f_{P, Q}(c) := \gamma_c(\delta^NP, \delta^NQ)/\widetilde{\gamma}_c(\delta^NP, \delta^NQ).$$  By Theorem \ref{existence-theorem} and the properties of $K'_c$ listed in the previous paragraph, we see that, for those $c \in \pfr$ with $|c_s|$ sufficiently small for all $s \in S$, the pairings $\gamma_c$ and $\widetilde{\gamma}_c$ agree up to a nonzero complex scalar multiple.  It follows that all of the functions $f_{P, Q}$ for $P, Q \in \cplx[\hfr] \otimes \lambda$ coincide in a neighborhood of $0$ in $\pfr$ and hence coincide with a single meromorphic function $f : U_M \rightarrow \cplx$.  Let $K_c : \hfr_{\real, reg} \rightarrow \End_\cplx(\lambda)$ be defined by $K_c(x) = f(c)K'_c(x)$.  Then $K_c(x)$ is meromorphic in $c \in U_M$, holomorphic in $c$ in a neighborhood of $0$ in $\pfr$, satisfies the remaining properties of $K'_c(x)$ from the previous paragraph, and $$\gamma_c(\delta^NP, \delta^NQ) = \int_{\hfr_\real} Q(x)^\dagger \delta^{2N}(x)K_c(x)P(x)e^{-|x|^2/2}dx \ \ \ \ \text{ for all } P, Q \in \cplx[\hfr] \otimes \lambda$$ for all $c \in U_M$ for which $K_c(x)$ is holomorphic.
 
 It remains to show that $K_c(x)$ is in fact holomorphic for all $c \in U_M$.  To this end, it suffices to show that $K_c(x)$ has no singularities for $c$ in $U_M \cap \lfr$ for any complex line $\lfr$ in $\pfr$.  By the Weierstrass theorem on the existence of meromorphic functions with prescribed zeros and poles, there exists a meromorphic function $g$ on $U_M \cap \lfr$ such that $K_c(x_0) = g(c)K_c''(x_0)$, where $K_c''(x_0) \in \End_\cplx(\lambda)$ is nonzero for all $c \in U_M \cap \lfr$.  As above, we may extend $K_c''(x_0)$, via $W$-equivariance and the system of differential equations (\ref{eq:2-sided-KZ}), to a non-vanishing function $K'' : (U_M \cap \lfr) \times \hfr_{\real, reg} \rightarrow \End_\cplx(\lambda)$, holomorphic in $c \in U_M \times \lfr$ and analytic in $x \in \hfr_{\real, reg}$, so that $K_c(x) = g(c)K_c''(x)$ for all $c \in U_M \cap \lfr$ and $x \in \hfr_{\real, reg}$.  As above, we have
 \begin{equation}\label{eq:show-no-sings}
 \gamma_c(\delta^NP, \delta^NQ) = g(c)\int_{\hfr_\real} Q(x)^\dagger \delta^{2N}(x)K''_c(x)P(x)e^{-|x|^2/2}dx \ \ \ \ \text{ for all } P, Q \in \cplx[\hfr] \otimes \lambda
 \end{equation}
 for all $c \in U_M \cap \lfr$ for which $g(c)$ is holomorphic.  But by the density of $\cplx[\hfr]e^{-|x|^2/2}$ in the Schwartz space $\mathscr{S}(\hfr_\real)$ and the non-vanishing of $\delta^{2N}(x)K''_c(x)$ for $x \in \hfr_{\real, reg}$, it follows that for any $c \in U_M \cap \lfr$ there are $P, Q \in \cplx[\hfr] \otimes \lambda$ such that the integral on the righthand side of equation (\ref{eq:show-no-sings}) is nonzero.  It follows that $g(c)$ is holomorphic for all $c \in U_M \cap \lfr$, as needed.\end{proof}
 
 \begin{corollary}\label{unitarity-corollary} a)  Let $q : S \rightarrow \cplx^\times$ be a $W$-invariant function satisfying $|q_s| = 1$ for all $s \in S$.  Then every irreducible representation of the Hecke algebra $H_q(W)$ admits a nondegenerate $B_W$-invariant Hermitian form, unique up to $\real^\times$-scaling.
 
 b)  Let $c \in \pfr_\real$ and let $x_0 \in \hfr_{\real, reg}$.  If $\lambda \in \irr(W)$ is such that $KZ_{x_0}(L_c(\lambda))$ is nonzero and unitary, i.e. it admits a positive-definite $B_W$-invariant Hermitian form, then $L_c(\lambda)$ is quasi-unitary in the sense of Definition \ref{def:quasi-unitary}.\end{corollary}
 
\begin{remark}
Chlouveraki-Gordon-Griffeth \cite[Section 4]{CGG} have considered certain invariant symmetric forms on representations of the Hecke algebra $H_q(W)$ for arbitrary $q$.
\end{remark}

\begin{proof}[Proof of Corollary \ref{unitarity-corollary}] Any $q$ as in (a) is of the form $q_s = e^{-2 \pi i c_s}$ for some $c \in \pfr_\real$.  Fix such $c \in \pfr_\real$ and choose a point $x_0 \in \hfr_{\real, reg}$.  Any irreducible representation of $H_q(W)$ is isomorphic to some $KZ_{x_0}(L_c(\lambda))$ for some $\lambda \in \irr(W)$ such that $\supp(L_c(\lambda)) = \hfr$.  Let $N_c(\lambda)$ be the maximal proper submodule of $\Delta_c(\lambda)$, so that $L_c(\lambda) = \Delta_c(\lambda)/N_c(\lambda)$.  Recall from Proposition \ref{beta-props}(iv) that $N_c(\lambda)$ is the radical of $\gamma_c$.
 
Let $K_c : \hfr_{\real, reg} \rightarrow \End_\cplx(\lambda)$ be as in Theorem \ref{integral-rep-theorem}, so that in particular $K_c(x_0)$ defines a $B_W$-invariant Hermitian form on $\lambda = _{v.s.} KZ_{x_0}(\Delta_c(\lambda))$.  It follows from the system of differential equations (\ref{eq:2-sided-KZ}) that if $K_c(x_0) = 0$ then $K_c(x) = 0$ for all $x \in \hfr_{\real, reg}$.  In this case, by Theorem \ref{integral-rep-theorem}, for sufficiently large $N > 0$, we have that the restriction of $\gamma_c$ to $\delta^N\Delta_c(\lambda)$ is zero.  As $\delta \in \cplx[\hfr]$ is self-adjoint with respect to the form $\gamma_c$, it follows that $\delta^{2N}\Delta_c(\lambda) \subset N_c(\lambda)$, so $\delta^{2N}L_c(\lambda) = 0$, contradicting $\supp(L_c(\lambda)) = \hfr$.  In particular, $K_c(x_0)$ defines a nonzero $B_W$-invariant Hermitian form on $KZ_{x_0}(\Delta_c(\lambda))$.
 
As $KZ_{x_0}$ is exact, we have $KZ_{x_0}(L_c(\lambda)) \cong KZ_{x_0}(\Delta_c(\lambda))/KZ_{x_0}(N_c(\lambda))$.  As $KZ_{x_0}(L_c(\lambda))$ is irreducible, to show that it admits a nondegenerate $B_W$-invariant Hermitian form, unique up to $\real^\times$ scaling, it suffices to show that $KZ_{x_0}(N_c(\lambda))$ lies in the radical of $K_c(x_0)$.  To this end, take a vector $v \in KZ_{x_0}(N_c(\lambda))$.  Viewing $KZ_{x_0}(N_c(\lambda))$ as a submodule of $KZ_{x_0}(\Delta_c(\lambda))$ and identifying the latter with $\lambda$ as a vector space, let $\widetilde{v} \in N_c(\lambda) \subset \Delta_c(\lambda) = \cplx[\hfr] \otimes \lambda$ be such that its value at $x_0$ is $v$.  In particular, $\delta^N\widetilde{v}$ is in the radical of $\gamma_c$, and it follows from Theorem \ref{integral-rep-theorem} that $$0 = \gamma_c(\delta^N\widetilde{v}, \delta^NQ) = \int_{\hfr_\real} Q(x)^\dagger \delta^{2N}(x)K_c(x)\widetilde{v}e^{-|x|^2/2}dx \ \ \ \ \text{ for all } Q \in \cplx[\hfr] \otimes \lambda.$$  As $\delta^{2N}K_c\widetilde{v}$ defines a tempered distribution on $\hfr_\real$ with values in $\lambda$, the density of $\cplx[\hfr]e^{-|x|^2/2}$ in the Schwartz space $\mathscr{S}(\hfr_\real)$ implies that $\delta^{2N}K_c\widetilde{v} = 0$ pointwise on $\hfr_{\real, reg}$.  In particular, $K_c(x_0)v = 0$.  As $v \in KZ_{x_0}(N_c(\lambda))$ was arbitrary, it follows that $KZ_{x_0}(N_c(\lambda)) \subset \rad(K_c(x_0))$, where $\rad(K_c(x_0))$ denotes the radical of the form $K_c(x_0)$.  In particular, $K_c(x_0)$ descends to a nondegenerate $B_W$-invariant Hermitian form on $KZ_{x_0}(\Delta_c(\lambda))$, proving (a).  Note that this also shows that in fact $KZ_{x_0}(N_c(\lambda)) = \rad(K_c(x_0)).$

For (b), if $KZ_{x_0}(L_c(\lambda))$ is unitary, it follows from the argument above that $K_c(x_0)$, appropriately scaled, is positive-semidefinite and descends to a positive-definite Hermitian form on $KZ_{x_0}(L_c(\lambda))$.  The same is then true for all points $x \in \hfr_{\real, reg}$, and the integral formula in Theorem \ref{integral-rep-theorem} then implies that the restriction of $\gamma_c$ to $\delta^NL_c(\lambda)$ is positive-definite.  As $L_c(\lambda)$ has full support in $\hfr$ and as $L_c(\lambda)/\delta^NL_c(\lambda)$ is supported on the reflection hyperplanes, it follows that the Hilbert polynomial of $L_c(\lambda)$ is of strictly higher degree than the Hilbert polynomial of $L_c(\lambda)/\delta^NL_c(\lambda)$.  In particular, the Hilbert polynomial of $\delta^NL_c(\lambda)$ has the same leading term as the Hilbert polynomial of $L_c(\lambda)$ itself, so we have $$\lim_{n \rightarrow \infty} \frac{\dim (\delta^NL_c(\lambda))^{\leq n}}{\dim L_c(\lambda)^{\leq n}} = 1.$$  As $\gamma_c$ is positive-definite on $\delta^NL_c(\lambda)$, we see that $L_c(\lambda)$ is quasi-unitary, as needed.\end{proof}

We gave the proof of Corollary \ref{unitarity-corollary}(a) as above to demonstrate a use of the Dunkl weight function and because the arguments appearing in that proof would be repeated later for the proof of Theorem \ref{KZ-preserves-signatures-theorem}.  However, there is an alternate proof of Corollary \ref{unitarity-corollary}(a) that is also valid for finite \emph{complex} reflection groups, due to Rapha\"el Rouquier and communicated to me by Pavel Etingof, as follows.

\begin{proposition} Corollary \ref{unitarity-corollary}(a) is valid for finite complex reflection groups $W$.\end{proposition}

\begin{proof} Let $c \in \pfr_\real$ be a real parameter corresponding to $q$ via the KZ functor.  Let $L$ be an irreducible representation of $H_q(W)$, and let $\lambda$ be an irreducible representation of $W$ such that $L_c(\lambda)$ has full support and $L \cong KZ(L_c(\lambda))$.  To show that $L$ admits a nondegenerate $B_W$-invariant Hermitian form, necessarily uniquely determined up to $\real^\times$-multiple, it suffices to show that $L$ is isomorphic to its Hermitian dual $L^h$ as a $B_W$-representation.  By Remark \ref{complex-herm-form-remark}, there is a holomorphic function $B : \pfr \rightarrow \End_\cplx(\lambda)$ satisfying the conditions of Lemma \ref{mero-form-lemma}.  Restricting to the complex line in $\pfr$ spanned by $c$ and scaling $B$ by an appropriate meromorphic function, we may assume that $B(c)$ is nonzero and therefore determines a nonzero $B_W$-invariant Hermitian form on $KZ(\Delta_c(\lambda))$, which we will regard as a nonzero map $\beta : KZ(\Delta_c(\lambda)) \rightarrow KZ(\Delta_c(\lambda))^h$.  As $\beta$ defines a Hermitian form, it factors through an isomorphism $\bar{\beta} : KZ(\Delta_c(\lambda))/\ker \beta \cong (KZ(\Delta_c(\lambda))/\ker \beta)^h$.  The head of the module $KZ(\Delta_c(\lambda))$ is $KZ(L_c(\lambda))$; this follows from the fact that $L_c(\lambda)$ has full support and is the head of $\Delta_c(\lambda)$ and the fact that the KZ functor admits a right adjoint $\pi^* : H_q(W)\mhyphen\text{mod}_{f.d.} \rightarrow \oscr_c(W, \hfr)$ such that $KZ \circ \pi^* \cong \id_{H_q(W)\mhyphen\text{mod}_{f.d.}}$.  In particular, $KZ(L_c(\lambda))$ is also the head of $KZ(\Delta_c(\lambda))/\ker\beta$, and it appears with multiplicity 1 as a composition factor because the same is true for $L_c(\lambda)$ in $\Delta_c(\lambda)$.  If $KZ(\Delta_c(\lambda))/\ker\beta$ is irreducible, then it isomorphic to $KZ(L_c(\lambda))$, and the proof is complete.  Otherwise, let $S \subset KZ(\Delta_c(\lambda))/\ker \beta$ be a simple submodule, necessarily not isomorphic to $KZ(L_c(\lambda))$ because the latter appears with multiplicity 1.  Taking the Hermitian dual, we obtain a surjection $(KZ(\Delta_c(\lambda))/\ker\beta)^h \rightarrow S^h$, and composing with $\bar{\beta}$ we conclude that $S^h$ appears in the head of $KZ(\Delta_c(\lambda))/\ker\beta$, so $S^h \cong KZ(L_c(\lambda))$.  But $S$ must be isomorphic to $KZ(L_c(\mu))$ for some $\mu \in \irr(W)$, $\mu \neq \lambda$, such that $L_c(\mu)$ has full support and appears as a composition factor in $\Delta_c(\lambda)$.  By induction on the highest weight order $\leq_c$ of $\oscr_c(W, \hfr)$, we may assume that $S \cong KZ(L_c(\mu))$ admits a nondegenerate $B_W$-invariant Hermitian form, so that $S \cong S^h$.  But then $S \cong KZ(L_c(\lambda))$, a contradiction.  It follows that $KZ(\Delta_c(\lambda))/\ker \beta$ is irreducible and isomorphic to $KZ(L_c(\lambda))$, and the claim follows.\end{proof}

\begin{remark} In Section \ref{signature-comparision-section}, using techniques from semiclassical analysis, we will substantially generalize part (b) of Corollary \ref{unitarity-corollary}.  In particular, we will see that, whenever $L_c(\lambda)$ has full support, the asymptotic signature $a_{c, \lambda}$ of $L_c(\lambda)$ equals, up to sign, the signature of an invariant Hermitian form on $KZ_{x_0}(L_c(\lambda))$ normalized by its dimension.  In particular, unitarity of $KZ_{x_0}(L_c(\lambda))$ does not only imply quasi-unitarity of $L_c(\lambda)$, as in part (b) of Corollary \ref{unitarity-corollary}, but is equivalent to it.\end{remark}
 
\subsection{Extension to Tempered Distribution via the Wonderful Model}\label{extension-to-hyperplanes-section}

In this section we will complete the proof of the existence of the Dunkl weight function, i.e. of Theorem \ref{main-existence-theorem}.  As we will see, essentially what remains is to extend the function $K_c : \hfr_{\real, reg} \rightarrow \End_\cplx(\lambda)$ from Theorem \ref{integral-rep-theorem} to a tempered distribution on $\hfr_\real$ in a natural way.  To achieve this we will use a known approach (see, e.g., \cite{BBK}) in which the extension is carried out on the De Concini-Procesi wonderful model \cite{DP1}; the desired extension is then obtained by pushing forward to $\hfr_\real$.  The advantage of working in the wonderful model rather than in $\hfr_\real$ is that the hyperplane arrangement is replaced by a normal crossings divisor, allowing the application of standard fact that for any $\lambda \in \cplx$ the function $|x|^{z\lambda}$, locally integrable for $|z|$ small, has a natural distributional meromorphic continuation to $z \in \cplx$.

For $R > 0$, let $B_R(0)$ denote the open disk $B_R(0) := \{z \in \cplx : |z| < R\}$ of radius $R$ in the complex plane centered at 0.  For any integer $N > 0$ let $Mat_{N \times N}(\cplx)$ denote the space of $N \times N$ complex matrices, and for any $a \in Mat_{N \times N}(\cplx)$ let $\spec(A)$ denote the set of eigenvalues of $a$.

\begin{lemma}\label{1D-DE-lemma} Let $R > 0$, let $a : \pfr \rightarrow Mat_{N \times N}(\cplx)$ be a linear function, and let $A(z; c)$ be a holomorphic function of $(c, z) \in \pfr \times B_R(0)$, taking values in the space $Mat_{N \times N}(\cplx)$ of $N \times N$ complex matrices, such that $A(0; c) = a(c)$ for all $c \in \pfr$.  Then for any bounded domain $U \subset \pfr$ there is a $Mat_{N \times N}(\cplx)$-valued holomorphic function $Q(z; c)$ of $(c, z) \in U \times B_R(0)$ such that $$F(z; c) = Q(z; c)z^{a(c)}$$ is a solution of the differential equation 
\begin{equation}\label{eq:1D-DE}
z\frac{dF(z; c)}{dz} = A(z; c)F(z; c)
\end{equation}
for all $c \in \pfr$ and a fundamental solution whenever $a(c)$ has no two eigenvalues differing by a nonzero integer.\end{lemma}

\begin{proof} The lemma follows from a straightforward adaptation of the arguments and results appearing in \cite[Section 5]{W} as follows.  Expand the function $A(z; c)$ as $A(z; c) = \sum_{n = 0}^\infty a_{c, n}z^n$, where the $a_{c, n}$ are entire functions of $c \in \pfr$ and $a_{c, 0} = a(c)$.  First, consider a formal series $P(z; c) = \sum_{n = 0}^\infty p_{c, n}z^n$ with values in $Mat_{N \times N}(\cplx)$.  A direct calculation shows that the formal series $P(z; c)z^{a(c)}$ satisfies differential equation (\ref{eq:1D-DE}) if and only if
\begin{equation}\label{eq:recursion-relation}
p_{c, n}(n + a(c)) - a(c)p_{c, n} = \sum_{k = 1}^n a_{c, k}p_{c, n - k} \ \ \ \ \text{ for all } n \geq 1.
\end{equation} 
Let $\rho_{n + a(c)}$ denote the operator of right multiplication by $n + a(c)$ and let $\lambda_{a(c)}$ denote the operator of left multiplication by $a(c)$.  Note that $(\rho_{n + a(c)} - \lambda_{a(c)})^{-1}$ is a meromorphic $GL(Mat_{N \times N}(\cplx))$-valued function of $c \in \pfr$ that is holomorphic after multiplication by the polynomial $\det(\rho_{n + a(c)} - \lambda_{a(c)})$.  In particular, we may define meromorphic functions $p_{c, n}$, $n \geq 0$, of $c \in \pfr$ by setting $p_{c, 0} = \id$ and 
\begin{equation}
p_{c, n} := (\rho_{n + c_ia} - \lambda_{c_ia})^{-1}\sum_{k = 1}^n a_{c, k}p_{c, n - k} \ \ \ \ \text{ for all } n \geq 1.
\end{equation}
Let $q_{c, n} := k(c_i)p_{c, n}$ for all $n \geq 0$.

It is a standard fact that the operator $\rho_{n + a(c)} - \lambda_{a(c)}$ is singular if and only if $n + a(c)$ and $a(c)$ have an eigenvalue in common, i.e. if $n = \lambda - \lambda'$ for some eigenvalues $\lambda, \lambda' \in \spec(a(c))$.  As $a(c)$ depends linearly on $c \in \pfr$ and $U \subset \pfr$ is bounded, it follows that there exists $m > 0$ such that $\rho_{n + a(c)} - \lambda_{a(c)}$ is invertible for all $n > m$ and $c \in \bar{U}$, where $\bar{U}$ denotes the closure of $U$.  Let $k(c) = \prod_{n = 1}^m \det(\rho_{n + a(c)} - \lambda_{a(c)})$, and let $q_{c, n} := k(c)p_{c, n}$ for all $n \geq 0$.  Then $q_{c, n}$ is holomorphic function of $c$ in a neighborhood of the closure $\bar{U}$ for all $n \geq 0$ and the formal series $Q(z; c) := \sum_{n = 0}^\infty q_{c, n}z^n$ satisfies differential equation (\ref{eq:1D-DE}).

It now suffices to show that there exists $r > 0$ such that the series $Q(z; c) = \sum_{n = 0}^\infty q_{c, n}z^n$ converges absolutely and uniformly for $(c, z) \in U \times B_r(0)$.  In particular, it then follows that $Q(z; c)$ is holomorphic for $(c, z) \in U \times B_r(0)$, and therefore also for $(c, z) \in U \times B_R(0)$ by standard results on holomorphic dependence of solutions on parameters and initial conditions (see, e.g., \cite[Theorem 1.1]{IY}).  That $Q(z; c)$ is invertible, and hence that $Q(z; c)z^{a(c)}$ is a fundamental solution to (\ref{eq:1D-DE}), for those $c \in U$ such that $a(c)$ has no two eigenvalues differing by a nonzero integer follows from the equality $Q(0; c) = k(c)\id$.

The desired uniform convergence of $Q(z; c) = \sum_{n = 0}^\infty q_{c, n}z^n$ can be shown by a modification of the proof of \cite[Theorem 5.3]{W} as follows.  As $A(z; c) = \sum_{n = 0}^\infty a_{c, n}z^n$ is holomorphic on $\pfr \times B_R(0)$, it follows by considering the Taylor expansions of the $a_{c, n}$ and the boundedness of $U$ that there is a scalar-valued formal series $b(z) = \sum_{n = 1}^\infty b_nz^n$ defining a holomorphic function on $B_R(0)$ and such that $||a_{c, n}|| \leq b_n$ for all $n \geq 1$ and $c \in U$.  By the boundedness of $U$ again, it follows that there exist $N', C, D> 0$ such that $$||(\rho_{n + a(c)} - \lambda_{a(c)})^{-1}|| \leq C \ \ \ \ \text{ for all } n > N', c \in U$$ and $$||q_{c,n}|| \leq D \ \ \ \ \text{ for all } n \leq N', c \in U.$$  Now, consider the formal scalar-valued series $m(z) = \sum_{n = 0}^\infty m_nz^n$ defined by $$m_r = D \ \ \ \ \text{ for all } r \leq N'$$ and $$m_r = C\sum_{s = 1}^rb_sm_{r - s} \ \ \ \ \text{ for all } r > N'.$$  As $m(z)$ satisfies the equation $$m(z) = Cb(z)m(z) + D + \sum_{s = 1}^N (D - C\sum_{t = 1}^s b_tD)z^s$$ and $b(z)$ is holomorphic near $0$ and $b(0) = 0$, it follows that there exists $r > 0$ such that $m(z)$ is holomorphic in $B_r(0)$.  By construction we have $||q_{c, n}|| \leq m_n$ for all $n \geq 0$ and $c \in U$, so for all $c \in U$ and $z \in B_r(0)$ the series $Q(z; c) = \sum_{n = 0}^\infty q_{c, n}z^n$ is majorized by the series $m(z)$.  In particular, $Q(z; c) = \sum_{n = 0}^\infty q_{c, n}z^n$ is absolutely and uniformly convergent for all $c \in U$ and $z \in B_r(0)$, as needed.
\end{proof}

\begin{lemma}\label{nD-DE-lemma} Let $n, N > 0$ be positive integers, let $R > 0$ be a positive real number, let $a_1, ..., a_n : \pfr \rightarrow Mat_{N \times N}(\cplx)$ be linear functions, and let $$A_c = \sum_{j = 1}^n a(c)z_j^{-1}dz_j + \Omega_c$$ be a meromorphic 1-form on $B_R(0)^n \subset \cplx^n$ with values in $Mat_{N \times N}(\cplx)$ such that

(1) the form $\Omega_c$ is holomorphic on $B_R(0)^n$ with holomorphic dependence on $c \in \pfr$

(2) $d + A_c$ defines a flat connection on $(B_R(0) \backslash\{0\})^n$ for all $c \in \pfr$.

\noindent Then for any bounded domain $U \subset \pfr$ there exists, for all $1 \leq j \leq n$, a function $Q_j(z_j, ..., z_n; c)$ with values in $Mat_{N \times N}(\cplx)$, holomorphic in $(c, z_j, ..., z_n) \in U \times B_R(0)^{n - j + 1}$, such that $$F(z; c) := Q_1(z_1, ..., z_n; c)z_1^{a_1(c)}\cdots Q_{n - 1}(z_{n - 1}, z_n; c)z_{n - 1}^{a_{n - 1}(c)}Q_n(z_n; c)z_n^{a_n(c)}$$ is a solution of the differential equation $$(d + A_c)F(z; c) = 0$$ for all $c \in U$ and a fundamental solution whenever no two eigenvalues of any $a_i(c)$ differ by a nonzero integer.

\end{lemma}

\begin{proof}  The claim follows from iterated applications of Lemma \ref{1D-DE-lemma} and standard results (e.g. \cite[Theorem 1.1]{IY}) on holomorphic dependence on parameters and initial conditions of solutions of differential equations.
\end{proof}

We can now complete the proof of the existence of the Dunkl weight function as a tempered distribution with values in $\End_\cplx(\lambda)$ with holomorphic dependence on $c \in \pfr$:

\begin{proof}[Proof of Theorem \ref{main-existence-theorem}] Let $K_c : \hfr_{\real, reg} \rightarrow \End_\cplx(\lambda)$ be the family of analytic functions with holomorphic dependence on $c \in \pfr$ in the statement of Theorem \ref{integral-rep-theorem}.  Note that for $|c_s|$ sufficiently small $K_c$ already defines a holomorphic family of tempered distributions on $\hfr_\real$ with the desired properties.  It therefore suffices to show that this family of distributions for $|c_s|$ small has a holomorphic continuation to all $c \in \pfr$.  To see this, note that by Proposition \ref{dunkl-weight-characterization-proposition} the distribution $K_c$ for $|c_s|$ small satisfies properties (i)-(iii) of that proposition, and therefore by Proposition \ref{dunkl-weight-props} is homogeneous of degree $-2\chi_\lambda(\sum_{s \in S}c_ss)/\dim\lambda$.  It follows that the holomorphic continuation is homogeneous of degree $-2\chi_\lambda(\sum_{s \in S}c_ss)/\dim\lambda$, and hence tempered, for all $c \in \pfr$.  The characterizing properties (i)-(iii) appearing in Proposition \ref{dunkl-weight-characterization-proposition} then must hold for all parameters $c \in \pfr$ for the holomorphic continuation as they hold for $|c_s|$ sufficiently small.  It follows from Proposition \ref{dunkl-weight-characterization-proposition} that such a holomorphic continuation is a family of tempered distributions as in Theorem \ref{main-existence-theorem}.  So, it remains to construct such a holomorphic continuation.

Let $\pi : Y \rightarrow \hfr$ be the De Concini-Procesi wonderful model \cite{DP1} for the hyperplane arrangement $\cup_{s \in S} \ker(\alpha_s) \subset \hfr$.  Recall that $Y$ is a smooth $\cplx$-variety with $W$-action, $\pi$ is proper and $W$-equivariant, $\pi$ restricts to an isomorphism $\pi^{-1}(\hfr_{reg}) \rightarrow \hfr_{reg}$, and $\pi^{-1}(\cup_{s \in S} \ker(\alpha_s))$ is a normal crossings divisor.  Furthermore, let $\pi_\real : Y_\real \rightarrow \hfr_\real$ denote the real locus of $\pi$, which shares the same properties as $\pi$ but for the corresponding real hyperplane arrangement in $\hfr_\real$.  Let $Y_{reg} = \pi^{-1}(\hfr_{reg})$ and let $Y_{\real, reg} = Y_{reg} \cap Y_\real = \pi_\real^{-1}(\hfr_{\real, reg})$.

Note that it follows from the $W$-equivariance of $\pi$ and the fact that $\pi|_{Y_{reg}} : Y_{reg} \rightarrow \hfr_{reg}$ is an isomorphism that $\pi$ preserves stabilizers in $W$, i.e. that for all $d \in Y$ we have $\stab_W(d) = \stab_W(\pi(d))$.  The inclusion $\stab_W(d) \subset \stab_W(\pi(d))$ is trivial, so consider $w \in \stab_W(\pi(d))$.  Let $U \subset Y$ be an open neighborhood of $d$ in $Y$ that is stable under $\stab_W(d)$ and such that the $W/\stab_W(d)$-orbits of $U$ are disjoint.  Let $U_{reg} = U \cap Y_{reg}$.  As $w(\pi(d)) = \pi(d)$ it follows that $w\pi(U) \cap \pi(U) \neq \emptyset$, and hence $w\pi(U_{reg}) \cap \pi(U_{reg}) \neq \emptyset$ as well.  It then follows that $wU_{reg} \cap U_{reg} = w\pi^{-1}\pi(U_{reg}) \cap \pi^{-1}\pi(U_{reg}) \neq \emptyset$ as well.  In particular, $w \in \stab_W(d)$, so $\stab_W(d) = \stab_W(\pi(d))$ for all $d \in Y$, as claimed.

Consider the modified KZ connection
\begin{equation}\label{eq:modified-KZ}
\nabla_{KZ}' := d - \sum_{s \in S} c_s\frac{d\alpha_s}{\alpha_s}s = d - \sum_{s \in S} c_sd(\log \alpha_s)s
\end{equation}
on $\hfr_{reg}$.  Fix a point $x_0 \in \ccal$, and let $F_c(z)$ be a $\End_\cplx(\lambda)$-valued fundamental solution of $\nabla_{KZ}'$ with $F_c(x_0) = \id$.  By standard results on holomorphic dependence of solutions on parameters, it follows that $F_c(z)$ is holomorphic in $c \in \pfr$.  Regard $F_c(z)$ as a multivalued function on $\hfr_{reg}$.  As the function $K_c : \ccal \rightarrow \End_\cplx(\lambda)$ satisfies differential equation (\ref{eq:2-sided-KZ}) it follows that $K_c(x) = F_{c^\dagger}(x)^{\dagger, -1}K_c(x_0)F_{c}(x)^{-1}$ for all $c \in \pfr$ and $x \in \ccal$.  In fact, by the proof of the equivalence of statements (b) and (c) in Theorem \ref{existence-theorem}, we have that the function $\widetilde{K}_c : \hfr_{reg} \rightarrow \End_\cplx(\lambda)$, $\widetilde{K}_c(z) := F_{c^\dagger}(z)^{\dagger, -1}K_c(x_0)F_{c}(z)^{-1}$, is single-valued and satisfies $\widetilde{K}_c(x) = K_c(x)$ for all $x \in \hfr_{\real, reg}$.  

The pullback $\pi^*\nabla_{KZ}'$ is a meromorphic 1-form on $Y$ with values in $\End_\cplx(\lambda)$ and singularities along the components of the divisor $Y\backslash Y_{reg} = \pi^{-1}(\cup_{s \in S}\ker(\alpha_s))$, with holomorphic dependence on $c \in \pfr$.  Following the discussion in the introduction to \cite{DP2}, consider the residue of $\pi^*\nabla_{KZ}'$ along the components of $Y \backslash Y_{reg}$ as follows.  As $Y \backslash Y_{reg}$ is a normal crossings divisor, for any point $d \in Y_\real$ there is a local complex coordinate system $z_1, ..., z_l$ such that $z_1(d) = \cdots z_l(d) = 0$ and the divisor $Y\backslash Y_{reg}$ is given near $d$ by the equation $z_1\cdots z_m = 0$ for some $m$, $0 \leq m  \leq l$.  It follows that for each reflection $s \in S$ the pullback $\pi^*\alpha_s$ is of the form $\pi^*\alpha_s = \prod_{j = 1}^m z_j^{k_{j, s}}p_s(z)$ for some integers $k_{j, s} \in \ints^{\geq 0}$ and regular functions $p_s(z)$ with $p_s(0) \neq 0$.  In these local coordinates, we have, by equation (\ref{eq:modified-KZ}) $$\pi^*\nabla_{KZ}' = d - \sum_{s \in S} c_sd(\log(\prod_{j = 1}^m z_j^{k_{j, s}}p_s))s = d - \sum_{j = 1}^m \left(\sum_{s \in S} c_sk_{j, s}s\right)d\log z_j + \Omega_c$$ where $\Omega_c$ is a 1-form with values in $\End_\cplx(\lambda)$, holomorphic in $z$ and $c$.

Note that this connection is of the form considered in Lemma \ref{nD-DE-lemma}.  For some $1 \leq j \leq m$, consider the residue $a_j(c) := \sum_{s \in S} c_sk_{j, s}s$ along the hyperplane $z_j = 0$.  Recall that the map $\pi$ is $W$-equivariant.  Let $d'$ be a generic point on the hyperplane $z_j = 0$, and consider the stabilizer $W' = \stab_W(\pi(d')) = \stab_W(d')$, a parabolic subgroup of $W$.  For $s \in S \backslash W'$, we have $s(\pi(d')) \neq \pi(d')$ and hence $\alpha_s(\pi(d')) \neq 0$.  In particular, in this case $\pi^*\alpha_s(d') \neq 0$, so $k_{j, s} = 0$.  Now consider instead $s \in S \cap W'$ and $w \in W'$.  The element $w \in W' = \stab_W(d')$ stabilizes the hyperplane $z_1 = 0$ and has finite order, from which it follows that $w^*z_1 = z_1p(z)$ for some holomorphic function $p(z)$ with $p(0) \neq 0$.  The $W$-equivariance of $\pi$ then implies that $k_{j, s} = k_{j, wsw^{-1}}$.  We see that the residue $a_j(c)$ is a $\pfr$-linear combination of sums of conjugacy classes of reflections in $W'$.  It follows that the residues $\{a_j(c)\}_{c \in \pfr}$ are simultaneously diagonalizable with eigenvalues that are linear functions of $c \in \pfr$.

By Lemma \ref{nD-DE-lemma}, for every bounded open set $U \subset \pfr$ there is a meromorphic function $G(c)$ of $c \in U$ such that, in the local coordinates $z_1, ..., z_l$ on $Y$ near $d$ as above, we have $$\pi^*F_c(z) = Q_1(z_1, ..., z_l; c)z_1^{a_1(c)}\cdots Q_{l - 1}(z_{l - 1}, z_l; c)z_{l - 1}^{a_{l - 1}(c)}Q_l(z_l; c)z_l^{a_l(c)}G(c)$$ where the functions $Q_j(z_j, ..., z_l; c)$ are holomorphic both in $z$ in a neighborhood of $d$ in $Y$ and also in $c \in U$.  As the residues $\{a_j(c)\}_{c \in \pfr}$ are simultaneously diagonalizable and as $\widetilde{K}_c(z) = F_{c^\dagger}(z)^{\dagger, -1}K(x_0)F_c(z)^{-1}$ is single valued on $\hfr_{reg}$ and holomorphic in $c$, it follows from the form of $\pi^*F_c(z)$ above that the matrix entries of $\pi^*\widetilde{K}_c(z)$ are linear combinations of functions of the form $f_c(z)|z_1|^{g_1(c)}\cdots |z_l|^{g_l(c)}$ for some linear functions $g_1, ..., g_l : \pfr \rightarrow \cplx$ and function $f_c(z)$ holomorphic in both $c \in U$ and also $z$ in a neighborhood of $d$ in $Y$.  It is a standard fact that for $\lambda \in \cplx$ the function $|x|^\lambda$ is locally integrable on $\real$ when the real part of $\lambda$ satisfies $\text{Re}(\lambda) > -1$ and therefore defines a homogeneous distribution for such $\lambda$, and this distribution $|x|^\lambda$ has a meromorphic continuation to all $\lambda \in \cplx$ with (simple) poles at the negative odd integers.  As $f_c(z)$ is holomorphic in $c \in U$ and holomorphic, in particular smooth, in $z$, it follows that the function $\pi^*K_c(x) = \pi^*\widetilde{K}_c(x)$ on $Y_{\real, reg}$ has an extension, meromorphic in $c \in U$, to a distribution on an open neighborhood of $d$ in $Y_\real$.  As the point $d \in Y_\real$ and the bounded open set $U \subset \pfr$ were arbitrary, it follows that the function $\pi^*K_c(x)$ on $Y_{\real, reg}$ has an extension to a distribution on $Y_{\real, reg}$ meromorphic in $c \in \pfr$.  Taking the pushforward of this distribution along $\pi$, one obtains an extension of the function $K_c$ to a distribution on $\hfr_{\real}$ that is meromorphic in $c$ and coincides with $K_c$ as a distribution when $|c_s|$ is small for all $s \in S$.  Denote this family of distributions also by $\mathscr{K}_c$

It only remains to show that $\mathscr{K}_c$ is in fact holomorphic, not merely meromorphic, in $c \in \pfr$.  This follows immediately from the finiteness of the Hermite coefficients of $K_c$, which are given by the Gaussian pairing $\gamma_{c, \lambda}$ by analyticity, as this is so for $|c_s|$ small: $$\int_{\hfr_\real} Q(x)^\dagger\mathscr{K}_c(z)P(x)e^{-|x|^2/2}dx = \gamma_{c,\lambda}(P, Q) \in \cplx \ \ \ \ \text{ for all } P, Q \in \cplx[\hfr] \otimes \lambda \text{ and } c \in \pfr.$$\end{proof}

\section{Signatures and the KZ Functor}\label{signature-comparision-section}

In this section we will apply the Dunkl weight function and methods from semiclassical analysis to prove the following comparison theorem for signatures of irreducible representations of rational Cherednik algebras and Hecke algebras, generalizing Corollary \ref{unitarity-corollary}:

\begin{theorem}\label{KZ-preserves-signatures-theorem} Let $\lambda \in \irr(W)$ be an irreducible representation of the finite Coxeter group $W$, and let $c \in \pfr_\real$ be a real parameter.  If $KZ(L_c(\lambda))$ is nonzero then the asymptotic signature $a_{c, \lambda}$ of $L_c(\lambda)$ is given by the formula
\begin{equation}\label{eq:asym-sig}
a_{c, \lambda} = \frac{p - q}{\dim KZ(L_c(\lambda))}
\end{equation}
where $p - q$ is the signature, up to sign, of a $B_W$-invariant Hermitian form on $KZ(L_c(\lambda))$.\end{theorem}

In the case that $W$ is a symmetric group, Theorem \ref{KZ-preserves-signatures-theorem} was proved by Venkateswaran \cite[Theorem 1.4]{Ve} by providing exact formulas for the left and right sides of equation (\ref{eq:asym-sig}).  The generalization above to arbitrary finite Coxeter groups was expected by Venkateswaran \cite{Ve} and Etingof.

We will begin with a reminder on semiclassical analysis in Section \ref{semiclassical-reminder-section}, prove a key analytic lemma in Section \ref{key-analysis-lemma-section}, and finally prove Theorem \ref{KZ-preserves-signatures-theorem} in Section \ref{proof-of-comparison-theorem-section}.

\subsection{Reminder on Semiclassical Analysis}\label{semiclassical-reminder-section} We briefly review notation and results from semiclassical analysis,
referring the reader to \cite{Z} for details.

Denote by $\Opw$ the Weyl quantization on $\real^n$,
mapping a symbol $a(x,\xi)\in C_c^\infty(\real^{2n})$
to an $h$-dependent family of operators
$\Opw(a):L^2(\real^n)\to L^2(\real^n)$ defined as follows:
\begin{equation}
  \label{e:Op-h-w}
\Opw(a)f(x)=\int_{\real^{2n}}e^{{\frac{i}{h}}\langle x-y,\xi\rangle}a\Big(\frac{x + y}{2},\xi\Big)
f(y)\,dyd\xi.
\end{equation}
More generally one can quantize $h$-dependent symbols
$a(x,\xi;h)$ which satisfy the following
derivative bounds for all multiindices $\alpha$ on $\real^{2n}$:
\begin{equation}
  \label{e:S-m}
|\partial^\alpha_{(x,\xi)}a(x,\xi;h)|\leq C_\alpha m(x,\xi),\quad
(x,\xi)\in\real^{2n},\quad
0<h\leq 1,
\end{equation}
where $C_\alpha$ is an $h$-independent constant depending only on the multiindex $\alpha$ and where $m(x,\xi)$ is an \emph{order function} as defined in~\cite[\S4.4.1]{Z}.
The order functions we will use here are $m\equiv 1$ and
\begin{equation}
  \label{e:m-1}
m_1(x,\xi) := 1+|x|^2+|\xi|^2.
\end{equation}
As in \cite[Definition 4.4.2]{Z}, for any order function $m$ we denote by $S(m)$ the class of (possibly $h$-dependent) symbols $a(x, \xi; h)$ such that $a(x,\xi; h) \in C^\infty(\real^{2n})$ for all $h \in (0, 1]$ and that satisfy the derivative bounds \eqref{e:S-m}.
The Weyl quantization $\Opw$ is defined for symbols $a \in S(m)$ and has the following properties:
\begin{enumerate}
\item for $a\in S(m)$, the operator
$\Opw(a)$ acts continuously on the space of Schwartz functions $\mathscr S(\real^n)$ and also on the space of tempered distributions and $\mathscr S'(\real^n)$ \cite[Theorem~4.16]{Z};
\item if $a\in S(1)$, then by \cite[Theorem~4.23]{Z}
\begin{equation}
  \label{e:L2-bdd}
\sup_{0<h\leq 1} \|\Opw(a)\|_{L^2\to L^2}<\infty;
\end{equation}
\item if $a\in S(m)$, $b\in S(\widetilde m)$, then there exists
$a\#b$ such that by \cite[Theorem~4.18]{Z}
\begin{equation}
  \label{e:product}
\Opw(a)\Opw(b)=\Opw(a\# b),\quad
a\# b\in S(m\widetilde m),\quad
a\# b = ab+\mathcal O(h)_{S(m\widetilde m)}.
\end{equation}
\item if $a\in S(m)$, then by \cite[(4.1.13)]{Z}
\begin{equation}
  \label{e:adjoint}
\Opw(a)^*=\Opw(\bar{a});
\end{equation}
\item if $a\in S(1)$ is real-valued and satisfies
$a\geq c>0$ for some constant $c$ and all sufficiently small $h$,
then there exists $h_0>0$ such that by \cite[Theorem~4.30]{Z}
\begin{equation}
  \label{e:garding}
\langle \Opw(a) f,f\rangle_{L^2}\geq {\frac{c}{2}}\|f\|_{L^2}^2\quad\text{for all }
f\in L^2(\real^n),\quad
0<h\leq h_0.
\end{equation}
\end{enumerate}

\begin{remark}\label{Ck-remark} In fact, as will soon be important for the proof of Theorem \ref{KZ-preserves-signatures-theorem}, one may also consider, for any sufficiently large positive integer $k$, $h$-dependent symbols that are only $C^k$.  Specifically, for any integer $k \geq 0$ and order function $m$ one may also consider the symbol class $S^{(k)}(m)$ of symbols $a(x, \xi; h)$ such that $a(x, \xi; h) \in C^k(\real^{2n})$ and for which the derivative bounds (\ref{e:S-m}) hold for all mutiindices $\alpha$ with $|\alpha| \leq k$.  The construction of the Weyl quantization makes sense for such symbols (in fact for distributions as well, see \cite[Theorem 4.2]{Z}), and the proofs of the properties recalled above are based on only a fixed finite number of terms in the semiclassical expansion, and therefore on only a fixed finite number of derivatives of the symbol.  In particular, for $k$ sufficiently large, the properties of the Weyl quantization listed above are valid for symbols from the class $S^{(k)}(m)$.\end{remark}

\subsubsection{The quantum harmonic oscillator}

Consider the quantum harmonic oscillator in $\real^n$~\cite[\S6.1]{Z},
$$
P := -h^2\Delta+|x|^2.
$$
It is a nonnegative self-adjoint operator on $L^2(\real^n)$.
We have $P=\Opw(p)$ where (recalling~\eqref{e:m-1})
$$
p(x,\xi)=|\xi|^2+|x|^2\in S(m_1).
$$
We will also use the shifted operators determined by $(y,\eta)\in\real^{2n}$
\begin{equation}
  \label{e:P-shifted}
\begin{aligned}
P_{(y,\eta)}&=\sum_{j=1}^n (-ih\partial_{x_j}-\eta_j)^2
+|x-y|^2;\\
P_{(y,\eta)}&=\Opw(p_{(y,\eta)}),\quad
p_{(y,\eta)}(x,\xi)=|\xi-\eta|^2+|x-y|^2\in S(m_1).
\end{aligned}
\end{equation}
Note that $P_{(y,\eta)}$ is conjugate to $P$:
\begin{equation}
  \label{e:P-conjugate}
P_{(y,\eta)}=T_{(y,\eta)}PT_{(y,\eta)}^{-1},\quad
T_{(y,\eta)}f(x)=e^{{\frac{i}{h}}\langle x,\eta\rangle}f(x-y).
\end{equation}
We use the following consequence of functional calculus
of pseudodifferential operators,
see~\cite[Theorem~14.9]{Z} and~\cite[\S8]{DS}:

\begin{lemma}
  \label{l:funcal}
Assume that $\varphi\in C_c^\infty(\real)$ is $h$-independent
and fix $(y,\eta)\in\real^{2n}$.
Then $\varphi(P_{(y,\eta)})=\Opw(a)$ for some $a$ satisfying
for all $N$
$$
a\in S(m_1^{-N}),\quad
a=\varphi\circ p_{(y,\eta)}+\mathcal O(h)_{S(m_1^{-N})}.
$$
\end{lemma}

\subsection{A Key Lemma}\label{key-analysis-lemma-section}

Before proving the main lemma, we will need the following elementary measure theoretic result:

\begin{lemma} \label{l:covering-lemma} For any Lebesgue measurable set $X \subset \real^n$ let $\vol(X)$ denote its volume.  For any compact set $K \subset \real^n$ and every $\epsilon, r > 0$, there exist finitely many points $x_1, ..., x_N \in \real^n$ and positive real numbers $r_i \in (0, r)$ such that

(a) $K \subset \bigcup_{i = 1}^N B_{r_i}(x_i)$

(b) $\sum_{i = 1}^N \vol(B_{r_i}(x_i)) < \vol(K) + \epsilon.$\end{lemma}

\begin{proof} Fix the compact set $K \subset \real^n$ and positive numbers $\epsilon, r > 0$.  Recall that the outer measures defined by coverings by rectangles or balls coincide, and that each gives rise to Lebesgue measure.  In particular, there are countably many open rectangles $R_i$ such that $K \subset \bigcup_{i = 1}^\infty R_i$ and $\sum_{i = 1}^\infty \vol(R_i) < \vol(K) + \epsilon/2.$  Subdividing and slightly enlarging the rectangles if necessary, we may further assume that $\vol(R_i) < \vol(B_{r/2}(0)).$  Similarly, there are countably many open balls $\{B_{r_{ij}}(x_{ij})\}_{j = 1}^\infty$ such that $R_i \subset \cup_{j = 1}^\infty B_{r_{ij}}(x_{ij})$ and $\sum_{j = 1}^\infty \vol(B_{r_{ij}}(x_{ij})) < \vol(R_{ij}) + \min\{\epsilon/2^{i + 1}, \vol(B_{r/2}(0))\}.$  In particular, for each $i, j$ we have $\vol(B_{r_{ij}}(x_{ij})) < 2\vol(B_{r/2}(0)) \leq \vol(B_r(0))$ so $r_{ij} < r$.  Also, the countably many open balls $\{B_{r_{ij}}(x_{ij})\}_{i, j = 1}^\infty$ cover the compact set $K$, and we may extract a finite subcover $\{B_{r_i}(x_i)\}_{i = 1}^N$ of $K$.  We have $\sum_{i = 1}^N \vol(B_{r_i}(x_i)) < \sum_{i, j = 1}^\infty \vol(B_{r_{ij}}(x_{ij})) < \vol(K) + \sum_{i = 1}^\infty \epsilon/2^i = \vol(K) + \epsilon$, as needed.\end{proof}

For $\beta\geq 0$,
denote by $V_h(\beta)$ the range of the spectral projection $\mathbf 1_{[0,\beta]}(P)$:
$$
V_h(\beta)=\spn\{u\in L^2(\real^n)\mid Pu=\lambda u\text{ for some }\lambda\in [0,\beta]\}\subset
L^2(\real^n).
$$
By Weyl's Law~\cite[Theorem~6.3]{Z} we have for any fixed $\beta$
\begin{equation}
  \label{e:V-h-dim}
\dim V_h(\beta)=c_n\beta^nh^{-n}+o(h^{-n})\quad\text{as }h\to 0,\quad
c_n:=\frac{1}{2^n\cdot n!}.
\end{equation}
Denote by
$$
\Pi_h:L^2(\real^n)\to L^2(\real^n)
$$
the orthogonal projector onto $V_h(1)$.  It is a direct consequence of \cite[Theorem 6.2]{Z} on the eigenfunctions of $P$ that we have
\begin{equation}\label{eq:Vh-vs-polys}
V_h(1) = \cplx[\real^n]^{\leq \frac{1}{2}(\frac{1}{h} - n)}e^{-|x|^2/2h},
\end{equation}
where $\cplx[\real^n]^{\leq \frac{1}{2}(\frac{1}{h} - n)}$ denotes the space of complex valued polynomials on $\real^n$ with degree at most $\frac{1}{2}(\frac{1}{h} - n)$.

We will denote by $\text{Mat}_{m \times m}(\cplx)$ and $\text{Herm}_{m \times m}(\cplx)$ the spaces of $m \times m$ complex matrices and $m \times m$ complex Hermitian matrices, respectively.

\begin{lemma}\label{l:matrix-case-lemma} Let $m, n$ be positive integers, let $k >> n$, and let $q : \real^n \rightarrow \text{Herm}_{m \times m}(\cplx)$ be a $C^k$ function taking values in Hermitian $m \times m$ matrices and with matrix entries $q_{ij}(x)$ satisfying, for some fixed nonnegative integer $N_0$, the derivative bounds:
\begin{equation}
  \label{e:matrix-Q-bound}
|\partial_\alpha q_{ij}(x)| \leq C_\alpha (1 + |x|)^{N_0}
\end{equation} for all multiindices $\alpha$ with $|\alpha| \leq k$.  Consider the family of operators, for $h > 0$, $$Q = Q(h) : V_h(1) \otimes \cplx^m \rightarrow V_h(1) \otimes \cplx^m, \ \ Q(f) = (\Pi_h \otimes \id)(qf).$$  Let $\spec(q(x))$ denote the set of eigenvalues of $q(x)$ with multiplicity, and put (here $B_r(x, \xi)$ denotes the open ball with radius $r$ and center $(x, \xi)$ in $\real^{2n}$) $$\gamma := \frac{\int_{B_1(0)} \#\{\lambda \in \spec(q(x)) : \lambda \leq 0\}dxd\xi}{\vol(B_1(0))}.$$  Then, for every $\epsilon > 0$ there exists $h_0 > 0$ such that for $0 < h < h_0$ the number of non-positive eigenvalues of $Q(h)$, counted with multiplicity, is bounded above by $(\gamma + \epsilon)c_nh^{-n}.$ \end{lemma}

\begin{proof}  We follow the idea of the proof of \cite[Theorem~6.8]{Z}.

First, note that the operator $Q(h)$ is well-defined; it follows from (\ref{eq:Vh-vs-polys}) and the polynomial growth (\ref{e:matrix-Q-bound}) of the matrix entries $q_{ij}$ that multiplication by $q$ defines a map $$q : V_h(1) \otimes \cplx^m \rightarrow L^2(\real^n) \otimes \cplx^m,$$ so $Q(h)(f) = (\Pi_h \otimes \id)(qf)$ does define a linear operator on the finite-dimensional space $V_h(1)$.

Given $q$ and $\epsilon > 0$, we may choose finitely many open balls $$
B_j:=B_{r_j}((y_j,\eta_j))\subset\real^{2n},\quad
(y_j,\eta_j)\in \real^{2n},\quad
r_j>0;\quad
j=1,\dots,N,
$$
vectors $v_1, ..., v_N \in \cplx^m$, and constants $c_0, C_1, ..., C_N > 0$ all such that
\begin{equation}\label{e:matrix-kl-1}
q(x) + \sum_{j = 1}^N C_j\mathbbm{1}_{B_{r_j}((y_j, \eta_j))}(x,\xi)v_jv_j^T \geq 3c_0 \text{ for all } (x, \xi) \in B_1(0)
\end{equation}
\begin{equation}\label{e:matrix-kl-2}
\sum_j \vol(B_j)\leq \Big(\gamma+{\frac{\varepsilon}{3}}\Big)\vol(B(0,1)),
\end{equation}
where the inequality in the first condition indicates that all eigenvalues are at least $3c_0$ and where $\mathbbm{1}_{B_{r_j}((y_j, \eta_j))}$ denotes the indicator function of $B_{r_j}((y_j, \eta_j))$.  To see that this is possible, we use Lemma \ref{l:covering-lemma} and descending induction on $M := \max_{(x,\xi) \in B_1(0)} \#\{\lambda \in \spec(q(x)) : \lambda \leq 0\}.$  If $M = 0$, there is nothing to prove.  Otherwise, consider the compact set $$K_M := \{(x, \xi) \in B_1(0) : \#\{\lambda \in \spec(q(x)) : \lambda \leq 0\} = M\}.$$  For each point $(x, \xi) \in K_M$, let $v_{(x, \xi)} \in \cplx^m$ be an eigenvector of $q(x)$ with eigenvalue $\lambda_{(x, \xi)} \leq 0$, and let $C_{(x, \xi)} = 1 - \lambda_{(x, \xi)}$.  Then for all $(x, \xi) \in K_M$ the matrix $q(x) + C_{(x, \xi)}v_{(x, \xi)}v_{(x, \xi)}^T$ has at most $M - 1$ nonpositive eigenvalues counted with multiplicity, and it follows that there exists a number $r_{(x, \xi)} > 0$ such that for all $(y, \eta) \in B_{r_{(x, \xi)}}((x, \xi))$ the matrix $q(y) + C_{(x, \xi)}v_{(x, \xi)}v_{(x, \xi)}^T$ has at most $M - 1$ nonpositive eigenvalues counted with multiplicity.  As $K_M$ is compact, the open covering $\{B_{r_{(x, \xi)}}((x, \xi))\}_{(x, \xi) \in K_M}$ of $K_M$ admits a Lebesgue number $r > 0$.  Applying Lemma \ref{l:covering-lemma} to $K_M$ for this $r$ and replacing $q_1$ by a function defined similarly to the expression appearing in (\ref{e:matrix-kl-1}), we may reduce to the case that $\max_{(x,\xi) \in B_1(0)} \#\{\lambda \in \spec(q_1(x, \xi)) : \lambda \leq 0\} < M$ and continue in this manner by induction, defining $K_{M - 1}$ to be the closure of $\{(x, \xi) \in B_1(0) : \#\{\lambda \in \spec(q(x)) : \lambda \leq 0\} = M - 1\}$, etc.  The ordered eigenvalues of the resulting function in (\ref{e:matrix-kl-1}) are then positive and clearly lower-semicontinuous so are bounded away from 0 on the compact set $B_1(0)$, so the eigenvalues are bounded below by $3c_0$ for some $c_0 > 0$.

As above, for $(y, \eta) \in \real^{2n}$ let $p_{(y, \eta)}$ be defined by $p_{(y, \eta)}(x, \xi) = |x - y|^2 + |\xi - \eta|^2$, and let $p_j = p_{(y_j, \eta_j)}$.  It is clear that there exist real numbers $\tilde{r}_i > r_i$ and functions $\chi, \chi_1, ..., \chi_N \in C^\infty_c(\real^{2n})$ such that
\begin{gather}
  \label{e:matrix-kl-3}
\widetilde q:=(\chi\circ p)^2q+(1-\chi\circ p)\id+\sum_{j=1}^N (\chi_j\circ p_j)v_jv_j^T\geq 2c_0,\\
  \label{e:matrix-kl-4}
[0, 1] \cap \supp(1-\chi) = \emptyset,\quad
\supp(\chi_j)\subset (-\infty,\tilde r_j^2],\\
  \label{e:matrix-kl-5}
\sum_j \tilde r_j^{2n}\leq \gamma+{\frac{2\varepsilon}{3}},
\end{gather}
where $\id \in \text{Mat}_{m \times m}(\cplx)$ denotes the identity matrix.

By Remark \ref{Ck-remark}, we may take $k$ sufficiently large so that the properties of the Weyl quantization recalled in Section \ref{semiclassical-reminder-section} hold for the symbol classes $S^{(k)}(m)$.  Furthermore, let $\Opw$ act element-wise on matrices of such symbols.  Define the operator
$$
Q_1:=\chi(P)\Opw(q)\chi(P)+(I-\chi(P))\id+\sum_{j=1}^N \chi_j(P_j)v_jv_j^T.
$$ 
Here $\Opw(q)$ is the multiplication operator by $q$;
by~\eqref{e:matrix-Q-bound} we have $$q \in S^{(k)}(m_1^M) \otimes \text{Herm}_{m \times m}(\cplx)$$ for some $M$.  Using Lemma~\ref{l:funcal} in each matrix entry and 
the product formula~\eqref{e:product} we see that
$$
Q_1=\Opw(a)\quad\text{for some }a\in S^{(k)}(1) \otimes \text{Herm}_{m \times m}(\cplx),$$
$$
a=\widetilde q+\mathcal O(h)_{S^{(k)}(1) \otimes \text{Herm}_{m \times m}(\cplx)}.
$$  That we may take $a \in S^{(k)}(1) \otimes \text{Herm}_{m \times m}(\cplx)$ rather than simply $S^{(k)}(1) \otimes \text{Mat}_{m \times m}(\cplx)$ follows from \cite[(4.1.12)]{Z}.  In particular, we see that $Q_1$ is a bounded self-adjoint operator on $L^2(\real^n) \otimes \cplx^m$ by \cite[Theorem 4.23]{Z}.

As $\widetilde{q} \geq 2c_0$ it follows that there exists $h_0 > 0$ such that $a \geq \frac{3}{2}c_0$ for $0 < h < h_0$.  Recall that the square root of a positive definite matrix depends smoothly (in fact, analytically) on the matrix entries.  In particular, the square root $b := \sqrt{a - c_0}$ is defined for $0 < h < h_0$ and is an element of the symbol class $S^{(k)}(1) \otimes \text{Herm}_{m \times m}(\cplx)$.  As the operator $\Opw(b)$ is self-adjoint, it follows from product formula (\ref{e:product}) that we have $$\Opw(a - c_0) = \Opw(b)^*\Opw(b) + \mathcal{O}(h)_{L^2 \otimes \cplx^m}.$$  As $\Opw(b)^*\Opw(b)$ is manifestly a nonnegative operator, shrinking $h_0$ if necessary there is a constant $C > 0$ such that $\Opw(a - c_0) \geq -Ch$ for $0 < h < h_0.$  In particular, $\Opw(a) \geq c_0 - Ch$ for $0 < h < h_0$.  Taking $h_0$ smaller again if necessary, we have
\begin{equation}
\Opw(a) \geq \frac{c_0}{2} \ \ \ \ \text{ for } 0 < h < h_0,
\end{equation}
i.e.
\begin{equation}
\langle Q_1f, f \rangle = \langle \Opw(a)f, f \rangle \geq \frac{c_0}{2}||f||^2_{L^2 \otimes \cplx^m} \text{ for all } f \in L^2(\real^n) \otimes \cplx^m, \ \ \ 0 < h < h_0.
\end{equation}

In particular this applies to all $f\in V_h(1) \otimes \cplx^m$. However,
for such $f$ we have $f=\chi(P)f$, therefore
\begin{equation}
  \label{e:matrix-cling-2}
\langle Qf , f\rangle
+\langle Q_2 f,f\rangle = \langle Q_1f, f \rangle \geq \frac{c_0}{2}\|f\|^2\quad\text{for all }f\in V_h(1) \otimes \cplx^m
\end{equation}
where
\begin{equation}
Q_2:=\sum_{j=1}^N\chi_j(P_j)v_jv_j^T.
\end{equation}
and where the inner products and norms above are in $L^2(\real^n) \otimes \cplx^m$.
By~\eqref{e:P-conjugate}, \eqref{e:V-h-dim}, and~\eqref{e:matrix-kl-5} we estimate the rank of $Q_2$ 
for small $h$ by
\begin{equation}
  \label{e:matrix-cling-3}
\rank Q_2\leq \sum_{j=1}^N\dim V_h(\tilde r_j^2)
= c_nh^{-n}\sum_{j=1}^N \tilde r_j^{2n}+o(h^{-n})
\leq (\gamma+\varepsilon)c_nh^{-n}.
\end{equation}
Together~\eqref{e:matrix-cling-2} and~\eqref{e:matrix-cling-3} give the required estimate.\end{proof}

\subsection{Proof of the Signature Comparison Theorem}\label{proof-of-comparison-theorem-section}

\begin{proof}[Proof of Theorem \ref{KZ-preserves-signatures-theorem}]  Let $k$ be a positive integer sufficiently large so that Lemma \ref{l:matrix-case-lemma} holds for $C^k$ $\text{Herm}(\lambda)$-valued functions on $\hfr_\real$ satisfying the required derivative bounds appearing in that lemma.  Let $c \in \pfr_\real$ be a real parameter, and suppose $\supp L_c(\lambda) = \hfr$.  Let $$K_c : \hfr_{\real, reg} \rightarrow \text{Herm}(\lambda)$$ be the function appearing in Theorem \ref{integral-rep-theorem}, i.e. $K_c$ is the restriction of the Dunkl weight function at parameter $c$ to $\hfr_{\real, reg}$.  By Theorem \ref{integral-rep-theorem} there is a positive integer $N > 0$ such that $$\gamma_c(\delta^NP_1, \delta^NP_2) = \int_{\hfr_\real} P_2(x)^\dagger\delta^{2N}(x)K_c(x)P_1(x)e^{-|x|^2/2}dx \ \ \ \ \text{ for all } P_1, P_2 \in \cplx[\hfr] \otimes \lambda.$$  By Lemma \ref{homogeneous-L1-lemma} $K_c$ is a homogeneous function and we may take $N$ large enough so that $\delta^{2N}K_c$ extends continuously to all of $\hfr_{\real}$ with value $0$ on the union $\hfr_\real\backslash\hfr_{\real, reg}$ of the reflection hyperplanes.  By the homogeneity of $K_c$, it follows that we may take $N$ large enough so that $\delta^{2N}K_c : \hfr_\real \rightarrow \text{Herm}(\lambda)$ is in fact $C^k$ and satisfies the hypotheses of the function $q$ appearing in Lemma \ref{l:matrix-case-lemma}.

Let $d$ be the degree of homogeneity of the function $\delta^{2N}K_c$, let $h > 0$ be arbitrary, and let $Q(h) : V_h(1) \otimes \lambda \rightarrow V_h(1) \otimes \lambda$, $V_h(1) \subset L^2(\hfr_\real)$, be the operator $$Q(h)(f) := (\Pi_h \otimes \id)((\delta^{2N}K_c)f)$$ considered in Lemma \ref{l:matrix-case-lemma}.  Recall from (\ref{eq:Vh-vs-polys}) that$$V_h(1) = \cplx[\hfr]^{\leq \frac{1}{2}(\frac{1}{h} - l)}e^{-|x|^2/2h}$$ where $l = \dim \hfr$.  For any $P_1, P_2 \in \cplx[\hfr]^{\leq \frac{1}{2}(\frac{1}{h} - l)} \otimes \lambda$ we have $$\gamma_{c, \lambda}(\delta^NP_1, \delta^NP_2) = \int_{\hfr_\real} P_2(x)^\dagger (\delta^{2N}K_c)(x)P_1(x)e^{-|x|^2/2}dx$$ $$= \int_{\hfr_\real} (P_2(x)e^{-|x|^2/4})^\dagger (\delta^{2N}K_c)(x)(P_1(x)e^{-|x|^2/4})dx$$ $$= (2/h)^{l/2}\int_{\hfr_\real} (P_2(\sqrt{2/h}x)e^{-|x|^2/2h})^\dagger (\delta^{2N}K_c)(\sqrt{2/h}x)(P_1(\sqrt{2/h}x)e^{-|x|^2/2h})dx$$ $$= (2/h)^{(l + d)/2}\int_{\hfr_\real} (P_2(\sqrt{2/h}x)e^{-|x|^2/2h})^\dagger (\delta^{2N}K_c)(x)(P_1(\sqrt{2/h}x)e^{-|x|^2/2h})dx$$ $$= (2/h)^{(l + d)/2}\left\langle Q(h)\left(P_1(\sqrt{2/h}x)e^{-|x|^2/2h}\right), P_2(\sqrt{2/h}x)e^{-|x|^2/2h}\right\rangle_{L^2(\hfr_\real) \otimes \lambda}$$ where $\langle \cdot, \cdot \rangle_{L^2(\hfr_\real) \otimes \lambda}$ denotes the inner product on $L^2(\hfr_\real) \otimes \lambda$ induced from the standard inner product on $L^2(\hfr_\real)$ and the inner product $(\cdot, \cdot)_\lambda$ on $\lambda$.  In particular, for all $h > 0$,
\begin{equation}\label{eq:pos-def-comparison}
\begin{array}{c}
\max \{\dim U : U \subset \cplx[\hfr]^{\leq \frac{1}{2}(\frac{1}{h} - l)} \otimes \lambda, \gamma|_{\delta^NU} \text{ is positive definite}\} \\ = \#\{\mu \in \spec(Q(h)) : \mu > 0\},
\end{array}
\end{equation}
\begin{equation}\label{eq:neg-def-comparison}
\begin{array}{c}
\max \{\dim U : U \subset \cplx[\hfr]^{\leq \frac{1}{2}(\frac{1}{h} - l)} \otimes \lambda, \gamma|_{\delta^NU} \text{ is negative definite}\} \\ = \#\{\mu \in \spec(Q(h)) : \mu < 0\},
\end{array}
\end{equation}
and
\begin{equation}\label{eq:kernel-comparison}
\dim \rad\left(\gamma|_{\delta^N\cplx[\hfr]^{\leq \frac{1}{2}(\frac{1}{h} - l)} \otimes \lambda}\right) = \dim \ker Q(h).
\end{equation}

For a point $x \in \hfr_{\real, reg}$, let $p$ be the dimension of a maximal positive-definite subspace of $\lambda$ with respect to the Hermitian form $K_c(x)$, let $q$ be the dimension of a maximal negative-definite subspace of $\lambda$ with respect to $K_c(x)$, and let $r = \dim \rad(K_c(x)) = \dim \lambda - p - q$ be the dimension of the radical of $K_c(x)$.  Recall that the Hermitian forms $K_c(x)$ and $K_c(x')$ are equivalent for all $x, x' \in \hfr_{\real, reg}$, so the integers $p, q, r$ do not depend on the choice of the point $x \in \hfr_{\real, reg}$.  As $K_c(x)$ descends to a $B_W$-invariant non-degenerate Hermitian form on $KZ_x(L_c(\lambda))$, the quantity $p - q$ is as in the statement of Theorem \ref{KZ-preserves-signatures-theorem}.

As usual, let $N_c(\lambda)$ denote the maximal proper submodule of $\Delta_c(\lambda)$, and recall that $N_c(\lambda) = \rad(\gamma_{c, \lambda})$.  Recall also from the proof of Corollary \ref{unitarity-corollary} that for any $x \in \hfr_{\real, reg}$ we have $KZ_x(N_c(\lambda)) = \rad(K_c(x))$, where we make the usual identification of $KZ_x(\Delta_c(\lambda))$ with $\lambda$ as a vector space and the identification of $KZ_x(N_c(\lambda))$ as a subspace of $KZ_x(\Delta_c(\lambda))$.  In particular, $N_c(\lambda)$ is a graded $\cplx[\hfr]$-module whose restriction to $\hfr_{reg}$ is an algebraic vector bundle of rank $r = \dim \rad(K_c(x))$.  It follows from standard results on Hilbert series and equation (\ref{eq:kernel-comparison}) that we have
\begin{equation}\label{eq:kernel-limit}
\lim_{h \rightarrow 0} \frac{\dim \ker Q(h)}{\dim V_h(1)} = r.
\end{equation}
In the case of the function $\delta^{2N}K_c$, the constant $\gamma$ appearing in Lemma \ref{l:matrix-case-lemma} is given by $\gamma = q + r$.  In particular, it follows from that lemma and equation (\ref{eq:kernel-limit}) that we have
\begin{equation}\label{eq:neg-lim-sup}
\limsup_{h \rightarrow 0} \frac{\#\{\mu \in \spec(Q(h)) : \mu < 0\}}{\dim V_h(1)} = q.
\end{equation}
Similarly, applying Lemma \ref{l:matrix-case-lemma} to the function $-\delta^{2N}K_c$ we have 
\begin{equation}\label{eq:pos-lim-sup}
\limsup_{h \rightarrow 0} \frac{\#\{\mu \in \spec(Q(h)) : \mu > 0\}}{\dim V_h(1)} = p.
\end{equation}
As $$\frac{\#\{\mu \in \spec(Q(h)) : \mu > 0\}}{\dim V_h(1)} + \frac{\#\{\mu \in \spec(Q(h)) : \mu < 0\}}{\dim V_h(1)} + \frac{\dim \ker Q(h)}{\dim V_h(1)} = \dim \lambda$$ we see that equations (\ref{eq:neg-lim-sup}) and (\ref{eq:pos-lim-sup}) hold with limits replacing lim sups.  We therefore have, by equations (\ref{eq:pos-def-comparison}) and (\ref{eq:neg-def-comparison}), $$\lim_{n \rightarrow \infty} \frac{sign(\gamma_{c, \lambda}|_{(\delta^{N}L_c(\lambda))^{\leq n}})}{\dim (\delta^{N}L_c(\lambda))^{\leq n}}$$ $$= \lim_{h \rightarrow 0} \frac{\#\{\mu \in \spec(Q(h)) : \mu > 0\} - \#\{\mu \in \spec(Q(h)) : \mu < 0\}}{\#\{\mu \in \spec(Q(h)) : \mu > 0\} + \#\{\mu \in \spec(Q(h)) : \mu < 0\}}$$ $$= \frac{p - q}{p + q}$$ $$= \frac{p - q}{\dim KZ(L_c(\lambda))}.$$

As $\supp L_c(\lambda) = \hfr$ we have, as explained in the proof of Corollary \ref{unitarity-corollary}, $$\lim_{n \rightarrow \infty} \frac{\dim (\delta^NL_c(\lambda))^{\leq n}}{\dim L_c(\lambda)^{\leq n}} = 1.$$  In particular, the asymptotic signature of $L_c(\lambda)$ is given by $$a_{c, \lambda} = \lim_{n \rightarrow \infty} \frac{sign(\gamma_{c, \lambda}|_{L_c(\lambda)^{\leq n}})}{\dim L_c(\lambda)^{\leq n}} = \lim_{n \rightarrow \infty} \frac{sign(\gamma_{c, \lambda}|_{(\delta^{N}L_c(\lambda))^{\leq n}})}{\dim (\delta^{N}L_c(\lambda))^{\leq n}} = \frac{p - q}{\dim KZ(L_c(\lambda))},$$ as claimed.\end{proof}

\section{Conjectures and Further Directions}\label{conj-section}

In this section we will discuss various possible extensions of the results of this paper.

\subsection{Complex Reflection Groups}\label{complex-conj-section}  The Dunkl weight function $K_{c, \lambda}$ constructed in this paper is, at present, a phenomenon for finite real reflection groups.  There are at least two relevant objects that are absent for finite complex reflection groups.  First, the weight function $K_{c, \lambda}$ is a distribution on the real reflection representation $\hfr_\real$, but the complex reflection representation of a finite complex reflection group in general does not arise as the complexification of a real reflection representation.  Second, the Gaussian inner product $\gamma_{c, \lambda}$, for which $K_{c, \lambda}$ gives an integral formula, is itself absent for complex reflection groups.  While the contravariant form $\beta_{c, \lambda}$ is defined for an arbitrary finite complex reflection group $W$ and an irreducible representation $\lambda$, the definition of $\gamma_{c, \lambda}$ relies on the element $\mathbf{f}$ from the canonical $\mathfrak{sl}_2$-triple $\mathbf{e}, \mathbf{f}, \mathbf{h} \in H_c(W, \hfr)$ - but for complex $W$ there is no such $\mathfrak{sl}_2$-triple.

Nevertheless, from the proof of Theorem \ref{existence-theorem}, when $|c_s|$ is small for all $s \in S$ the Dunkl weight function $K_{c, \lambda}$ is given by integration against an analytic function $K : \hfr_{\real, reg} \rightarrow \End_\cplx(\lambda)$ that has a single-valued extension $\widetilde{K}(z)$ to $\hfr_{reg}$ of the form
\begin{equation}\label{K-tilde-eq}
\widetilde{K}(z) = F_{c^\dagger}(z)^{\dagger, -1}\widetilde{K}(x_0)F_c(z)^{-1},
\end{equation}
where $x_0 \in \hfr_{\real, reg}$, $\widetilde{K}(x_0) \in \End_\cplx(\lambda)$ defines a $B_W$-invariant sesquilinear pairing $$KZ_{x_0}(\Delta_c(\lambda)) \times KZ_{x_0}(\Delta_{c^\dagger}(\lambda)) \rightarrow \cplx,$$ and $F_c(z)$ is the monodromy of the modified KZ connection $$\nabla_{KZ}' = d - \sum_{s \in S} c_s\frac{d\alpha_s}{\alpha_s}s$$ from $x_0$ to $z$.  All of these ingredients are available for finite complex reflection groups as well.  By Remark \ref{complex-herm-form-remark}, when $W$ is a finite complex reflection group, for any $x_0 \in \hfr_{reg}$ there is an operator $A_c \in \End_\cplx(\lambda)$ defining a $B_W$-invariant sesquilinear pairing $$KZ_{x_0}(\Delta_c(\lambda)) \times KZ_{x_0}(\Delta_{c^\dagger}(\lambda)) \rightarrow \cplx,$$ unique up to scalar multiple for generic $c$.  The modified KZ connection is naturally generalized as $$\nabla_{KZ}' = d - \sum_{s \in S} 2c_s\frac{d\alpha_s}{(1 - \lambda_s)\alpha_s}s,$$ where $\lambda_s$ is the nontrivial eigenvalue of the complex reflection $s \in S$.  Taking $\widetilde{K}(x_0) = A_c$ and defining $\widetilde{K}(z)$ for $z \in \hfr_{reg}$ as in equation (\ref{K-tilde-eq}) defines a single-valued function $\widetilde{K} : \hfr_{reg} \rightarrow \End_\cplx(\lambda)$.  For generic $c$, such a function is uniquely determined up to a global scalar multiple.  It would be interesting to understand the meaning of these functions for finite complex reflection groups and to investigate their relationship to the contravariant form $\beta_{c, \lambda}$.

\subsection{Preservation of Jantzen Filtrations, Epsilon Factors, and Signatures}  In Section \ref{jantzen-section}, we introduced Jantzen filtrations for standard modules $\Delta_c(\lambda)$, depending on a base parameter $c_0 \in \pfr_\real$ and a deformation direction $c_1 \in \pfr_\real$ and arising from the analytic (in fact, polynomial) family of forms $\beta_{c, \lambda}$.  Given a point $x_0 \in \hfr_{\real, reg}$, the Dunkl weight function provides a natural family $K_{c, \lambda}(x_0)$ of $B_W$-invariant Hermitian forms on $KZ_{x_0}(\Delta_c(\lambda)),$ holomorphic in $c \in \pfr$.  Given a base parameter $c_0 \in \pfr_\real$ and deformation direction $c_1 \in \pfr$, one therefore obtains a Jantzen filtration on $KZ_{x_0}(\Delta_c(\lambda))$ by $H_q(W)$-submodules.  The proof of Corollary \ref{unitarity-corollary} shows that $KZ_{x_0}(\rad(\beta_{c, \lambda})) = \rad(K_c(x_0))$, i.e. $KZ_{x_0}$ intertwines the first two terms of the Jantzen filtration on $\Delta_c(\lambda)$ with the first two terms of the Jantzen filtration on $KZ_{x_0}(\Delta_c(\lambda))$.  It is natural to expect that this is true for all terms in the Jantzen filtration:

\begin{conjecture}\label{jantzen-intertwining-conjecture} For any base parameter $c_0 \in \pfr_\real$ and deformation direction $c_1 \in \pfr_\real$, the functor $KZ_{x_0}$ sends the Jantzen filtration on $\Delta_c(\lambda)$ to the Jantzen filtration on $KZ_{x_0}(\Delta_c(\lambda))$.\end{conjecture}

In fact, it may be possible to adapt the argument that $KZ_{x_0}(\rad(\beta{c, \lambda})) = \rad(K_c(x_0))$ appearing in the proof of Corollary \ref{unitarity-corollary} directly to prove Conjecture \ref{jantzen-intertwining-conjecture}.  Alternatively, Ivan Losev has suggested that Conjecture \ref{jantzen-intertwining-conjecture} may be able to be proved by a purely categorical argument by considering appropriate $\cplx[[t]]$-linear categories deforming $\oscr_c(W, \hfr)$ and $H_q(W)\mhyphen\text{mod}_{f.d.}$ and interpreting the relevant Hermitian forms as maps from standard objects to (complex conjugates of) costandard objects.

Additionally, by Theorem \ref{KZ-preserves-signatures-theorem}, the asymptotic signature of the graded Hermitian form on the first subquotient of the Jantzen filtration of $\Delta_c(\lambda)$ equals the normalized signature of the Hermitian form on the first subquotient of the Jantzen filtration of $KZ_{x_0}(\Delta_c(\lambda))$.  It is natural to expect that this too holds for all of the subquotients, although this should be significantly more difficult to show than Conjecture \ref{jantzen-intertwining-conjecture}.  In fact, this will follow from Lemmas \ref{wall-crossing-lemma} and \ref{K0-lemma}, and their analogues at the level of Hecke algebra representations, if the \emph{epsilon factors} $\epsilon_i \in \{\pm 1\}$ appearing in Lemma \ref{K0-lemma} are also compatible with $KZ_{x_0}$:

\begin{conjecture}\label{epsilon-factor-intertwining-conjecture} For any base parameter $c_0 \in \pfr_\real$ and deformation direction $c_1 \in \pfr_\real$, Conjecture \ref{jantzen-intertwining-conjecture} holds, and for each $k \geq 0$, the epsilon factors attached to each full-support irreducible constituent of the $k^{th}$ subquotient of the Jantzen filtration on $\Delta_c(\lambda)$ determine epsilon factors for the irreducible constituents of the $k^{th}$ subquotient of the Jantzen filtration on $KZ_{x_0}(\Delta_c(\lambda))$.  In particular, the asymptotic signature of the induced form on the $k^{th}$ subquotient of the Jantzen filtration of $\Delta_c(\lambda)$ equals the normalized signature, in the sense of Theorem \ref{KZ-preserves-signatures-theorem}, of the induced form on the $k^{th}$ subquotient of the Jantzen filtration on $KZ_{x_0}(\Delta_c(\lambda))$.\end{conjecture}

Furthermore, after making an appropriate definition of Jantzen filtrations on $KZ_{x_0}(\Delta_c(\lambda))$ in the complex reflection group case, perhaps following the ideas of Section \ref{complex-conj-section}, one may similarly formulate Conjecture \ref{epsilon-factor-intertwining-conjecture} in the complex reflection group case.  One possible approach to prove such a conjecture could be to develop an analogue for rational Cherednik algebras of the signed Kazhdan-Lusztig polynomials introduced by Yee \cite{Y}.  This would give an approach for providing an entirely algebraic proof of Theorem \ref{KZ-preserves-signatures-theorem} that is also valid for complex reflection groups.  

\subsection{Local Description of $K_{c, \lambda}$ and Modules of Proper Support} Let $c \in \pfr_\real$ be a real parameter.  Theorem \ref{KZ-preserves-signatures-theorem} shows that, when $L_c(\lambda)$ has full support, its asymptotic signature $a_{c, \lambda}$ can be described in a simple way in terms of the local nature of the distribution $K_{c,\lambda}$ near a generic point of its support in $\hfr_\real$.  It is natural to expect a similar statement to hold for arbitrary $L_c(\lambda)$.

From the proof of Theorem \ref{existence-theorem}, the distribution $K_{c, \lambda}$, locally near any $b \in \hfr_{\real, reg}$, is given by integration against an analytic function of the form $$B_{c}(x)^{\dagger, -1}K_{c, \lambda}(b)B_c(x)^{-1},$$ where $K_{c, \lambda}(b) \in \text{Herm}(\lambda)$ and $B_c(x)$ is an analytic $GL(\lambda)$-valued function of $x$ defined near $x_0$ and satisfying $B_c(b) = \id$.  Rearranging, we have
\begin{equation}\label{local-form-generic-equation}
B_c(x)^\dagger K_{c, \lambda}(x)B_c(x) = K_{c, \lambda}(b)
\end{equation}
for $x \in \hfr_{\real, reg}$ near $b$.

Now consider an arbitrary point $b \in \hfr_\real$ be an arbitrary point.  Let $W' = \stab_W(b)$, let $\hfr_{W'}$ be the unique $W'$-stable complement to $\hfr^{W'}$ in $\hfr$, and let $\hfr_{\real, W'} = \hfr_\real \cap \hfr_{W'}$.  The action of $W'$ on $\hfr_{\real, W'}$ realizes $W'$ as a finite real reflection group, generated by the reflections $S' := S \cap W'$.  In particular, we have the rational Cherednik algebra $H_{c'}(W', \hfr_{W'})$ and its category $\oscr_{c'}(W', \hfr_{W'})$ with parameter $c' := c|_{S'}$.    In a sense, the local structure of a module $M \in \oscr_c(W, \hfr)$ near the point $b \in \hfr$ respects the decomposition $\hfr = \hfr^{W'} \oplus \hfr_{W'}$, with the structure in the $\hfr_{W'}$ component described by the image $\res_{b} M \in \oscr_{c'}(W', \hfr_{W'})$ of $M$ under the Bezrukavnikov-Etingof parabolic restriction functor $\res_{b}$ \cite[Section 3.5]{BE} and with the structure in the $\hfr^{W'}$ component described by an associated local system \cite[Section 3.7, Proposition 3.20]{BE}.  The following conjecture describing the structure of $K_{c, \lambda}$ near $b$ is a natural analogue:

\begin{conjecture}\label{local-form-conjecture}  Let $c \in \pfr_\real$ be a real parameter, let $b \in \hfr_\real$ be an arbitrary point, let $W' = \stab_W(b)$, and let $x = (x', x'')$ denote the decomposition of a vector $x \in \hfr_\real$ with respect to the direct sum decomposition $\hfr_\real =  \hfr_{\real, W'} \oplus \hfr_\real^{W'}$.  Then there exists a $W'$-equivariant $\End_\cplx(\lambda)$-valued analytic function $B(x)$ defined in a neighborhood of $b$ such that $B(b) = \id$ and
\begin{equation}\label{local-form-equation}
B(x)^\dagger K_{c, \lambda}(x) B(x) = \sum_{\mu \in \irr(W')} K_{c, \mu}(x') \otimes h_\mu,
\end{equation}
where, for each $\mu \in \irr(W')$, $h_\mu$ is a Hermitian form on $\hom_{W'}(\mu, \lambda)$.\end{conjecture}

A similar statement should hold for all $c \in \pfr$.  For arbitrary parameters $c \in \pfr$, the lefthand side of equation (\ref{local-form-equation}) should be replaced by $B_{c^\dagger}(x)^\dagger K_{c, \lambda}(x)B_c(x)$, and the $h_\mu$ should be elements of $\End_\cplx(\hom_{W'}(\mu, \lambda))$, not necessarily Hermitian.  Generalizing the case $b \in \hfr_{\real, reg}$ in which $h_\lambda$ gives a $B_W$-invariant Hermitian form on the $H_q(W)$-representation $\lambda \cong_\cplx KZ_b(\Delta_c(\lambda))$, it is reasonable to expect that for all $\mu \in \irr(W')$ the space $\hom_{W'}(\mu, \lambda)$ is a representation of a generalized Hecke algebra, in the sense of \cite[Definition 3.25]{LS}, and that the operator $h_\mu$ is invariant under the action of the fundamental group $\pi_1(\hfr^{W'}_{reg}/I)$ appearing in \cite[Definition 3.25]{LS}.

We can already see that Conjecture \ref{local-form-conjecture} holds in several cases.  When $b \in \hfr_{\real, reg}$, we have $W' = 1$, and Conjecture \ref{local-form-conjecture} holds by equation (\ref{local-form-generic-equation}).  When $b = 0$, we have $W' = W$, and Conjecture \ref{local-form-conjecture} is trivial - simply take $B(x) = \id$.  When $b$ is a generic point on a reflection hyperplane, i.e. $W' = \langle s \rangle$ for some $s \in S$, that Conjecture \ref{local-form-conjecture} holds follows from the proof of Theorem \ref{existence-theorem} - the function $P_{i, c}(x)$ gives the required function $B(x)$ and the operators $K_i^{1, 1}, K_i^{-1, -1} \in \End_\cplx(\lambda)$, viewed as Hermitian forms on $\hom_{W'}(\triv, \lambda)$ and on $\hom_{W'}(\sgn, \lambda)$, respectively, give the required forms $h_\triv$ and $h_\sgn$, where $\triv$ and $\sgn$ denote the trivial and sign representations of $W'$, respectively.

In view of Conjecture \ref{local-form-conjecture}, a natural extension of Theorem \ref{KZ-preserves-signatures-theorem} to irreducible representations $L_c(\lambda)$ of arbitrary support is the following:

\begin{conjecture}\label{arbitrary-support-conjecture} Let $c \in \pfr_\real$ be a real parameter, let $\lambda \in \irr(W)$ be an irreducible representation of $W$, let $b \in \hfr_\real$ be a generic point in $\supp(L_c(\lambda)) \cap \hfr_\real = \supp K_{c, \lambda}$.  Then, using the notation from Corollary \ref{local-form-conjecture}, the asymptotic signature $a_{c, \lambda}$ of $L_{c, \lambda}$ is given by
\begin{equation}\label{arbitrary-support-equation}
a_{c, \lambda} = \frac{\sum_{\mu \in \irr(W'), \dim L_c(\mu) < \infty} \dim L_c(\mu)a_{c, \mu}\sign(h_\mu)}{\dim \res_b L_c(\lambda)}.
\end{equation}  \end{conjecture}

Recall that $\res_b L_c(\lambda)$ is nonzero and finite-dimensional if and only if $b$ is a generic point in $\supp (L_c(\lambda))$ \cite[Proposition 3.23]{BE}, so the righthand side of equation (\ref{arbitrary-support-equation}) is well-defined.  When $b = 0$, we have that $W' = W$, $h_\mu = 0$ unless $\mu = \lambda$, $h_\lambda$ is the standard form on $\hom_{W'}(\lambda, \lambda) = \cplx$, and $\res_b L_c(\lambda) = L_c(\lambda)$, so Conjecture \ref{arbitrary-support-conjecture} holds trivially.  When $b \in \hfr_{\real, reg}$, Conjecture \ref{arbitrary-support-conjecture} is precisely Theorem \ref{KZ-preserves-signatures-theorem}, as in that case we have $W' = 1$, $L_c(\mu) = L_c(\triv) = \cplx$ (there are no other irreducible representations in the sum), $a_{c, \mu} = 1$, and $\res_b L_c(\lambda) = KZ (L_c(\lambda))$ as vector spaces \cite[Remark 3.16]{BE}.


\begin{thebibliography}{9}

\bibitem[AvLTV]{AvLTV}

  J. Adams, M. can Leeuwen, P. Trapa, D. Vogan.
  \emph{Unitary representations of real reductive groups.}
  Preprint, arXiv:1212.2192v5.

\bibitem[BBK]{BBK}

  C. Bergbaurer, R. Brunetti, D. Kreimer.
  \emph{Renormalization and resolution of singularities.}
  Preprint, arXiv:0908.0633.

\bibitem[BE]{BE}

  R. Bezrukavnikov, P. Etingof.
  \emph{Parabolic induction and restriction functors for rational Cherednik algebras.}
  Selecta Math. (N.S.) 14 (2009), no. 3-4, 397-425.

\bibitem[BMR]{BMR}

  M. Brou\'{e}, G. Malle, R. Rouquier.
  \emph{Complex reflection groups, braid groups, Hecke algebras.}
  J. reine angew. Math. 500 (1998), 127-190.

\bibitem[C]{C}

  I. Cherednik.
  \emph{Affine Hecke algebras via DAHA,}
  talk, www-math.mit.edu/$\sim$etingof/hadaha.pdf

\bibitem[CGG]{CGG}

  M. Chlouverkai, I. Gordon, S. Griffeth.
  \emph{Cell modules and canonical basic sets for Hecke algebras from Cherednik algebras.}
  Contemp. Math. Vol. 562, New Trends in Noncommutative Algebra (2012), 77-89.

\bibitem[DP1]{DP1}

  C. De Concini and D. Procesi.
  \emph{Wonderful models of subspace arrangements.}
  Selecta Math. (N.S.) 1 (1995), 459-494.
  
\bibitem[DP2]{DP2}

  C. De. Concini and D. Procesi.
  \emph{Hyperplane arrangements and holonomy equations.}
  Selecta Math. (N.S.) 1 (1995), 495-535.

\bibitem[DS]{DS}

  Mouez Dimassi and Johannes Sj\"ostrand.
  \emph{Spectral asymptotics in the semi-classical limit.}
  Cambridge University Press, 1999.
  
\bibitem[D1]{Dunkl-dihedral}

  C. Dunkl.
  \emph{Vector polynomials and a matrix weight associated to dihedral groups.}
  SIGMA 10 (2014), 044.

\bibitem[D2]{Dunkl-B2}

  C. Dunkl.
  \emph{Vector-valued polynomials and a matrix weight function with $B_2$-action.}
  SIGMA 9 (2013), 007.
  
\bibitem[D3]{Dunkl-B2-II}
  C. Dunkl.
  \emph{Vector-valued polynomials and a matrix weight function with $B_2$-action, II.}
  SIGMA 9 (2013), 043.
  
\bibitem[D4]{Dunkl-trig-1}
  C. Dunkl.
  \emph{Orthogonality measure on the torus for vector-valued Jack polynomials.}
  SIGMA 12 (2016), 033.
  
\bibitem[D5]{Dunkl-trig-2}
  C. Dunkl.
  \emph{A linear system of differential equations related to vector-valued Jack polynomials on the torus.}
  SIGMA 13 (2017), 040.
  
\bibitem[D6]{Dunkl-trig-3}
  C. Dunkl.
  \emph{Vector-valued Jack polynomials and wavefunctions on the torus.}
  J. Phys. A: Math. Theor. 50 (2017) 245201 (21 pp).

\bibitem[E1]{E1}

  P. Etingof.
  \emph{Calogero-Moser systems and representation theory.}
  Zurich Lectures in Advanced Mathematics, 2007.

\bibitem[E2]{E2}

  P. Etingof.
  \emph{Supports of irreducible spherical representations of rational Cherednik algebras.}
  Adv. Math. 229(3) (2012), 2042-2054.

\bibitem[EG]{EG}

  P. Etingof, V. Ginzburg.
  \emph{Symplectic reflection algebras, Calogero-Moser space, and deformed Harish-Chandra homomorphism.}
  Invent. Math. 147 (2002), 243-348.

\bibitem[EGL]{EGL}

   P. Etingof, E. Gorsky, I. Losev.
   \emph{Representations of Cherednik algebras with minimal support and torus knots.}
   Adv. Math. 227 (2015), 124-180.

\bibitem[EM]{EM}

  P. Etingof, X. Ma.
  \emph{Lecture notes on rational Cherednik algebras.}
  arXiv:1001:0432.

\bibitem[ESG]{ESG}

  P. Etingof, E. Stoica, with an appendix by S. Griffeth.
  \emph{Unitary representations of rational Cherednik algebras.}
  Represent. Theory 13 (2009), 349-370.
  
 \bibitem[F]{F}

   M. Feigin.
   \emph{Generalized Calogero-Moser systems from rational Cherednik algebras.}
   Selecta Math. (N.S.) 18 (2012), N1, 253-281.

\bibitem[GGOR]{GGOR}

  V. Ginzburg, N. Guay, E. Opdam, and R. Rouquier.
  \emph{On the category $\oscr$ for rational Cherednik algebras.}
  Invent. Math. 154 (2003), 617-651.

\bibitem[GS1]{GS1}

   I. Gordon, J. T. Stafford.
   \emph{Rational Cherednik algebras and Hilbert schemes.}
   Adv. Math. 198 (2005), no. 1, 224-274.
  
\bibitem[GS2]{GS2}

   I. Gordon, J. T. Stafford.
   \emph{Rational Cherednik algebras and Hilbert schemes. II. Representations and sheaves.}
   Duke Math. J. 132 (2006), no. 1, 73-135.

\bibitem[GORS]{GORS}

  E. Gorsky, A. Oblomkov, J. Rasmussen, V. Shende.
  \emph{Torus knots and the rational DAHA.}
  Duke Math. J. 163, no. 14 (2014), 2709-2794.

\bibitem[G]{G}

  S. Griffeth.
  \emph{Unitary representations of cyclotomic rational Cherednik algebras.}
  Preprint, arXiv:1106.5084.

\bibitem[H]{H}

  J. E. Humphreys.
  \emph{Reflection groups and Coxeter groups.}
  Cambridge Studies in Advanced Mathematics, 29, Cambridge University Press (1990).

\bibitem[IY]{IY}

  Y. Ilyashenko, S. Yakovenko.
  \emph{Lectures on Analytic Differential Equations.}
  Graduate Studies in Mathematics, vol. 87, Amer. Math. Soc., Providence, RI, 2008.

\bibitem[J]{J}

  J. C. Jantzen.
  \emph{Modulen mit einem h\"{o}chsten Gewich.}
  Lect. Notes in Math. 750, Springer: Berlin-Heidelberg-New York (1979).

\bibitem[LS]{LS}

  I. Losev, S. Shelley-Abrahamson.
  \emph{On refined filtration by supports for rational Cherednik categories $\oscr$.}
  Selecta Math. (N.S.) (2018).

\bibitem[RF]{RF}

  H. L. Royden, P. M. Fitzpatrick.
  \emph{Real Analysis, Fourth Edition.}
  Prentice Hall, Boston (2010).

\bibitem[S]{S}

  P. Shan.
  \emph{Crystals of Fock spaces and cyclotomic double affine Hecke algebras.}
  Ann. Sci. Ecole Norm. Sup. 44 (2011), 147-182.

\bibitem[SV]{SV}

   P. Shan, E. Vasserot.
   \emph{Heisenberg algebras and rational double affine Hecke algebras.}
   J. Amer. Math. Soc. 25 (2012), 959-1031.

\bibitem[ST]{ST}

  G. C. Shephard, J. A. Todd.
  \emph{Finite unitary reflection groups.}
  Can. J. Math. 6 (1954), 274-304.

\bibitem[Ve]{Ve}

  V. Venkateswaran.
  \emph{Signatures of representations of Hecke algebras and rational Cherednik algebras.}
  Preprint, arXiv:1409.6663v3.

\bibitem[Vo]{Vo}
  
  D. Vogan.
  \emph{Unitarizability of certain series of representations.}
  Ann. of Math. 120 (1984), 141-187.

\bibitem[W]{W}

  W. Wasow.
  \emph{Asymptotic Expansions for Ordinary Differential Equations.}
  Krieger, Huntington, N.Y. (1976) (orig. Wiley, 1965).

\bibitem[Y]{Y}

  W. L. Yee.
  \emph{Signatures of invariant Hermitian forms on irreducible highest-weight modules.}
  Duke Math. J. 142 (2008), no. 1, 165-196.

\bibitem[Z]{Z}

  Maciej Zworski.
  \emph{Semiclassical analysis.}
  Graduate Studies in Mathematics, vol. 138, Amer. Math. Soc., Providence, RI, 2012.

\end{thebibliography}
\end{document}